\documentclass[12pt,reqno]{amsart}
\usepackage[headings]{fullpage}
\usepackage{amssymb,amsmath,mathtools,bbm,tikz,tikz-cd,color,relsize,multirow}
\usepackage[all,cmtip]{xy}
\usepackage{url}
\usetikzlibrary{positioning}
\usetikzlibrary{calc}
\usetikzlibrary{decorations.markings}
\usetikzlibrary{arrows}
\usetikzlibrary{calc}
\usetikzlibrary{decorations.markings}
\tikzstyle{hvector}=[inner sep=2pt,draw=blue!50,fill=blue!10,thick]
\tikzstyle{unit}=[inner sep=2pt,shape=circle, draw]
\tikzstyle{counit}=[inner sep=2pt,shape=circle, draw,fill=gray]
\tikzstyle{antipode}=[inner sep=2pt,shape=rectangle, draw]
\tikzstyle{cocycle}=[inner sep=2pt,shape=circle, draw]
\tikzstyle{twistedm}=[inner sep=2pt,shape=circle, fill=gray]
\tikzstyle{autom}=[inner sep=2pt,shape=circle, draw]
\tikzstyle{coact}=[inner sep=2pt,shape=circle, fill=black]

\usepackage[bookmarks=true,%
    colorlinks=true,%
    linkcolor=blue,%
    citecolor=blue,%
    filecolor=blue,%
    menucolor=blue,%
    urlcolor=blue,%
    breaklinks=true]{hyperref}
\usepackage{slashed}    
\usepackage{verbatim}
\usepackage[normalem]{ulem}  
\usepackage{diagbox}

\newtheorem{theorem}{Theorem}[section]
\theoremstyle{definition}
\newtheorem{proposition}[theorem]{Proposition}
\newtheorem{lemma}[theorem]{Lemma}
\newtheorem{definition}[theorem]{Definition}
\newtheorem{remark}[theorem]{Remark}

\newtheorem{conjecture}[theorem]{Conjecture}

\def\BZ{\mathbbm Z}

\def\BR{\mathbbm R}

\def\BF{\mathbbm F}

\def\s{\sigma}

\def\be{\begin{equation}}
\def\ee{\end{equation}}

\def\End{\mathrm{End}}

\newcommand{\bF}{\mathbb{F}}
\newcommand{\Z}{\mathbb{Z}}

\newcommand{\antipode}{S}
\newcommand{\mysym}{\sigma}

\makeatletter

\renewcommand\thepart{\@Roman\c@part}%
\renewcommand\part{%
   \if@noskipsec \leavevmode \fi
   \par
   \addvspace{6.7ex}%
   \@afterindentfalse
   \secdef\@part\@spart}
\def\@part[#1]#2{%
    \ifnum \c@secnumdepth >\m@ne
      \refstepcounter{part}%
      \addcontentsline{toc}{part}{Part~\thepart.\ #1}%
    \else
      \addcontentsline{toc}{part}{#1}%
    \fi
    {\parindent \z@ \raggedright
     \interlinepenalty \@M
     \normalfont
     \ifnum \c@secnumdepth >\m@ne
       \centering\large\scshape \partname~\thepart.%
       \hspace{1ex}%
     \fi%
     \large\scshape #2%
     \markboth{}{}\par}%
    \nobreak
    \vskip 4.7ex
    \@afterheading}
  \def\@spart#1{
  \refstepcounter{part}%
  \addcontentsline{toc}{part}{#1}%
    {\parindent \z@ \raggedright
     \interlinepenalty \@M
     \normalfont
     \centering\large\scshape #1\par}%
     \nobreak
     \vskip 4.7ex
     \@afterheading}
\renewcommand*\l@part[2]{%
  \ifnum \c@tocdepth >-2\relax
    \addpenalty\@secpenalty
    \addvspace{0.75em \@plus\p@}%
    \begingroup
      \parindent \z@ \rightskip \@pnumwidth
      \parfillskip -\@pnumwidth
      {\leavevmode
       \normalsize \bfseries #1\hfil \hb@xt@\@pnumwidth{\hss #2}}\par
       \nobreak
       \if@compatibility
         \global\@nobreaktrue
         \everypar{\global\@nobreakfalse\everypar{}}%
      \fi
    \endgroup
  \fi}

\def\l@subsection{\@tocline{2}{0pt}{2pc}{6pc}{}}
\makeatother

\begin{document}

\title[Multivariable knot polynomials]{
Multivariable knot polynomials from braided Hopf algebras with automorphisms}
\author{Stavros Garoufalidis}
\address{
  International Center for Mathematics, Department of Mathematics \\
  Southern University of Science and Technology \\
  Shenzhen, China \newline
  {\tt \url{http://people.mpim-bonn.mpg.de/stavros}}}
\email{stavros@mpim-bonn.mpg.de}
\author{Rinat Kashaev}
\address{Section de Math\'ematiques, Universit\'e de Gen\`eve \\
2-4 rue du Li\`evre, Case Postale 64, 1211 Gen\`eve 4, Switzerland \newline
         {\tt \url{http://www.unige.ch/math/folks/kashaev}}}
\email{Rinat.Kashaev@unige.ch}
\thanks{
  {\em Key words and phrases:}
  Knots, Jones polynomial, ADO invariant, $R$-matrices, Yang--Baxter equation,
  Hopf algebras, braided Hopf algebras, tensor algebras, braided tensor algebras,
  quantum groups, Yetter--Drinfel'd modules, Nichols algebras, chirality, mutation,
  knot genus. 
}

\date{22 April, 2024}

\begin{abstract}
  We construct knot invariants from solutions to the
  Yang--Baxter equation associated to appropriately generalized left/right
  Yetter--Drinfel'd modules over a braided Hopf algebra with an automorphism.
  When applied to Nichols algebras, our method reproduces known knot polynomials and
  naturally produces  multivariable polynomial invariants of knots. We discuss 
  in detail Nichols algebras of rank $1$ which recover the ADO and the colored Jones
  polynomials of a knot and two sequences of examples of rank $2$ Nichols
  algebras, one of which starts with the product of two Alexander polynomials, and
  then conjecturally the Harper polynomial. The second sequence starts with the
  Links--Gould invariant (conjecturally), and then with a new 2-variable knot
  polynomial that detects chirality and mutation, and whose degree gives sharp
  bounds for the genus for a sample of 30 computed knots. 
\end{abstract}

\maketitle

{\footnotesize
\tableofcontents
}


\section{Introduction}
\label{sec.intro}

Jones's discovery of his famous polynomial of knots had an enormous influence
in knot theory and connected the subject of low dimensional topology and
hyperbolic geometry to mathematical physics, giving rise to quantum
topology~\cite{Jones,Thurston,Witten:CS}.

The Jones polynomial was originally defined by thinking of a knot as the closure
of a braid, and by taking the (suitably normalized) trace of
representations of the braid groups (with an arbitrary number of strands), which
themselves were determined by a vector space $V$ and an automorphism
$R \in \operatorname{Aut}(V \otimes V)$ that satisfies the Yang--Baxter equation
\be
\label{YB}
R_1 R_2 R_1 = R_2 R_1 R_2 \in \End(V \otimes V \otimes V) \,, 
\ee
where $R_1 =R \otimes I, R_2 = I \otimes R$.

It was soon realized that representations of simple Lie algebras and their
deformations, known as quantum groups, were a natural source of solutions to
the Yang--Baxter equations. This led to a plethora of polynomial invariants
of knots; see for example Turaev~\cite{Tu:YB,RT:ribbon}. 

Another source of polynomial invariants (one for every complex root of unity)
came from the work of Akutsu--Deguchi--Ohtsuki~\cite{ADO}. It was conjectured by
Habiro~\cite[Conj.7.4]{Habiro:WRT} and later shown by Willets in~\cite{Willets} that
the collection of the colored Jones polynomials of a knot (colored by the irreducible
representations of $\mathfrak{sl}_2$) determines and is determined by the collection
of ADO invariants at roots of unity. 

The definition of the above invariants requires an $R$-matrix together with a
(ribbon) enhancement of it which, roughly speaking, is an endomorphism of $V$ required
to define the quantum trace, and hence the knot invariant. This comes from the 
Reshetikhin--Turaev functor which forms the basis of knot/link invariants in
arbitrary 3-manifolds~\cite{RT:ribbon}. 

An $R$-matrix alone is in principle sufficient to define knot invariants. This was
clarified by the second author by constructing invariants of knots from an $R$-matrix
that satisfies some non-degeneracy conditions, called rigidity in~\cite{Ka:longknots}.
Rigid $R$-matrices indeed allow one to define state-sum invariants of planar projections
of knots without any extra data. A description of how these invariants are defined
is given in Section~\ref{sec.RT} below.

One can construct rigid $R$-matrices from any Hopf algebra with invertible antipode
through Drinfel'd's quantum double construction which can be put into pure algebraic
setting of multilinear algebra without finiteness assumptions of the
Hopf algebra~\cite{Ka:hopf}. 

In this paper, we propose a different approach of producing rigid $R$-matrices
that does not use the quantum double of a Hopf algebra. The construction of these
$R$-matrices, and the corresponding knot invariants, is schematically summarized
in the following steps:
\be
\label{steps}
\Bigg\{ \begin{matrix}
  \text{Braided} \\ \text{Hopf algebras} \\ \text{with autos}
\end{matrix}\Bigg\}
\rightarrow
\Bigg\{ \begin{matrix} \text{Braided} \\ \text{left/right YD modules} \\
  \text{with automorphisms}
\end{matrix}\Bigg\}
\rightarrow \{\text{$R$-matrices}\}
\rightarrow \{\text{Knot invariants}\}
\ee

The last arrow in~\eqref{steps} is the well-known Reshetikhin--Turaev functor
reviewed in Section~\ref{sec.RT}.
The first and the second arrows are discussed in Sections~\ref{sub.H2DY}
and~\ref{sub.DY2R} below. These sections are written in the maximum level of
abstraction, using the language of category theory, for potential future applications
to braided categories not coming from vector spaces. 

The knot invariants defined by the steps in~\eqref{steps} require as input a braided
Hopf algebra with automorphisms. 
A concrete source of braided Hopf algebras with a rich group of automorphisms are
the Nichols algebras discussed in detail in Section~\ref{sec.tensor}. Roughly speaking,
a Nichols algebra is the quotient of a naturally graded tensor algebra of a braided
vector space by a suitable grading preserving maximal Hopf ideal. Any choice of an
automorphism gives rise to appropriately generalized left and right Yetter--Drinfel'd
module structures over the Nichols algebra. It turns out that the latter
admits a natural quotient space in the case
of the left generalized Yetter--Drinfel'd module and a natural submodule in the case
of the right generalized Yetter--Drinfel'd module. In a sense, the choice
of the braided Hopf algebra automorphism seems to correspond to a choice of a
quantum double representation in the traditional approach.

In Sections~\ref{sec.rank1} and~\ref{sec.rank2} we study these generalized
Yetter--Drinfel'd
modules of Nichols algebras of diagonal type in the cases when the input braided
vector space is of dimension one and two, which gives rise respectively to one and
two-variable polynomial invariants of knots.

We end this introduction with some further comments.

\noindent 1. 
An important feature of our construction is a braided Hopf algebra with an
automorphism. The nontriviality of the automorphism is an essential part for
constructing nontrivial knot invariants. In a sense the group of automorphisms
replaces the representation theory.

\noindent 2. 
Our approach unifies previous constructions of knot invariants (notably the colored
Jones and the ADO polynomials, the Links-Gould and the Harper polynomials) coming from
super/quantum groups, but also leads to a systematic construction of multivariable
polynomial invariants of knots beyond the quantum groups.

\noindent 3. 
A feature of the knot polynomials that we construct is that they depend on 
variables coming both from the braiding and the automorphism of the braided Hopf algebra.
It is likely that some of our polynomial invariants of knots
coincide with conjectured knot invariants that are discussed in the physics
literature; see for instance the work of Gukov et al~\cite{Gukov-Park}.

\noindent 4. 
We expect that some of these invariants come from finite type invariants of
knots~\cite{B-N:Vassiliev}, though we have not investigated this at the moment.

\noindent 5. 
We expect that our polynomial invariants give lower bounds for the Seifert genus
of a knot (as is known for the classical Alexander polynomial, but also in some
other cases already, see~\cite{NVdV2022,Kohli-Tahar}). 

\noindent 6. 
Regarding $q$-holonomic aspects, we expect our invariants to satisfy the analogous
$q$-holonomic properties (that is, linear $q$-difference equations), as those that
come from quantum groups (such as the colored Jones polynomials associated to a
simple Lie algebra and parametrized by weights of irreducible
representations~\cite{GL:qholo}) or those defined at roots of unity (such as the
ADO invariant~\cite{BDGG:ADO}). 

\noindent 7. 
Finally, regarding asymptotic aspects, we expect that our invariants satisfy
versions of the Volume Conjecture, analogous to those of the colored Jones
polynomials~\cite{K97, MM} and the ADO invariants~\cite{Murakami}. 


\section{From $R$-matrices to knot invariants}
\label{sec.RT}

In this section we briefly describe the Reshetikhin--Turaev functor which allows
to construct knot invariants from
$R$-matrices. These invariants are defined by state-sums~\cite{RT:ribbon}, using a
variation of the construction from the second author's paper~\cite{Ka:longknots}.
There are three ingredients involved in this construction, namely suitable
knot diagrams, rigid $R$-matrices, and the corresponding state-sums.

\subsection{Knot diagrams and rigid $R$-matrices}
\label{sub.normalR}

We use a diagrammatic notation which is very important for the construction of  knot
invariants  and has a long and successful history in knot
theory~\cite{RT:ribbon,Tu:book}. Basically, knots are represented by generic planar
projections composed of local pieces which correspond to structural morphisms 
of a braided vector space, while the compatibility conditions ensure invariance
under changes of planar projections. The notation leads naturally to the concept
of a braided monoidal category, not necessarily in an abelian category, that vastly
generalizes the notion of a braided vector space~\cite{Tu:Vi}.

Following~\cite{Ka:longknots}, we now explain concretely the knot diagrams used. 
An (oriented)  {\color{blue}  long knot diagram} $K$ is an oriented knot diagram
in $\BR^2$ with two open ends called ``{\color{blue} in}'' and ``{\color{blue} out}'':
$$
K=
\begin{tikzpicture}[scale=3/2,baseline=-3]
\node[draw] (a) at (0,0) {$K$};
\draw[thick,->] (0,-.5)--(a)--(0,.5);
\node (north) at (0.25,.4){\tiny out};\node (south) at (0.15,-.4){\tiny in};
\end{tikzpicture}\qquad \text{Examples}\colon 
\quad K=\begin{tikzpicture}[baseline=20,xscale=.5, yscale=.6]
\coordinate (a0) at (0,0);
\coordinate (a1) at (1,2);
\coordinate (a2) at (0,2);
\coordinate (a3) at (2,.5);
\coordinate (a4) at (1,.5);
\coordinate (a5) at (2,2.5);
\draw[thick] (a0) to [out=90,in=-90]  (a1);
\draw[thick] (a1) to [out=90,in=90]  (a2);
\draw[thick] (a3) to [out=-90,in=-90]  (a4);
\draw[thick,->] (a4) to [out=90,in=-90]  (a5);
\draw[line width=3, color=white] (a2) to [out=-90,in=90]  (a3);
\draw[thick] (a2) to [out=-90,in=90]  (a3);
\end{tikzpicture}, 
\quad K=
\begin{tikzpicture}[xscale=.25,yscale=.14,baseline=20]
\coordinate (a0) at (0,0);
\coordinate (a1) at (4,7);
\coordinate (a2) at (2,9);
\coordinate (a3) at (0,7);
\coordinate (a4) at (4,4);
\coordinate (a5) at (0,4);
\coordinate (a6) at (2,7);
\coordinate (a7) at (0,11);
\draw[thick] (a1) to [out=90,in=0] (a2);
\draw[thick] (a4) to [out=-90,in=-90] (a5);
\draw[thick,->]  (a6) to [out=90,in=-90] (a7);
\draw[line width=3pt,white] (a0) to [out=90,in=-90] (a1);
\draw[thick] (a0) to [out=90,in=-90] (a1);
\draw[line width=3pt,white] (a2) to [out=180,in=90] (a3);
\draw[thick] (a2) to [out=180,in=90] (a3);
\draw[line width=3pt,white] (a3) to [out=-90,in=90] (a4);
\draw[thick] (a3) to [out=-90,in=90] (a4);
\draw[line width=3pt,white] (a5) to [out=90,in=-90] (a6);
\draw[thick]  (a5) to [out=90,in=-90] (a6);
\end{tikzpicture} \,.
$$
A long knot diagram can be closed to a planar projection of a knot:
$\ 
\begin{tikzpicture}[baseline=-2]
\node[draw] (a) at (0,0) {$K$};
\draw[thick,->] (0,-.5)--(a)--(0,.5);
\end{tikzpicture}\ 
\mapsto
\ 
\begin{tikzpicture}[baseline=-2]
\coordinate (b) at (-.7,0);
\node[draw] (a) at (0,0) {$K$};
\draw[thick,->] (a) to [out=90,in=90] (b);
\draw[thick] (b) to [out=-90,in=-90] (a);
\end{tikzpicture} \,.
$

The vertical direction plays a preferred role for long knot diagrams. 

The  {\color{blue} normalization} $\dot{K}$ of $K$ is the diagram obtained from $K$
by the replacements of local extrema oriented from left to right
\be
\label{normalization}
\ 
\begin{tikzpicture}[baseline=5, xscale=0.4, yscale=0.45]
\draw[thick] (0,0) to [out=90,in=180] (1,1);
\draw[thick,->] (1,1) to [out=0,in=90] (2,0);
\end{tikzpicture}
\ \mapsto
\begin{tikzpicture}[baseline=5,xscale=.5,yscale=0.15]
\coordinate (a0) at (0,0);
\coordinate (a1) at (1,3);
\coordinate (a2) at (2,0);
\draw[thick] (a1) to [out=180,in=135]  (a2);
\draw[line width=3, color=white] (a0) to [out=45,in=0]  (a1);
\draw[thick,->] (a0) to [out=45,in=0]  (a1);
\end{tikzpicture}
\quad\text{ and } \quad
\begin{tikzpicture}[baseline=-7, xscale=0.4, yscale=0.45]
\draw[thick] (0,0) to [out=-90,in=180] (1,-1);
\draw[thick,->] (1,-1) to [out=0,in=-90] (2,0);
\end{tikzpicture}\ 
\mapsto
\begin{tikzpicture}[baseline=-7,xscale=.5,yscale=0.15]
\coordinate (a0) at (0,0);
\coordinate (a1) at (1,-3);
\coordinate (a2) at (2,0);
\draw[thick] (a1) to [out=180,in=-135]  (a2);
\draw[line width=3, color=white] (a0) to [out=-45,in=0]  (a1);
\draw[thick,->] (a0) to [out=-45,in=0]  (a1);
\end{tikzpicture}
\ee
at all posible locations of $K$. We say that $K$ is normal if $K=\dot{K}$.

The {\color{blue} building blocks} of normal diagrams are given by four types
of segments 
\be
\label{4segments}
\begin{tikzpicture}[yscale=.4,baseline]
\draw[thick,->] (0,0) to [out=90,in=-90] (0,1);
\end{tikzpicture}\ ,\ 
\begin{tikzpicture}[yscale=.4,baseline]
\draw[thick,<-] (0,0) to [out=90,in=-90] (0,1);
\end{tikzpicture}\ ,\ \quad
\begin{tikzpicture}[xscale=.7,yscale=1,baseline]
\draw[thick,<-] (0,0) to [out=90,in=90] (1,0);
\end{tikzpicture}\ ,\ 
\begin{tikzpicture}[xscale=.7,yscale=1,baseline=20]
\draw[thick,<-] (0,1) to [out=-90,in=-90] (1,1);
\end{tikzpicture}
\ee
and eight types of crossings (four positive and four negative ones)
\be
\label{8crossings}
\begin{tikzpicture}[scale=.7,baseline]
\draw[thick,<-] (0,1) to [out=-90,in=90] (1,0);
\draw[line width=3pt,white] (1,1) to [out=-90,in=90] (0,0);
\draw[thick,<-] (1,1) to [out=-90,in=90] (0,0);
\end{tikzpicture}\ ,\ 
\begin{tikzpicture}[scale=.7,baseline]
\draw[thick,<-] (1,1) to [out=-90,in=90] (0,0);
\draw[line width=3pt,white] (0,1) to [out=-90,in=90] (1,0);
\draw[thick,<-] (0,1) to [out=-90,in=90] (1,0);
\end{tikzpicture}\ ,\
\quad
\begin{tikzpicture}[scale=.7,baseline]
\draw[thick,->] (1,1) to [out=-90,in=90] (0,0);
\draw[line width=3pt,white] (0,1) to [out=-90,in=90] (1,0);
\draw[thick,<-] (0,1) to [out=-90,in=90] (1,0);
\end{tikzpicture}\ ,\ 
\begin{tikzpicture}[scale=.7,baseline]
\draw[thick,<-] (0,1) to [out=-90,in=90] (1,0);
\draw[line width=3pt,white] (1,1) to [out=-90,in=90] (0,0);
\draw[thick,->] (1,1) to [out=-90,in=90] (0,0);
\end{tikzpicture}\ ,\
\quad
\begin{tikzpicture}[scale=.7,baseline]
\draw[thick,->] (0,1) to [out=-90,in=90] (1,0);
\draw[line width=3pt,white] (1,1) to [out=-90,in=90] (0,0);
\draw[thick,->] (1,1) to [out=-90,in=90] (0,0);
\end{tikzpicture}\ ,\ 
\begin{tikzpicture}[scale=.7,baseline]
\draw[thick,->] (1,1) to [out=-90,in=90] (0,0);
\draw[line width=3pt,white] (0,1) to [out=-90,in=90] (1,0);
\draw[thick,->] (0,1) to [out=-90,in=90] (1,0);
\end{tikzpicture}\ ,\
\quad
\begin{tikzpicture}[scale=.7,baseline]
\draw[thick,<-] (1,1) to [out=-90,in=90] (0,0);
\draw[line width=3pt,white] (0,1) to [out=-90,in=90] (1,0);
\draw[thick,->] (0,1) to [out=-90,in=90] (1,0);
\end{tikzpicture}\ ,\ 
\begin{tikzpicture}[scale=.7,baseline]
\draw[thick,->] (0,1) to [out=-90,in=90] (1,0);
\draw[line width=3pt,white] (1,1) to [out=-90,in=90] (0,0);
\draw[thick,<-] (1,1) to [out=-90,in=90] (0,0);
\end{tikzpicture} \,.
\ee

We next define $R$-matrices and their rigid version. 
An  {\color{blue} $R$-matrix} over a vector space $V$ is an automorphism
$r\in \operatorname{Aut}( V\otimes V)$ of $V \otimes V$ that satisfies the
quantum {\color{blue} Yang--Baxter relation} 
\be
\label{rYB}
r_1r_2r_1=r_2r_1r_2,\qquad
r_1:=r\otimes \operatorname{id}_{V},\,\, r_2:=\operatorname{id}_{V}\otimes r \,.
\ee
Let $V^*$ denote the dual vector space and
$\langle \cdot, \cdot \rangle: V^* \otimes V \to \bF$ denote the natural evaluation
map. Assume that $V$ is a finite-dimensional, and fix a basis 
$B$ of $V$ and the corresponding dual basis $\{b^*\}_{b\in B}$ of $V^*$.

Given $f\in \operatorname{End}( V\otimes V)$, we define its
{\color{blue}partial transpose} $\tilde{f}\colon V^*\otimes V\to V\otimes V^*$ by
\be
\label{ftrans}
\tilde f (a^*\otimes b)=
\sum_{c,d\in B}\langle a^*\otimes c^*,f(b\otimes d)\rangle c\otimes d^*,\quad
\begin{tikzpicture}[baseline=10, xscale=1, yscale=1]
\coordinate (nw) at (0,1);
\coordinate (ne) at (1,1);
\coordinate (sw) at (0,0);
\coordinate (se) at (1,0);
\node[rectangle,inner sep=2pt,draw] (c) at (1/2,1/2) {\tiny$f$};
\draw[thick] (c.north west) to [out=135, in= -90] (nw);
\draw[thick] (c.north east) to [out=45, in= -90] (ne);
\draw[thick] (c.south east) to [out=-45, in= 90] (se);
\draw[thick] (c.south west) to [out=-135, in= 90] (sw);  
\end{tikzpicture}
\mapsto
\begin{tikzpicture}[baseline=10, xscale=1, yscale=1]
\coordinate (nw) at (0,1);
\coordinate (ne) at (1,1);
\coordinate (sw) at (0,0);
\coordinate (se) at (1,0);
\node[rectangle,inner sep=2pt,draw] (c) at (1/2,1/2) {\tiny$\tilde f$};
\draw[thick] (c.north west) to [out=135, in= -90] (nw);
\draw[blue, thick] (c.north east) to [out=45, in= -90] (ne);
\draw[thick] (c.south east) to [out=-45, in= 90] (se);
\draw[blue, thick] (c.south west) to [out=-135, in= 90] (sw);  
\end{tikzpicture}
=
\begin{tikzpicture}[baseline=10, xscale=1, yscale=1]
\coordinate (fsw) at (-1,0);
\coordinate (fne) at (2,1);
\coordinate (nw) at (0,8/9);
\coordinate (ne) at (1,1);
\coordinate (sw) at (0,0);
\coordinate (se) at (1,1/9);
\node[rectangle,inner sep=2pt,draw] (c) at (1/2,1/2) {\tiny$f$};
\draw[thick] (c.north west) to [out=135, in= 0] (nw);
\draw[blue,thick] (nw) to [out=180, in= 90] (fsw);
\draw[thick] (c.north east) to [out=45, in= -90] (ne);
\draw[thick] (c.south east) to [out=-45, in= 180] (se);
\draw[blue,thick] (se) to [out=0, in= -90] (fne);
\draw[thick] (c.south west) to [out=-135, in= 90] (sw);  
\end{tikzpicture} \,.
\ee

We call an $R$-matrix $r$ {\color{blue} rigid} if $\widetilde{r^{\pm1}}$
are invertible.

\subsection{State-sum invariants of knots}
\label{sub.RTsums}

We now have all the ingredients to define the state-sum invariants of normal
knot diagrams. Fix a rigid $R$-matrix $r$ over a finite dimensional vector
space $V$, equipped with a basis $B$. For a normal long knot diagram $K$, 
let $E_K$ and $C_K$ denote its sets of edges and crossings, respectively.

A {\color{blue} state} $s$ of $K$ is a map $s\colon E_K\to B$ that assigns an
element of $B$ to each edge of $K$. The {\color{blue} weight} $w_s(K)$ of
the state $s$ of $K$ is the product of local weights
\be
\label{wsK}
w_s(K)=\prod_{c\in C_K}w_s(c) \,,
\ee
where the local weights are defined by
\be
\label{wsKlocal}
\begin{tikzpicture}[yscale=.7,xscale=.4,baseline=7]
\draw[thick,<-] (0,1) to [out=-90,in=90] (1,0);
\draw[line width=3pt,white] (1,1) to [out=-90,in=90] (0,0);
\draw[thick,<-] (1,1) to [out=-90,in=90] (0,0);
\node (sw) at (0,-.2){\tiny $a$};\node (se) at (1,-.2){\tiny $b$};
\node (nw) at (0,1.2){\tiny $c$};\node (ne) at (1,1.2){\tiny $d$};
\end{tikzpicture},
\begin{tikzpicture}[yscale=.7,xscale=.4,baseline=7]
\draw[thick,->] (1,1) to [out=-90,in=90] (0,0);
\draw[line width=3pt,white] (0,1) to [out=-90,in=90] (1,0);
\draw[thick,<-] (0,1) to [out=-90,in=90] (1,0);
\node (sw) at (0,-.2){\tiny $c$};\node (se) at (1,-.2){\tiny $a$};
\node (nw) at (0,1.2){\tiny $d$};\node (ne) at (1,1.2){\tiny $b$};
\end{tikzpicture},
\begin{tikzpicture}[yscale=.7,xscale=.4,baseline=7]
\draw[thick,->] (0,1) to [out=-90,in=90] (1,0);
\draw[line width=3pt,white] (1,1) to [out=-90,in=90] (0,0);
\draw[thick,->] (1,1) to [out=-90,in=90] (0,0);
\node (sw) at (0,-.2){\tiny $d$};\node (se) at (1,-.2){\tiny $c$};
\node (nw) at (0,1.2){\tiny $b$};\node (ne) at (1,1.2){\tiny $a$};
\end{tikzpicture}
\xmapsto{w_s}
\langle c^*\otimes d^*,r(a\otimes b)\rangle,
\qquad
\begin{tikzpicture}[yscale=.7,xscale=.4,baseline=7]
\draw[thick,<-] (1,1) to [out=-90,in=90] (0,0);
\draw[line width=3pt,white] (0,1) to [out=-90,in=90] (1,0);
\draw[thick,->] (0,1) to [out=-90,in=90] (1,0);
\node (sw) at (0,-.2){\tiny $b$};\node (se) at (1,-.2){\tiny $d$};
\node (nw) at (0,1.2){\tiny $a$};\node (ne) at (1,1.2){\tiny $c$};
\end{tikzpicture}\
\xmapsto{w_s}
\big\langle a\otimes c^*,(\widetilde{r^{-1}})^{-1}(b\otimes d^*)\big\rangle
\ee
for positive crossings and likewise for negative crossings
\be
\label{wsKlocalminus}
\begin{tikzpicture}[yscale=.7,xscale=.4,baseline=7]
\draw[thick,<-] (1,1) to [out=-90,in=90] (0,0);
\draw[line width=3pt,white] (0,1) to [out=-90,in=90] (1,0);
\draw[thick,<-] (0,1) to [out=-90,in=90] (1,0);
\node (sw) at (0,-.2){\tiny $a$};\node (se) at (1,-.2){\tiny $b$};
\node (nw) at (0,1.2){\tiny $c$};\node (ne) at (1,1.2){\tiny $d$};
\end{tikzpicture},
\begin{tikzpicture}[yscale=.7,xscale=.4,baseline=7]
\draw[thick,<-] (0,1) to [out=-90,in=90] (1,0);
\draw[line width=3pt,white] (1,1) to [out=-90,in=90] (0,0);
\draw[thick,->] (1,1) to [out=-90,in=90] (0,0);
\node (sw) at (0,-.2){\tiny $c$};\node (se) at (1,-.2){\tiny $a$};
\node (nw) at (0,1.2){\tiny $d$};\node (ne) at (1,1.2){\tiny $b$};
\end{tikzpicture},
\begin{tikzpicture}[yscale=.7,xscale=.4,baseline=7]
\draw[thick,->] (1,1) to [out=-90,in=90] (0,0);
\draw[line width=3pt,white] (0,1) to [out=-90,in=90] (1,0);
\draw[thick,->] (0,1) to [out=-90,in=90] (1,0);
\node (sw) at (0,-.2){\tiny $d$};\node (se) at (1,-.2){\tiny $c$};
\node (nw) at (0,1.2){\tiny $b$};\node (ne) at (1,1.2){\tiny $a$};
\end{tikzpicture}
\xmapsto{w_s}
\langle c^*\otimes d^*,r^{-1}(a\otimes b)\rangle,
\qquad
\begin{tikzpicture}[yscale=.7,xscale=.4,baseline=7]
\draw[thick,->] (0,1) to [out=-90,in=90] (1,0);
\draw[line width=3pt,white] (1,1) to [out=-90,in=90] (0,0);
\draw[thick,<-] (1,1) to [out=-90,in=90] (0,0);
\node (sw) at (0,-.2){\tiny $b$};\node (se) at (1,-.2){\tiny $d$};
\node (nw) at (0,1.2){\tiny $a$};\node (ne) at (1,1.2){\tiny $c$};
\end{tikzpicture}\
\xmapsto{w_s}
\big\langle a\otimes c^*,(\widetilde{r})^{-1}(b\otimes d^*)\big\rangle \,.
\ee

These arrangements of the $R$-matrices at the crossings are the same as
in~\cite[Eqns.(16)-(19)]{Ka:longknots}.

The main theorem of this construction is the topological invariance of the
state-sum; see~\cite{Reshetikhin:quasi,RT:ribbon,Tu:book} and in the form stated
below,~\cite[Thm.1]{Ka:longknots}.

\begin{theorem}
Let a normal long knot diagram $K$ have an equal number of negative and positive
crossings. Then, the linear map
\be
\label{JKdef}
J_r(K)\colon V\to V,\quad  J_r(K)a
=\sum_{s\in B^{E_K},\ s_{\text{in}}=a}w_s(K) s_{\text{out}}
\ee
is independent of the basis of $V$ and it is thus an $End(V)$-valued
invariant of oriented knots.
\end{theorem}

When the vector space $V$ is equipped with a basis $B$ (as in all of our examples
below), then the oriented knot invariant above is a matrix-valued invariant.

Note that this construction can be extended to the context of arbitrary
monoidal categories with duality.

The above invariant should not be confused with the universal invariants of knots
taking values in quotients of completed Hopf algebras that come from quantum groups,
considered by Lawrence~\cite{Lawrence}, Ohtsuki~\cite{Ohtsuki:colored} and
Habiro~\cite{Habiro:WRT}.
  

\section{From braided Hopf algebras with automorphisms to
  $R$-matrices}
\label{sec.hopf2R}

\subsection{From Hopf f-objects to left/right Yetter--Drinfel'd
f-objects}
\label{sub.H2DY}

In this section we discuss the left arrow of~\eqref{steps}.
We deliberately phrase our results in the language of braided monoidal
(not-necessarily abelian) categories to allow versatility of future applications.
A detailed discussion of the concepts of a monoidal category, braided monoidal
category, rigid monoidal category, category of functors, algebra and coalgebra
objects in monoidal categories, modules over algebra objects and comodules over
coalgebra objects and their morphisms can be found in the book by
Turaev--Virelizier~\cite[Sec.1.6]{Tu:Vi}.

All monoidal categories that we consider are assumed to be strict.
In writing compositions of morphisms $g: X \to Y$ and $f: Y \to Z$ in a
category, we will suppress the composition symbol, so that we write $fg$ instead of
$f\circ g$. Moreover, in the case of monoidal
categories, we assume the preference of the composition over the monoidal product,
so that, for example, $fg\otimes h$ will mean $(fg)\otimes h$.

When a functor $F\colon \mathcal{D}\to\mathcal{C}$ is considered as an object of the
functorial category $\mathcal{C}^{\mathcal{D}}$, it will be called
\emph{\color{blue}functorial object} or just  \emph{\color{blue}f-object} for brevity.

Let $\mathcal{C}$ be a braided monoidal category. Denote by $\mathcal{C}^{\Z}$ the
braided monoidal category of functors $F\colon \Z\to \mathcal{C}$ where the additive
group of integers $\Z$ is viewed as a category with one object $*$ whose automorphism
group is $\BZ$. We denote by
$\tau\colon \otimes\to \otimes^{\text{op}}$ the braiding of $\mathcal{C}^{\Z}$ which
assigns to any pair of f-objects $F$ and $G$ a functorial morphism
$\tau_{F,G}\colon F\otimes G\to G\otimes F$ that at the unique object $*$ of $\Z$
evaluates  to the morphism  of $\mathcal{C}$
$$(\tau_{\mathcal{C}})_{F(*),G(*)}\colon
F(*)\otimes G(*)\to G(*)\otimes F(*)$$ where $\tau_{\mathcal{C}}$
is the braiding in $\mathcal{C}$.

\begin{remark}
Given the fact that the group $\Z$ is freely generated by one element $1$, an
object~$G$ of the functor category $\mathcal{C}^{\Z}$ is uniquely determined by the
pair $(A,\phi)$ where $A$ is the object of $\mathcal{C}$ obtained as the image by $G$
of the unique object $*$ of $\Z$, and $\phi \colon A\to A$ is the automorphism of $A$
obtained as the image by $G$ of the generating element $1$ of $\Z$. With this
interpretation, a morphism from $(A,\phi)$ to $(B,\psi)$ is a morphism
$f\colon A\to B$ in $\mathcal{C}$ such that $\psi f=f\phi$. The monoidal product of
two pairs $(A,\phi)\otimes (B,\psi)$ is given by the pair
$(A\otimes B, \phi\otimes\psi)$.
\end{remark}

\begin{definition}
A  \emph{\color{blue} Hopf f-object} is an f-object $H\colon\Z\to\mathcal{C}$
together with functorial morphisms (natural transformations)
\begin{equation}
\label{eq:hopff}  
  \nabla\colon H\otimes H\to H,\quad \eta\colon\mathbb{I}\to H,\quad
  \Delta\colon H\to H\otimes H,\quad \epsilon\colon H\to \mathbb{I},
  \quad S\colon H\to H
\end{equation}
such that
$(H,\nabla,\eta)$ is an algebra f-object,
$(H,\Delta,\epsilon)$ is a coalgebra f-object and
\begin{equation}
(\nabla\otimes\nabla)
(\operatorname{id}_H\otimes \tau_{H,H}\otimes\operatorname{id}_H)(\Delta\otimes\Delta)
=\Delta\nabla,
\end{equation}
\begin{equation}
\nabla(S\otimes\operatorname{id}_H)\Delta=\eta\epsilon
=\nabla(\operatorname{id}_H\otimes S)\Delta.
\end{equation}
We will always assume that $S$ is an invertible (functorial) morphism. As in the
theory of Hopf algebras, the functorial morphisms $\nabla$, $\eta$, $\Delta$,
$\epsilon$ and $S$ are respectively called product, unit, coproduct, counit and
antipode.
\end{definition}

\begin{definition}
\label{def:yet-drin}
Let  $H\colon \Z\to\mathcal{C}$ be a Hopf f-object. A \emph{\color{blue} left
  Yetter--Drinfel'd f-object} over $H$ is a triple $(Y,\lambda_L,\delta_L)$ where
$Y\colon \Z\to\mathcal{C}$ is an f-object of $\mathcal{C}^\Z$, and
$\lambda_L\colon H\otimes Y\to Y$, $\delta_L\colon Y\to H\otimes Y$ are morphisms of
$\mathcal{C}^\Z$ such that $(Y,\lambda_L)$ is a left $H$-module f-object,
$(Y,\delta_L)$ is a left $H$-comodule f-object, and 
\be
\label{eq:DYL}
\begin{aligned}
(\nabla\otimes\operatorname{id}_Y)(\operatorname{id}_H\otimes\tau_{Y,H})
(\delta_L\lambda_L\otimes \phi_H)(\operatorname{id}_H\otimes\tau_{H,Y})
(\Delta\otimes\operatorname{id}_Y) & \\
& \hspace{-5cm}
=(\nabla\otimes\lambda_L)
(\operatorname{id}_H\otimes \tau_{H,H}\otimes\operatorname{id}_Y)(\Delta\otimes\delta_L)
\end{aligned}
\ee
where $\phi_H\colon H\to H$ is the functorial isomorphism that at the unique object
$*$ of $\Z$ evaluates as
$$
(\phi_H)_*=H(1)\colon H(*)\to H(*).
$$
\end{definition}

Taking into account the self-dual nature of Hopf objects, it is useful to have
the dual version of Definition~\ref{def:yet-drin} which reads as follows.

\begin{definition}
A \emph{\color{blue} right Yetter--Drinfel'd f-object} over a Hopf f-object
$H\colon \Z\to\mathcal{C}$ is a triple $(Y,\lambda_R,\delta_R)$ where
$Y\colon \Z\to\mathcal{C}$ is an f-object of $\mathcal{C}^\Z$, and
$\lambda_R\colon Y\otimes H\to Y$, $\delta_R\colon Y\to Y\otimes H$ are functorial
morphisms of $\mathcal{C}^\Z$ such that $(Y,\lambda_R)$ is a right $H$-module f-object,
$(Y,\delta_R)$ is a right $H$-comodule f-object, and 
\be
\label{eq:DYR}
\begin{aligned}
(\operatorname{id}_Y\otimes\nabla)(\tau_{H,Y}\otimes\operatorname{id}_H)
(\phi_H\otimes \delta_R\lambda_R)(\tau_{Y,H}\otimes\operatorname{id}_H)
(\operatorname{id}_Y\otimes\Delta) & \\
& \hspace{-5cm}
=(\lambda_R\otimes\nabla)
(\operatorname{id}_Y\otimes \tau_{H,H}\otimes\operatorname{id}_H)
(\delta_R\otimes\Delta) \,.
\end{aligned}
\ee
\end{definition}
We will return and give further clarifications to these definitions later in
Subsection~\ref{sub.DYdiag} after introducing the graphical notation of string
diagrams.

For a Hopf f-object $H\colon\Z\to\mathcal{C}$, we denote by $\Delta^{(2)}$
and $\nabla^{(2)}$ the twice iterated coproduct and product, respectively, defined
by

\be
\label{DD2}
\begin{aligned}
\nabla^{(2)} : H \otimes H \otimes H \to H, &
\qquad\qquad \nabla^{(2)}=\nabla (\nabla \otimes \operatorname{id}_H)
\\
\Delta^{(2)} : H \to H \otimes H \otimes H, &
\qquad\qquad \Delta^{(2)}= (\Delta \otimes \operatorname{id}_H) \Delta \,.
\end{aligned}
\ee

The following theorem provides constructions of  left/right Yetter--Drinfel'd
f-objects over a Hopf f-object $H\colon\Z\to\mathcal{C}$.

\begin{theorem}
\label{thm:1}
For any Hopf f-object $H\colon\Z\to\mathcal{C}$, 
\newline
\rm{(a)} the triple $(H,\nabla,\delta_L)$ is a left Yetter--Drinfel'd f-object
over $H$, where
\begin{equation}
\label{eq:coaction}
\delta_L :=(\nabla\otimes\operatorname{id}_H)(\operatorname{id}_H\otimes \tau_{H,H})
(\operatorname{id}_{H\otimes H}\otimes S\phi_H)\Delta^{(2)};
\end{equation}
\newline
\rm{(b)} the triple $(H,\lambda_R,\Delta)$ is a right Yetter--Drinfel'd f-object
over $H$, where
\begin{equation}
\label{eq:action}
\lambda_R :=\nabla^{(2)}(S\phi_H\otimes\operatorname{id}_{H\otimes H} )
(\tau_{H,H}\otimes \operatorname{id}_H)(\operatorname{id}_H\otimes\Delta) \,.
\end{equation}
\end{theorem}

\subsection{From left/right Yetter--Drinfel'd f-objects to $R$-matrices}
\label{sub.DY2R}

The next theorem constructs $R$-matrices from left/right Yetter--Drinfel'd
f-objects corresponding to the second arrow in~\eqref{steps}. 

\begin{theorem}
\label{thm:2}
Let  $H\colon \Z\to\mathcal{C}$ be a Hopf f-object and $(Y,\lambda,\delta)$ be
a left, respectively a right, Yetter--Drinfel'd f-object over $H$. Then
\begin{equation}
\label{eq:r-matrix}
\rho_L=(\lambda\otimes\operatorname{id}_Y)(\operatorname{id}_H\otimes \tau_{Y,Y})
(\delta\otimes \phi_Y),
\end{equation}
respectively 
\begin{equation}
\label{eq:r-matrix-r}
\rho_R=(\phi_Y\otimes \lambda)(\tau_{Y,Y}\otimes \operatorname{id}_H)
(\operatorname{id}_Y\otimes\delta),
\end{equation}
is an $R$-matrix, that is a solution of  the following braid group type Yang--Baxter
relation in the automorphism group $\operatorname{Aut}(Y\otimes Y\otimes Y)$:
\begin{equation}
\label{tauYB}  
(\rho\otimes \operatorname{id}_Y)(\operatorname{id}_Y\otimes\rho)
(\rho\otimes \operatorname{id}_Y)=(\operatorname{id}_Y\otimes\rho)
(\rho\otimes\operatorname{id}_Y)(\operatorname{id}_Y\otimes\rho).
\end{equation}
Moreover, this $R$-matrix is rigid if the f-object $Y$ is rigid.
\end{theorem}

The proof of these theorems is given in the next section, using a diagrammatic
calculus.

A corollary of Theorem~\eqref{thm:2}  gives an invariant of knots.

\begin{theorem}
\label{thm:3}
Fix a rigid  left or a right Yetter--Drinfel'd f-object $Y$ over a
Hopf f-object $H$ with corresponding $R$-matrix $\rho$. Then, there exists a knot
invariant
\be
\label{WYKdef}
\{ \text{Knots in $S^3$} \} \to \mathrm{End}(Y), \qquad
K \mapsto J_\rho(K) \,. 
\ee
\end{theorem}


\section{Proofs}
\label{sec.proofs}

\subsection{Diagrammatics of braided Hopf algebras with automorphisms}
\label{sub.hopfdiag}

The Hopf f-objects introduced in Section~\ref{sec.hopf2R} are categorical versions
of pairs $(H,\phi)$ where $H$ is a braided Hopf algebra and $\phi$ is an automorphism
of $H$. At around the same time of the Reshetikhin--Turaev construction of knot
invariants via diagrammatics, there was a parallel intense activity in the theory
of Hopf algebras motivated in part by the theory of quantum groups developed by
Drinfel'd~\cite{Drinfeld} and Jimbo~\cite{Jimbo}. There is a string diagrammatic
calculus designed to prove tensor identities in Hopf algebras that avoids using
explicit coordinate formulas for the tensors involved.

This string diagrammatic calculus extends to the case of braided Hopf algebras,
introduced by Majid around 1990~\cite{Majid:braided, Majid}, and used extensively by
many authors including Radford, Kuperberg and
Kauffman~\cite{Radford, Kuperberg, KR:hopf}. A survey of the various
directions of braided Hopf algebras around 2000 is given by Takeuchi in~\cite{Takeuchi}.

The string diagrammatics of the generators and relations of a Hopf algebra are
given in~\cite{Majid}. For a recent treatement, see~\cite{Ka:hopf}, namely,
Eqns.~(1.81)-(1.85) for the generators, Eqns.~(1.86)-(1.91) for the relations and
Eqns.~(1.68)-(1.73) for the diagrammatic notation. For the convenience of the reader,
we recall below the definitions of these morphisms, relations, and the string
diagrammatic notation. 

Let $\mathcal{C}$ be a category. To any morphism $f\colon X\to Y$ in $\mathcal{C}$,
we associate a graphical picture
\begin{equation}
f\ =:\ \begin{tikzpicture}[baseline=-2pt]
\node (x) [hvector] { $f$};
\node (j)  [] at (0,-23pt){$X$};
\node (k) [] at (0,23pt){ $Y$};
\draw[thick] (x)-- (j)(x)--(k);
\end{tikzpicture}.
\end{equation}
If $f\colon X\to Y$ and $g\colon Y\to Z$ are two composable morphisms, then their
composition is described by the vertical concatenation of graphs
\begin{equation}
g\circ f\ =\ 
\begin{tikzpicture}[baseline=-2pt,yscale=1.35]
\node (x) [hvector] {$g\circ f$};
\node (j) [] at (0,28pt){ $Z$};
\node (k) [] at (0,-29pt){ $X$};
\draw[thick] (x)--(j)(x)--(k);
\end{tikzpicture}
\ =\ 
\begin{tikzpicture}[baseline=-18pt]
\node (x) [hvector] {$g$};
\node (j) [] at (0,22pt){ $Z$};
\node (y) [hvector] at (0,-30pt){$f$};
\node (k) [] at (0,-55pt){$X$};
\draw[thick]  (x) to node[auto,inner sep=1pt]{} (j);
\draw[thick]  (x)-- (y);
\draw[thick]  (y) to node[auto,inner sep=1pt]{} (k);
\node[thick]  at (0.3, -15pt) {$Y$};
\end{tikzpicture}
\end{equation}
In particular, for the identity morphism $\operatorname{id}_X$ it is natural to
use just a line
\begin{equation}
\operatorname{id}_X\ =\ 
\begin{tikzpicture}[baseline=-2pt]
\node (x) [hvector] { $\operatorname{id}_X$};
\node (j) [] at (0,-30pt){ $X$};
\node (k) [] at (0,30pt){ $X$};
\draw[thick]  (x) to node[auto,inner sep=1pt]{} (j);
\draw[thick]  (x) to node[auto,inner sep=1pt]{} (k);
\end{tikzpicture}
\ =\ 
\begin{tikzpicture}[baseline=-2pt]
\node (x) at (0,30pt){ $X$};
\node (j) at (0,-30pt){$X$};
\draw[thick]  (x) -- (j);
\end{tikzpicture}.
\end{equation}

The string diagrams are especially useful in the case when $\mathcal{C}$ is a strict
monoidal category, because the tensor (monoidal) product can be drawn by the
horizontal juxtaposition. Namely, for two morphisms $f\colon X\to Y$ and
$g\colon U\to V$, their tensor  product $f\otimes g\colon X\otimes U\to Y\otimes V$
is drawn as follows:
\begin{equation}
f\otimes g\ =\ 
\begin{tikzpicture}[baseline=-2pt]
\node (x) [hvector] {  $f\otimes g$};
\node (j) [] at (0,23pt){  $Y\otimes V$};
\node (k) [] at (0,-23pt){  $X\otimes U$};
\draw[thick]  (x) -- (j) (x) -- (k);
\end{tikzpicture}
\ =\ 
\begin{tikzpicture}[baseline=-2pt]
\node (x) [hvector] {  $f\otimes g$};
\path  (x.east)+(0,23pt) node (j) {  $V$};
\path  (x.west)+(0,23pt) node (j1) {  $Y$};
\path (x.east) +(0,-23pt) node (k) {  $U$};
\path (x.west) +(0,-23pt) node  (k1) {  $X$};
\draw[thick]  (x.north east) --(j);
\draw[thick]  (x.south east) -- (k);
\draw[thick]  (x.north west) -- (j1);
\draw[thick]  (x.south west) --(k1);
\end{tikzpicture}
\ =\ 
\begin{tikzpicture}[baseline=-2pt]
\node (x) [hvector] {  $g$};
\path (x)+(0,23pt) node (j) {  $V$};
\path (x)+(0,-23pt) node (k) {  $U$};
\path (x)+(-16pt,0) node (y) [hvector]{  $f$};
\path (y)+(0,23pt) node (j1) {  $Y$};
\path (y)+(0,-23pt) node (k1) {  $X$};
\draw[thick]  (x)-- (j);
\draw[thick]  (x) -- (k);
\draw[thick]  (y) --(j1);
\draw[thick]  (y)--(k1);
\end{tikzpicture}.
\end{equation}
For example, the graphical equalities
\begin{equation}
\begin{tikzpicture}[baseline=-5pt,yscale=1.45]
\node (x) [hvector] {  $g$};
\path (x)+(0,23pt) node (j) {  $V$};
\path (x)+(0,-23pt) node (k) {  $U$};
\path (x)+(-16pt,0) node (y) [hvector]{  $f$};
\path (y)+(0,23pt) node (j1) {  $Y$};
\path (y)+(0,-23pt) node (k1) {  $X$};
\draw[thick]  (x)-- (j);
\draw[thick]  (x) -- (k);
\draw[thick]  (y) --(j1);
\draw[thick]  (y)--(k1);
\end{tikzpicture}
=
\begin{tikzpicture}[baseline=5]
\node (x) [hvector] {  $g$};
\path (x)+(0,43pt) node (j) {  $V$};
\path (x)+(0,-23pt) node (k) {  $U$};
\path (x)+(-16pt,20pt) node (y) [hvector]{  $f$};
\path (y)+(0,23pt) node (j1) {  $Y$};
\path (y)+(0,-43pt) node (k1) {  $X$};
\draw[thick]  (x)-- (j);
\draw[thick]  (x) -- (k);
\draw[thick]  (y) --(j1);
\draw[thick]  (y)--(k1);
\end{tikzpicture}
=
\begin{tikzpicture}[baseline=-15]
\node (x) [hvector] {  $g$};
\path (x)+(0,23pt) node (j) {  $V$};
\path (x)+(0,-43pt) node (k) {  $U$};
\path (x)+(-16pt,-20pt) node (y) [hvector]{  $f$};
\path (y)+(0,43pt) node (j1) {  $Y$};
\path (y)+(0,-23pt) node (k1) {  $X$};
\draw[thick]  (x)-- (j);
\draw[thick]  (x) -- (k);
\draw[thick]  (y) --(j1);
\draw[thick]  (y)--(k1);
\end{tikzpicture}
\end{equation}
correspond to the well known  relations in the tensor calculus
\begin{equation}
f\otimes g=(f\otimes\operatorname{id}_V)(\operatorname{id}_X\otimes g)
=(\operatorname{id}_Y\otimes g)(f\otimes\operatorname{id}_U).
\end{equation}

By taking into account the distinguished role of the identity object  $I$, it is
natural to associate to it the empty graph. 

Let $\mathcal{C}$ be a symmetric monoidal category with tensor product $\otimes$,
the opposite tensor product $\otimes^{\text{op}}$, unit object $I$ and symmetry
$\sigma\colon\otimes\to \otimes^{\text{op}}$.  Recall that a Hopf object in
$\mathcal{C}$  is an object $H$ endowed with five structural morphisms
\be
\label{eq:hopfmaps}  
\nabla\colon H\otimes H\to H,\quad \eta\colon I \to H,\quad
\Delta\colon H\to H\otimes H,\quad \epsilon\colon H\to I,
\quad S\colon H\to H
\ee
called, respectively, product, unit, coproduct, counit and antipode,
that satisfy the following relations or axioms
\begin{subequations}
\begin{align}
\label{Hr1}
\text{associativity :} & \hspace{0.2cm}
\nabla(\nabla \otimes \operatorname{id}_H) = \nabla(\operatorname{id}_H \otimes \nabla)
\\
\label{Hr2}
\text{coassociativity :} & \hspace{0.2cm}
(\Delta \otimes \operatorname{id}_H)\Delta = (\operatorname{id}_H \otimes \Delta)\Delta
\\
\label{Hr3}
\text{unitality :} & \hspace{0.2cm}
\nabla(\eta \otimes \operatorname{id}_H) = \operatorname{id}_H =
\nabla(\operatorname{id}_H \otimes \eta)
\\
\label{Hr4}
\text{counitality :} & \hspace{0.2cm}
(\epsilon \otimes \operatorname{id}_H)\Delta = \operatorname{id}_H =
(\operatorname{id}_H \otimes \epsilon)\Delta
\\
\label{Hr5}
\text{invertibility :} & \hspace{0.2cm}
\nabla(\operatorname{id}_H \otimes S) \Delta = \eta \epsilon =
\nabla(S \otimes \operatorname{id}_H) \Delta
\\
\label{Hr6}
\text{compatibility :} & \hspace{0.2cm}
(\nabla \otimes \nabla)(\operatorname{id}_H \otimes \sigma_{H,H} \otimes
\operatorname{id}_H)(\Delta \otimes \Delta) = \Delta \nabla.
\end{align}
\end{subequations}

Let us introduce the following graphical notation for the structural maps
of $H$ (all lines correspond to the object $H$)
\begin{equation}
\nabla\ =
\begin{tikzpicture}[baseline=-3]
\node (x) [hvector] {  $\nabla$};
\draw[thick]  (x.north)--+(0,10pt);
\draw[thick]  (x.south west) --+(0,-10pt);
\draw[thick]  (x.south east) --+(0,-10pt);
\end{tikzpicture}
\ \ =\ 
\begin{tikzpicture}[scale=1.6,baseline=-2]
\draw[thick]  (0,0)--(0,10pt)(0,0) to [out=-135,in=90] (-5pt,-10pt) (0,0)
  to [out=-45,in=90] (5pt,-10pt);
\end{tikzpicture}
\quad (\text{product}),\qquad
\Delta\ =\ 
\begin{tikzpicture}[baseline=-3]
\node (x) [hvector] {  $\Delta$};
\draw[thick]  (x.south)--+(0,-10pt) (x.north west)--+(0,10pt) (x.north east)--+(0,10pt);
\end{tikzpicture}
\ =\
\begin{tikzpicture}[scale=1.65,baseline=-2]
\draw[thick]  (0,0)--(0,-10pt) (0,0) to [out=135,in=-90] (-5pt,10pt) (0,0)
  to [out=45,in=-90] (5pt,10pt);
\end{tikzpicture}
\quad (\text{coproduct}),
\end{equation}
\begin{equation}
\eta\ =\ 
\begin{tikzpicture}[baseline]
\node (x) [hvector] {  $\eta$};
\draw[thick]  (x.north)--+(0,10pt);
\end{tikzpicture}
\ =\ 
\begin{tikzpicture}[scale=1.5,baseline=2]
\node (x) [unit]{};
\draw[thick]  (x)--+(0,10pt);
\end{tikzpicture}
\quad (\text{unit}),\qquad
\epsilon\ =\ 
\begin{tikzpicture}[baseline=-7]
\node (x) [hvector] {  $\epsilon$};
\draw[thick]  (x.south) --+(0,-10pt);
\end{tikzpicture}
\ =\ 
\begin{tikzpicture}[scale=1.5,baseline=-7]
\node (x)[counit] {};
\draw[thick]  (x)--+(0,-10pt);
\end{tikzpicture}
\quad (\text{counit}),
\end{equation}
\begin{equation}\label{eq:inverse-map-group}
\antipode\ =\ 
\begin{tikzpicture}[baseline=-3]
\node (x) [hvector] {  $\antipode$};
\draw[thick]  (x.north)--+(0,10pt) (x.south)--+(0,-10pt);
\end{tikzpicture}
\ =\ 
\begin{tikzpicture}[scale=1.5,baseline=-3]
\node (x)[antipode]{};
\draw[thick]  (x)--+(0,10pt)(x)--+(0,-10pt);
\end{tikzpicture}
\quad (\text{antipode}).
\end{equation}
We complete this with the graphical notation for the symmetry
\begin{equation}
 \mysym_{H,H} \ =\ 
\begin{tikzpicture}[baseline=-3]
\node (x)[hvector] {  $ \mysym_{H,H} $};
\draw[thick]  (x.north west)--+(0,5pt);
\draw[thick]  (x.north east)--+(0,5pt);
\draw[thick]  (x.south west)--+(0,-5pt);
\draw[thick]  (x.south east)--+(0,-5pt);
\end{tikzpicture}
\ =\ 
\begin{tikzpicture}[scale=1.5,baseline=-3]
\draw[thick]  (-5pt,-10pt) to [out=90,in=-90] (5pt,10pt) (5pt,-10pt)
  to [out=90,in=-90] (-5pt,10pt);
\end{tikzpicture}.
\end{equation}
The relations or axioms of a Hopf object take the following graphical form:
\begin{equation}
\begin{tikzpicture}[scale=2,baseline=10pt]
\draw[thick] (0,0) to [out=90,in=-135](5pt,5pt)(10pt,0) to [out=90,in=-45](5pt,5pt)
(5pt,5pt) to [out=90,in=-135](10pt,10pt)(20pt,0) to [out=90,in=-45](10pt,10pt)
(10pt,10pt)--(10pt,15pt);
\end{tikzpicture}
\ =\ 
\begin{tikzpicture}[scale=2,baseline=10pt]
\draw[thick] (0,0) to [out=90,in=-135] (10pt,10pt) (10pt,0)
to [out=90,in=-135](15pt,5pt)
(15pt,5pt) to [out=90,in=-45] (10pt,10pt) (20pt,0) to [out=90,in=-45] (15pt,5pt)
(10pt,10pt)--(10pt,15pt);
\end{tikzpicture} \quad(\text{associativity}), 
\qquad
\begin{tikzpicture}[scale=2,baseline=6pt]
\node (x) [unit]{};
\draw[thick] (x) to [out=90,in=-135] +(5pt,5pt);
\draw[thick] (x)+(10pt,0) to [out=90,in=-45] ++(5pt,5pt)--+(0,5pt);
\end{tikzpicture}
\ =\ 
\begin{tikzpicture}[scale=2,baseline=6pt]
\draw[thick] (0,0)--(0,10pt);
\end{tikzpicture}
\ =\ 
\begin{tikzpicture}[scale=2,baseline=6pt]
\node (x) [unit]{};
\draw[thick] (x) to [out=90,in=-45] +(-5pt,5pt);
\draw[thick] (x)+(-10pt,0) to [out=90,in=-135] ++(-5pt,5pt)--+(0,5pt);
\end{tikzpicture}
\quad (\text{unitality}),
\end{equation}
\begin{equation}
\begin{tikzpicture}[scale=2,baseline=10pt]
\draw[thick] (0,0)--(0,5pt) to [out=135,in=-90] (-5pt,10pt)
to [out=135,in=-90] (-10pt,15pt)
(-5pt,10pt) to [out=45,in=-90] (0,15pt) (0,5pt) to [out=45,in=-90] (10pt,15pt);
\end{tikzpicture}
\ =\ 
\begin{tikzpicture}[scale=2,baseline=10pt]
\draw[thick] (0,0)--(0,5pt) to [out=45,in=-90] (5pt,10pt) to [out=45,in=-90] (10pt,15pt)
(5pt,10pt) to [out=135,in=-90] (0,15pt) (0,5pt) to [out=135,in=-90] (-10pt,15pt);
\end{tikzpicture}
\quad (\text{coassociativity}),
\qquad
\begin{tikzpicture}[scale=2,baseline=-12]
\node (x) [counit]{};
\draw[thick] (x) to [out=-90,in=135] +(5pt,-5pt);
\draw[thick] (x)+(10pt,0) to [out=-90,in=45] ++(5pt,-5pt)--+(0,-5pt);
\end{tikzpicture}
\ =\ 
\begin{tikzpicture}[scale=2,baseline=8]
\draw[thick] (0,0)--(0,10pt);
\end{tikzpicture}
\ =\ 
\begin{tikzpicture}[scale=2,baseline=-12]
\node (x) [counit]{};
\draw[thick] (x) to [out=-90,in=45] +(-5pt,-5pt);
\draw[thick] (x)+(-10pt,0) to [out=-90,in=135] ++(-5pt,-5pt)--+(0,-5pt);
\end{tikzpicture}\quad
(\text{counitality}),
\end{equation}
\begin{equation}
\begin{tikzpicture}[scale=2,baseline=-3]
\node (x)[antipode]{};
\draw[thick] (x) to [out=90,in=-135](5pt,5pt)(x) to [out=-90,in=135] (5pt,-5pt)
to [out=45,in=-45] (5pt,5pt)
--(5pt,10pt) (5pt,-5pt)--(5pt,-10pt);
\end{tikzpicture}
\ =\ 
\begin{tikzpicture}[scale=2,baseline=-10]
\node (x)[unit]{};
\path (x)+(0,-7pt) node (y) [counit]{};
\draw[thick] (x)--+(0,7pt);
\draw[thick] (y)--+(0,-7pt);
\end{tikzpicture}
\ =\ 
\begin{tikzpicture}[scale=2,baseline=-3]
\node (x)[antipode]{};
\draw[thick] (x) to [out=90,in=-45](-5pt,5pt) (x) to [out=-90,in=45] (-5pt,-5pt)
to [out=135,in=-135] (-5pt,5pt)
--(-5pt,10pt) (-5pt,-5pt)--(-5pt,-10pt);
\end{tikzpicture}\quad
(\text{invertibility}),
\end{equation}
\begin{equation}\label{graph-comp}
\begin{tikzpicture}[scale=3,baseline=20]
\draw[thick] (0,0)--(0,5pt) to [out=135,in=-135] (0,10pt)--(0,15pt)
(10pt,0)--(10pt,5pt) to [out=45,in=-45] (10pt,10pt)--(10pt,15pt)
(0,5pt)--(10pt,10pt) (10pt,5pt)--(0,10pt);
\end{tikzpicture}
\ =\ 
\begin{tikzpicture}[scale=3,baseline=20]
\draw[thick] (0,0) to [out=90,in=-135] (5pt,5pt)--(5pt,10pt)
to [out=135,in=-90] (0,15pt)
(10pt,0) to [out=90,in=-45] (5pt,5pt) (5pt,10pt) to [out=45,in=-90] (10pt,15pt);
\end{tikzpicture}\quad
(\text{compatibility}).
\end{equation}

Our first refinement is the notion of a braided Hopf algebra in a braided monoidal
category. It generalizes the notion of a Hopf object (defined in the context of a
symmetric monoidal category) by replacing the symmetry $\sigma_{H,H}$ in the
compatibility axiom~\eqref{Hr6}, \eqref{graph-comp} by the braiding 
$\tau=\tau_{H,H}: H \otimes H \to H \otimes H$ that satisfies the Yang--Baxter
equation
\begin{equation}
\label{2YB}
(\tau\otimes \operatorname{id}_H)(\operatorname{id}_H\otimes\tau)
(\tau\otimes \operatorname{id}_H)=(\operatorname{id}_H\otimes\tau)
(\tau\otimes\operatorname{id}_H)(\operatorname{id}_H\otimes\tau).
\end{equation}
In other words,  a braided Hopf object (in a braided monoidal category) is defined
by the same set of structural maps~\eqref{eq:hopfmaps} that satisfy
relations~\eqref{Hr1}--\eqref{Hr5}, while in the compatibility relation
~\eqref{Hr6}, \eqref{graph-comp} the symmetry $\sigma$ is replaced by the braiding
$\tau$ 
\begin{equation}
\label{Hr6tau}
\text{compatibility :}  \hspace{0.2cm}
(\nabla \otimes \nabla)(\operatorname{id}_H \otimes \tau_{H,H} \otimes
\operatorname{id}_H)(\Delta \otimes \Delta) = \Delta \nabla.
\end{equation}
One can show that in any braided Hopf algebra, the antipode satisfies the relations
\begin{equation}
\label{Hr5tau}
S \nabla = \nabla \tau_{H,H} (S \otimes S), \qquad 
\Delta S= (S \otimes S)  \tau_{H,H} \Delta
\end{equation}
which can be proven, for example, following the same line of reasoning as in
Section~1.9 of~\cite{Ka:hopf}.

In the diagrammatic language, we denote the braiding morphism by
\begin{equation}
\begin{tikzpicture}[baseline=-3]
\node (x)[hvector] {  $ \tau_{H,H} $};
\draw[thick]  (x.north west)--+(0,10pt);
\draw[thick]  (x.north east)--+(0,10pt);
\draw[thick]  (x.south west)--+(0,-10pt);
\draw[thick]  (x.south east)--+(0,-10pt);
\end{tikzpicture}
\ =\ 
\begin{tikzpicture}[baseline=13,yscale=1.1]
\draw[thick] (0,1) to [out=-90,in=90] (1,0);
\draw[line width=3pt,white] (1,1) to [out=-90,in=90] (0,0);
\draw[thick] (1,1) to [out=-90,in=90] (0,0);
\end{tikzpicture}
\end{equation}
so that the compatibility relation~\eqref{Hr6tau} takes the graphical form
(cf.~\eqref{graph-comp})
\begin{equation}
\begin{tikzpicture}[scale=3,baseline=20]
\draw[thick]
(10pt,5pt)--(0,10pt);
\draw[line width=3pt,white](0,5pt)--(10pt,10pt);
\draw[thick] (0,5pt)--(10pt,10pt);
\draw[thick] (0,0)--(0,5pt) to [out=135,in=-135] (0,10pt)--(0,15pt)
(10pt,0)--(10pt,5pt) to [out=45,in=-45] (10pt,10pt)--(10pt,15pt);
\end{tikzpicture}
\ =\ 
\begin{tikzpicture}[scale=3,baseline=20]
\draw[thick] (0,0) to [out=90,in=-135] (5pt,5pt)--(5pt,10pt)
to [out=135,in=-90] (0,15pt)
(10pt,0) to [out=90,in=-45] (5pt,5pt) (5pt,10pt) to [out=45,in=-90] (10pt,15pt);
\end{tikzpicture}
\end{equation}
and relations~\eqref{Hr5tau} become 
\begin{equation}
\begin{tikzpicture}[yscale=1.25,baseline=34]
\coordinate (pr) at (0.5,1);
\node (x) at (0.5,1.5) [antipode]{};
\draw[thick] (0,0) to [out=90,in=-135] (pr)
 (1,0) to [out=90,in=-45] (pr)--(x)--(0.5,2);
\end{tikzpicture}
\ =\ 
\begin{tikzpicture}[baseline=20]
\coordinate (pr) at (0.5,1.5);
\node (x1) at (0,0) [antipode]{};
\node (x2) at (1,0) [antipode]{};
\draw[thick] (0,1) to [out=-90,in=90] (x2);
\draw[line width=3pt,white] (1,1) to [out=-90,in=90] (x1);
\draw[thick] (1,1) to [out=-90,in=90] (x1);
\draw[thick] (0,-0.5)--(x1)(1,-0.5)--(x2)
(0,1) to [out=90, in=-135] (pr) (1,1) to [out=90, in=-45] (pr)--(0.5,2);
\end{tikzpicture}\ ,\qquad 
\begin{tikzpicture}[yscale=1.25,rotate=180,baseline=-37]
\coordinate (pr) at (0.5,1);
\node (x) at (0.5,1.5) [antipode]{};
\draw[thick] (0,0) to [out=90,in=-135] (pr)
 (1,0) to [out=90,in=-45] (pr)--(x)--(0.5,2);
\end{tikzpicture}
\ =\ 
\begin{tikzpicture}[rotate=180,baseline=-22]
\coordinate (pr) at (0.5,1.5);
\node (x1) at (0,0) [antipode]{};
\node (x2) at (1,0) [antipode]{};
\draw[thick] (0,1) to [out=-90,in=90] (x2);
\draw[line width=3pt,white] (1,1) to [out=-90,in=90] (x1);
\draw[thick] (1,1) to [out=-90,in=90] (x1);
\draw[thick] (0,-0.5)--(x1)(1,-0.5)--(x2)
(0,1) to [out=90, in=-135] (pr) (1,1) to [out=90, in=-45] (pr)--(0.5,2);
\end{tikzpicture}\ .
\end{equation}

The second refinement that we need is the notion of a Hopf f-object or Hopf
f-algebra which corresponds to a pair $(H,\phi)$ composed of a braided Hopf
algebra $H$ and a braided Hopf algebra automorphism $\phi: H \to H$. In addition to
the axioms of a braided Hopf algebra for $H$, the pair $(H,\phi)$ satisfies the
extra compatibility conditions between $\phi$ and all the structural morphisms of $H$:
\begin{equation}
\label{eq:fi-pr-copr}
\nabla (\phi \otimes \phi) = \phi \nabla ,\quad
(\phi \otimes \phi) \Delta=\Delta \phi,
\end{equation}
\begin{equation}\label{eq:fi-s-et-eps}
S \phi = \phi S,\quad \phi \eta = \eta,\quad \epsilon\phi=\epsilon.
\end{equation}
In the diagrammatic notation, we denote the automorphism $\phi$ by
\begin{equation}
\begin{tikzpicture}[baseline=-3]
\node (f) at (0,0)[hvector]{$\phi$};
\draw[thick] (0,-0.5)--(f)--(0,0.5);
\end{tikzpicture}
\ =\ 
\begin{tikzpicture}[baseline=-3]
\node (f) at (0,0)[autom]{};
\draw[thick] (0,-0.5)--(f)--(0,0.5);
\end{tikzpicture}
\end{equation}
so that the additional compatibility relations~\eqref{eq:fi-pr-copr}
and~\eqref{eq:fi-s-et-eps} take the form
\begin{equation}
\begin{tikzpicture}[baseline=5]
\coordinate (pr) at (0.5,0.5);
\node (f1) at (0,0)[autom]{};
\node (f2) at (1,0)[autom]{};
\draw[thick] (0,-0.5)--(f1) to [out=90,in=-135] (pr)
 (1,-0.5)--(f2) to [out=90,in=-45] (pr)--(0.5,1);
\end{tikzpicture}
\ =\ 
\begin{tikzpicture}[baseline=20]
\coordinate (pr) at (0.5,0.5);
\node (f) at (0.5,1)[autom]{};
\draw[thick] (0,0) to [out=90,in=-135] (pr)
 (1,0) to [out=90,in=-45] (pr)--(f)--(0.5,1.5);
\end{tikzpicture},
\qquad
\begin{tikzpicture}[baseline=-9,rotate=180]
\coordinate (pr) at (0.5,0.5);
\node (f1) at (0,0)[autom]{};
\node (f2) at (1,0)[autom]{};
\draw[thick] (0,-0.5)--(f1) to [out=90,in=-135] (pr)
 (1,-0.5)--(f2) to [out=90,in=-45] (pr)--(0.5,1);
\end{tikzpicture}
\ =\ 
\begin{tikzpicture}[baseline=-23,rotate=180]
\coordinate (pr) at (0.5,0.5);
\node (f) at (0.5,1)[autom]{};
\draw[thick] (0,0) to [out=90,in=-135] (pr)
 (1,0) to [out=90,in=-45] (pr)--(f)--(0.5,1.5);
\end{tikzpicture},
\end{equation}
\begin{equation}
\begin{tikzpicture}[baseline=5]
\node (s) at (0,0.5)[antipode]{};
\node (f) at (0,0)[autom]{};
\draw[thick] (0,-0.5)--(f)--(s)--(0,1);
\end{tikzpicture}
\ =\ 
\begin{tikzpicture}[baseline=-9,rotate=180]
\node (s) at (0,0.5)[antipode]{};
\node (f) at (0,0)[autom]{};
\draw[thick] (0,-0.5)--(f)--(s)--(0,1);
\end{tikzpicture}\ ,
\qquad
\begin{tikzpicture}[baseline=10]
\node (s) at (0,0.5)[autom]{};
\node (u) at (0,0)[unit]{};
\draw[thick] (u)--(s)--(0,1);
\end{tikzpicture}
\ = \ 
\begin{tikzpicture}[baseline=10]
\node (u) at (0,0)[unit]{};
\draw[thick] (u)--(0,1);
\end{tikzpicture}\ ,
\qquad
\begin{tikzpicture}[baseline=-15,rotate=180]
\node (s) at (0,0.5)[autom]{};
\node (u) at (0,0)[counit]{};
\draw[thick] (u)--(s)--(0,1);
\end{tikzpicture}
\ = \ 
\begin{tikzpicture}[baseline=-15,rotate=180]
\node (u) at (0,0)[counit]{};
\draw[thick] (u)--(0,1);
\end{tikzpicture}\ .
\end{equation}

\subsection{Diagrammatics of Yetter-Drinfel'd f-objects}
\label{sub.DYdiag}

In this section we recall the definitions for Yetter--Drinfel'd f-objects over
Hopf f-objects and provide the diagrammatic notation for them.

The original Yetter--Drinfel'd modules were defined by Yetter~\cite{Yetter} and
they are essentially modules over Drinfel'd's  quantum double of a Hopf algebra
(hence the name of Drinfel'd). In the early literature, they were also called
crossed modules; see eg.~\cite[p.~385]{Radford}. A detailed definition
of these modules, their properties in the setting of braided Hopf algebras is given
by Takeuchi~\cite{Takeuchi}.

A left Yetter-Drinfel'd f-object over a  Hopf f-object $H$ is a triple
$(Y,\lambda_L,\delta_L)$ where $\lambda_L\colon H\otimes Y\to Y$ and
$\delta_L\colon Y\to H\otimes Y$ satisfy the left module and left comodule equations

\begin{subequations}
\begin{align}
\label{HDYL1}
\text{left action :} & \hspace{0.2cm}
\lambda_L (\operatorname{id}_H \otimes \lambda_L)
= \lambda_L (\nabla \otimes \operatorname{id}_Y)
\\
\label{HDYL2}
\text{left action of unit :} & \hspace{0.2cm}
\lambda_L (\eta \otimes \operatorname{id}_Y) = \operatorname{id}_Y
\\
\label{HDYL3}
\text{left coaction :} & \hspace{0.2cm}
(\operatorname{id}_H \otimes \delta_L) \delta_L
= (\Delta \otimes \operatorname{id}_Y) \delta_L
\\
\label{HDYL4}
\text{left coaction of counit :} & \hspace{0.2cm}
(\epsilon \otimes \operatorname{id}_Y) \delta= \operatorname{id}_Y
\end{align}
\end{subequations}

In the diagrammatic setting, we will color the left f-objects by the blue color
and the right f-objects by the red color. Using this coloring scheme, the
morphisms $\lambda_L$ and $\delta_L$ of the left Yetter-Drinfel'd f-objects are drawn
graphically as
\begin{equation}
\begin{tikzpicture}[baseline=10]
\node (a) at (0.5,0.5)[hvector]{$\lambda_L$};
\draw[thick,blue] (1,0) to [out=90,in=-45](a.south east) (a)--(0.5,1);
\draw[thick] (0,0) to [out=90,in=-135] (a.south west);
\end{tikzpicture}
\ = \ 
\begin{tikzpicture}[baseline=10]
\coordinate (a) at (0.5,0.5);
\draw[thick] (0,0) to [out=90,in=-135] (a);
\draw[thick,blue] (1,0) to [out=90,in=-45] (a)--(0.5,1);
\end{tikzpicture}\ ,
\qquad
\begin{tikzpicture}[baseline=10]
\node (a) at (0.5,0.5)[hvector]{$\delta_L$};
\draw[thick,blue] (1,1) to [out=-90,in=45](a.north east) (a)--(0.5,0);
\draw[thick] (0,1) to [out=-90,in=135] (a.north west);
\end{tikzpicture}
\ = \ 
\begin{tikzpicture}[baseline=10]
\coordinate (a) at (0.5,0.5);
\draw[thick] (0,1) to [out=-90,in=135] (a);
\draw[thick,blue] (1,1) to [out=-90,in=45] (a)--(0.5,0);
\end{tikzpicture},
\end{equation}
so that we obtain the graphical form of Equations~\eqref{HDYL1}--\eqref{HDYL2}
\begin{equation}
\begin{tikzpicture}[scale=3,baseline=10pt]
\coordinate (u) at (10pt,10pt);
\coordinate (c) at (15pt,5pt);
\coordinate (d1) at (0,0);
\coordinate (d2) at (10pt,0);
\coordinate (d3) at (20pt,0);
\draw[thick] (d1) to [out=90,in=-135] (u) (d2) to [out=90,in=-135](c);
\draw[thick,blue](d3) to [out=90,in=-45] (c) to [out=90,in=-45] (u) --(10pt,15pt);
\end{tikzpicture} 
\ =\ 
\begin{tikzpicture}[scale=3,baseline=10pt]
\coordinate (u) at (10pt,10pt);
\coordinate (c) at (5pt,5pt);
\coordinate (d1) at (0,0);
\coordinate (d2) at (10pt,0);
\coordinate (d3) at (20pt,0);
\draw[thick] (d1) to [out=90,in=-135](c)(d2) to [out=90,in=-45](c)
to [out=90,in=-135](u);
\draw[thick,blue](d3) to [out=90,in=-45](u)--(10pt,15pt);
\end{tikzpicture}\ ,
\qquad
\begin{tikzpicture}[scale=3,baseline=10pt]
\coordinate (u) at (10pt,10pt);
\node (d) at (0,0)[unit]{};
\draw[thick] (d) to [out=90,in=-135](u);
\draw[thick,blue](20pt,0) to [out=90,in=-45](u)--(10pt,15pt);
\end{tikzpicture}
\ = \ 
\begin{tikzpicture}[scale=3,baseline=10pt]
\draw[thick,blue] (0,0)--(0,15pt); 
\end{tikzpicture}\ ,
\end{equation}
of Equations~\eqref{HDYL3}--\eqref{HDYL4}
\begin{equation}
\begin{tikzpicture}[scale=3,yscale=-1,baseline=-20pt]
\coordinate (u) at (10pt,10pt);
\coordinate (c) at (15pt,5pt);
\coordinate (d1) at (0,0);
\coordinate (d2) at (10pt,0);
\coordinate (d3) at (20pt,0);
\draw[thick] (d1) to [out=90,in=-135] (u) (d2) to [out=90,in=-135](c);
\draw[thick,blue](d3) to [out=90,in=-45] (c) to [out=90,in=-45] (u) --(10pt,15pt);
\end{tikzpicture} 
\ =\ 
\begin{tikzpicture}[scale=3,yscale=-1,baseline=-20pt]
\coordinate (u) at (10pt,10pt);
\coordinate (c) at (5pt,5pt);
\coordinate (d1) at (0,0);
\coordinate (d2) at (10pt,0);
\coordinate (d3) at (20pt,0);
\draw[thick] (d1) to [out=90,in=-135](c)(d2) to [out=90,in=-45](c)
to [out=90,in=-135](u);
\draw[thick,blue](d3) to [out=90,in=-45](u)--(10pt,15pt);
\end{tikzpicture},
\qquad
\begin{tikzpicture}[scale=3,yscale=-1,baseline=-20pt]
\coordinate (u) at (10pt,10pt);
\node (d) at (0,0)[counit]{};
\draw[thick] (d) to [out=90,in=-135](u);
\draw[thick,blue](20pt,0) to [out=90,in=-45](u)--(10pt,15pt);
\end{tikzpicture}
\ = \ 
\begin{tikzpicture}[scale=3,baseline=25pt]
\draw[thick,blue] (0,0)--(0,15pt); 
\end{tikzpicture}
\end{equation}
and the compatibility relation~\eqref{eq:DYL}
\begin{equation}
\label{greq:comp-blue}
\begin{tikzpicture}[baseline=30,yscale=.5]
\coordinate (b1) at (0,0);
\coordinate (b2) at (1,0);
\coordinate (cp) at (0,.5);
\coordinate (ac) at (0,2);
\coordinate (cac) at (0,2.5);
\coordinate (pr) at (0,4);
\coordinate (t1) at (0,4.5);
\coordinate (t2) at (1,4.5);
\node (f) at (1,2.25)[autom]{}; 
\draw[thick,blue] (b2) to [out=90,in=-45] (ac);
\draw[thick] (f) to [out=90,in=-45] (pr);
\draw[line width=3pt,white] (cac) to [out=45,in=-90] (t2)(cp) to [out=45,in=-90] (f);
\draw[thick] (b1)--(cp) to  [out=135,in=-135] (ac) (cac)
to [out=135,in=-135] (pr)--(t1) (cp) to [out=45,in=-90] (f);
\draw[thick,blue] (ac)--(cac) to [out=45,in=-90] (t2);
\end{tikzpicture}
\ =\ 
\begin{tikzpicture}[baseline=30,yscale=.5]
\coordinate (b1) at (0,0);
\coordinate (b2) at (1,0);
\coordinate (cp) at (0,.5);
\coordinate (cac) at (1,.5);
\coordinate (pr) at (0,4);
\coordinate (ac) at (1,4);
\coordinate (t1) at (0,4.5);
\coordinate (t2) at (1,4.5);
\draw[thick] (cac)--(pr)--(t1);
\draw[line width=3pt,white] (cp)-- (ac);
\draw[thick] (cp) --(ac)(b1)--(cp) to  [out=105,in=-105](pr);
\draw[thick,blue] (b2)--(cac) to [out=75,in=-75] (ac)--(t2);
\end{tikzpicture}.
\end{equation}
Likewise, a right Yetter-Drinfel'd f-object over a  Hopf f-object $H$ is a triple
$(Y,\lambda_R,\delta_R)$ where $\lambda_R\colon Y\otimes H\to Y$ and
$\delta_R\colon Y\to Y\otimes H$ satisfy the right module and right comodule equations

\begin{subequations}
\begin{align}
\label{HDYR1}
\text{right coaction :} & \hspace{0.2cm}
(\delta_R \otimes \operatorname{id}_H) \delta_R
= (\operatorname{id}_Y \otimes \Delta)   \delta_R
\\
\label{HDYR2}
\text{right  coaction of counit :} & \hspace{0.2cm}
(\operatorname{id}_Y \otimes \epsilon) \delta_R = \operatorname{id}_Y
\\
\label{HDYR3}
\text{right action :} & \hspace{0.2cm}
\lambda_R (\lambda_R \otimes \operatorname{id}_H) =
\lambda_R (\operatorname{id}_Y \otimes \Delta)
\\
\label{HDYR4}
\text{right  action of unit :} & \hspace{0.2cm}
\lambda_R (\operatorname{id}_Y \otimes \eta) = \operatorname{id}_Y
\end{align}
\end{subequations}
The corresponding maps $\lambda_R$ and $\delta_R$ of
the right Yetter-Drinfel'd f-objects are denoted by
\begin{equation}
\begin{tikzpicture}[baseline=-18,rotate=180]
\node (a) at (0.5,0.5)[hvector]{$\lambda_R$};
\draw[thick,red] (1,1) to [out=-90,in=45](a.south west) (a)--(0.5,0);
\draw[thick] (0,1) to [out=-90,in=135] (a.south east);
\end{tikzpicture}
\ = \ 
\begin{tikzpicture}[baseline=-18,rotate=180]
\coordinate (a) at (0.5,0.5);
\draw[thick] (0,1) to [out=-90,in=135] (a);
\draw[thick,red] (1,1) to [out=-90,in=45] (a)--(0.5,0);
\end{tikzpicture}\ ,\qquad
\begin{tikzpicture}[baseline=-18,rotate=180]
\node (a) at (0.5,0.5)[hvector]{$\delta_R$};
\draw[thick,red] (1,0) to [out=90,in=-45](a.north west) (a)--(0.5,1);
\draw[thick] (0,0) to [out=90,in=-135] (a.north east);
\end{tikzpicture}
\ = \ 
\begin{tikzpicture}[baseline=-18,rotate=180]
\coordinate (a) at (0.5,0.5);
\draw[thick] (0,0) to [out=90,in=-135] (a);
\draw[thick,red] (1,0) to [out=90,in=-45] (a)--(0.5,1);
\end{tikzpicture},
\end{equation}
so that we have graphical form of Equations~\eqref{HDYR1} and \eqref{HDYR2} 
\begin{equation}
\begin{tikzpicture}[scale=3,baseline=-20pt,rotate=180]
\coordinate (u) at (10pt,10pt);
\coordinate (c) at (15pt,5pt);
\coordinate (d1) at (0,0);
\coordinate (d2) at (10pt,0);
\coordinate (d3) at (20pt,0);
\draw[thick] (d1) to [out=90,in=-135] (u) (d2) to [out=90,in=-135](c);
\draw[thick,red](d3) to [out=90,in=-45] (c) to [out=90,in=-45] (u) --(10pt,15pt);
\end{tikzpicture} 
\ =\ 
\begin{tikzpicture}[scale=3,baseline=-20pt,rotate=180]
\coordinate (u) at (10pt,10pt);
\coordinate (c) at (5pt,5pt);
\coordinate (d1) at (0,0);
\coordinate (d2) at (10pt,0);
\coordinate (d3) at (20pt,0);
\draw[thick] (d1) to [out=90,in=-135](c)(d2) to [out=90,in=-45](c)
to [out=90,in=-135](u);
\draw[thick,red](d3) to [out=90,in=-45](u)--(10pt,15pt);
\end{tikzpicture}\ ,
\qquad
\begin{tikzpicture}[scale=3,baseline=-20pt,rotate=180]
\coordinate (u) at (10pt,10pt);
\node (d) at (0,0)[counit]{};
\draw[thick] (d) to [out=90,in=-135](u);
\draw[thick,red](20pt,0) to [out=90,in=-45](u)--(10pt,15pt);
\end{tikzpicture}
\ = \ 
\begin{tikzpicture}[scale=3,baseline=-20pt,rotate=180]
\draw[thick,red] (0,0)--(0,15pt); 
\end{tikzpicture}\ ,
\end{equation}
Equations~\eqref{HDYR3} and \eqref{HDYR4} 
\begin{equation}
\begin{tikzpicture}[scale=3,yscale=-1,baseline=25pt,rotate=180]
\coordinate (u) at (10pt,10pt);
\coordinate (c) at (15pt,5pt);
\coordinate (d1) at (0,0);
\coordinate (d2) at (10pt,0);
\coordinate (d3) at (20pt,0);
\draw[thick] (d1) to [out=90,in=-135] (u) (d2) to [out=90,in=-135](c);
\draw[thick,red](d3) to [out=90,in=-45] (c) to [out=90,in=-45] (u) --(10pt,15pt);
\end{tikzpicture} 
\ =\ 
\begin{tikzpicture}[scale=3,yscale=-1,baseline=25pt,rotate=180]
\coordinate (u) at (10pt,10pt);
\coordinate (c) at (5pt,5pt);
\coordinate (d1) at (0,0);
\coordinate (d2) at (10pt,0);
\coordinate (d3) at (20pt,0);
\draw[thick] (d1) to [out=90,in=-135](c)(d2) to [out=90,in=-45](c)
to [out=90,in=-135](u);
\draw[thick,red](d3) to [out=90,in=-45](u)--(10pt,15pt);
\end{tikzpicture},
\qquad
\begin{tikzpicture}[scale=3,yscale=-1,baseline=25pt,rotate=180]
\coordinate (u) at (10pt,10pt);
\node (d) at (0,0)[unit]{};
\draw[thick] (d) to [out=90,in=-135](u);
\draw[thick,red](20pt,0) to [out=90,in=-45](u)--(10pt,15pt);
\end{tikzpicture}
\ = \ 
\begin{tikzpicture}[scale=3,baseline=25pt]
\draw[thick,red] (0,0)--(0,15pt); 
\end{tikzpicture}\ ,
\end{equation}
and the compatibility relation~\eqref{eq:DYR}
 \begin{equation}
\begin{tikzpicture}[baseline=-36,yscale=.5,rotate=180]
\coordinate (b1) at (0,0);
\coordinate (b2) at (1,0);
\coordinate (cp) at (0,.5);
\coordinate (ac) at (0,2);
\coordinate (cac) at (0,2.5);
\coordinate (pr) at (0,4);
\coordinate (t1) at (0,4.5);
\coordinate (t2) at (1,4.5);
\node (f) at (1,2.25)[autom]{}; 
\draw[thick,red] (b2) to [out=90,in=-45] (ac);
\draw[thick] (f) to [out=90,in=-45] (pr);
\draw[line width=3pt,white] (cac) to [out=45,in=-90] (t2)(cp) to [out=45,in=-90] (f);
\draw[thick] (b1)--(cp) to  [out=135,in=-135] (ac) (cac)
  to [out=135,in=-135] (pr)--(t1) (cp) to [out=45,in=-90] (f);
\draw[thick,red] (ac)--(cac) to [out=45,in=-90] (t2);
\end{tikzpicture}
\ =\ 
\begin{tikzpicture}[baseline=-36,yscale=.5,rotate=180]
\coordinate (b1) at (0,0);
\coordinate (b2) at (1,0);
\coordinate (cp) at (0,.5);
\coordinate (cac) at (1,.5);
\coordinate (pr) at (0,4);
\coordinate (ac) at (1,4);
\coordinate (t1) at (0,4.5);
\coordinate (t2) at (1,4.5);
\draw[thick] (cac)--(pr)--(t1);
\draw[line width=3pt,white] (cp)-- (ac);
\draw[thick] (cp) --(ac)(b1)--(cp) to  [out=105,in=-105](pr);
\draw[thick,red] (b2)--(cac) to [out=75,in=-75] (ac)--(t2);
\end{tikzpicture}.
\end{equation}
Note that the diagrammatic form of morphisms and relations for right Yetter-Drinfel'd
f-objects are obtained from those of the left Yetter-Drinfel'd f-objects after
rotating the diagram by 180 degrees and replacing the blue color by the red color.

\subsection{Proof of Theorem~\ref{thm:1}}
\label{sub.thm1}

In this section we prove Theorem~\ref{thm:1} using the diagrammatic language
that we have already described. Before doing so, we need the following diagrammatic
notation for the double-iterated coproduct~\eqref{DD2}
\begin{equation}
\begin{tikzpicture}[baseline=12, yscale=.6,xscale=1.3]
\coordinate (t1) at (-.5,2);
\coordinate (t2) at (0,2);
\coordinate (t3) at (.5,2);
\coordinate (b) at (0,0);
\node (d) at (0,1) [hvector]{$\Delta^{(2)}$};
\draw[thick] (b)--(d)--(t2) (d.north west) to [out=135,in=-90] (t1) (d.north east)
to [out=45,in=-90] (t3); 
\end{tikzpicture}
\ =\ 
\begin{tikzpicture}[baseline=12, yscale=.6,xscale=1.3]
\coordinate (t1) at (-.5,2);
\coordinate (t2) at (0,2);
\coordinate (t3) at (.5,2);
\coordinate (b) at (0,0);
\coordinate (d) at (0,1);
\draw[thick] (b)--(d)--(t2) (d) to [out=135,in=-90] (t1) (d) to [out=45,in=-90] (t3); 
\end{tikzpicture}
\ =\ 
\begin{tikzpicture}[xscale=2,yscale=2.1,baseline=10pt]
\draw[thick] (0,0)--(0,5pt) to [out=135,in=-90] (-5pt,10pt) 
to [out=135,in=-90] (-10pt,15pt)
(-5pt,10pt) to [out=45,in=-90] (0,15pt) (0,5pt) to [out=45,in=-90] (10pt,15pt);
\end{tikzpicture}
\ =\ 
\begin{tikzpicture}[xscale=2,yscale=2.1,baseline=10pt]
\draw[thick] (0,0)--(0,5pt) to [out=45,in=-90] (5pt,10pt) to [out=45,in=-90] (10pt,15pt)
(5pt,10pt) to [out=135,in=-90] (0,15pt) (0,5pt) to [out=135,in=-90] (-10pt,15pt);
\end{tikzpicture}\ .
\end{equation}
We will use similar multivalent vertices for higher-iterated coproducts and products.

Using the following graphical representation of the left coaction~\eqref{eq:coaction}
\begin{equation}
\begin{tikzpicture}[yscale=2,baseline=-37,rotate=180]
\node (a) at (0.5,0.5)[hvector]{$\delta_L$};
\draw[thick] (1,0) to [out=90,in=-45](a.north west) (a)--(0.5,1);
\draw[thick] (0,0) to [out=90,in=-135] (a.north east);
\end{tikzpicture}
\ =\ 
\begin{tikzpicture}[baseline=20]
\node (d) at (0,1)[coact]{};
\coordinate (t1) at (-.5,2);
\coordinate (t2) at (.5,2);
\coordinate (b) at (0,0);
\draw[thick] (b)--(d) to  [out=135,in=-90] (t1) (d) to [out=45,in=-90] (t2); 
\end{tikzpicture}
\ =\ 
\begin{tikzpicture}[xscale=1,baseline=20]
\coordinate (t1) at (-.5,2);
\coordinate (t2) at (.5,2);
\node (s) at (.5,1.1)[antipode]{};
\node (f) at (.5,.8)[autom]{};
\coordinate (pr) at (-.5,1.7);
\coordinate (cp) at (0,.3);
\coordinate (b) at (0,0);
\draw[thick] (s) to  [out=90,in=-45] (pr); 
\draw[line width=3pt, white] (cp) to  [out=90,in=-90]  (t2);
\draw[thick] (b)-- (cp) to  [out=90,in=-90]  (t2) (cp)
to  [out=135,in=-135]  (pr)--(t1)(cp) to  [out=15,in=-90] (f)--(s); 
\end{tikzpicture}\ ,
\end{equation}
we can  prove the left coaction property
\begin{equation}
\begin{tikzpicture}[scale=3,yscale=-1,baseline=-20pt]
\node (u) at (10pt,10pt)[coact]{};
\node (c) at (15pt,5pt)[coact]{};
\coordinate (d1) at (0,0);
\coordinate (d2) at (10pt,0);
\coordinate (d3) at (20pt,0);
\draw[thick] (d1) to [out=90,in=-135] (u) (d2) to [out=90,in=-135](c);
\draw[thick](d3) to [out=90,in=-45] (c) to [out=90,in=-45] (u) --(10pt,15pt);
\end{tikzpicture} 
\ =\ 
\begin{tikzpicture}[scale=3,yscale=-1,baseline=-20pt]
\node (u) at (10pt,10pt)[coact]{};
\coordinate (c) at (5pt,5pt);
\coordinate (d1) at (0,0);
\coordinate (d2) at (10pt,0);
\coordinate (d3) at (20pt,0);
\draw[thick] (d1) to [out=90,in=-135](c)(d2) to [out=90,in=-45](c)
to [out=90,in=-135](u);
\draw[thick](d3) to [out=90,in=-45](u)--(10pt,15pt);
\end{tikzpicture}
\end{equation}
as follows:
\begin{equation}
\begin{tikzpicture}[scale=2,yscale=-3,baseline=-50]
\node (u) at (10pt,10pt)[coact]{};
\node (c) at (15pt,5pt)[coact]{};
\coordinate (d1) at (0,0);
\coordinate (d2) at (10pt,0);
\coordinate (d3) at (20pt,0);
\draw[thick] (d1) to [out=90,in=-135] (u) (d2) to [out=90,in=-135](c);
\draw[thick](d3) to [out=90,in=-45] (c) to [out=90,in=-45] (u) --(10pt,15pt);
\end{tikzpicture} 
\ =\ 
\begin{tikzpicture}[xscale=2/3,yscale=.8,baseline=-4]
\coordinate (t1) at (-.5,2);
\coordinate (t2) at (.5,2);
\node (s) at (.5,1.1)[antipode]{};
\node (f) at (.5,.8)[autom]{};
\coordinate (pr) at (-.5,1.7);
\coordinate (cp) at (0,.3);
\coordinate (b) at (0,0);
\draw[thick] (s) to  [out=90,in=-45] (pr) ; 
\draw[line width=3pt, white] (cp) to  [out=90,in=-90]  (t2);
\draw[thick] (b)--(cp) to  [out=90,in=-90]  (t2) (cp)
to  [out=135,in=-135]  (pr)--(t1)(cp) to  [out=15,in=-90] (f)--(s); 
\coordinate (dt1) at (-.5-.5-.5,2);
\coordinate (dt2) at (.5-.5,2-2);
\node (ds) at (.5-.5,1.1-2)[antipode]{};
\node (df) at (.5-.5,.8-2)[autom]{};
\coordinate (dpr) at (-.5-.5,1.7-2);
\coordinate (dcp) at (0-.5,.3-2);
\coordinate (db) at (0-.5,0-2);
\draw[thick] (ds) to  [out=90,in=-45] (dpr) ; 
\draw[line width=3pt, white] (dcp) to  [out=90,in=-90]  (dt2);
\draw[thick] (db)--(dcp) to  [out=90,in=-90]  (dt2) (dcp)
to  [out=135,in=-135]  (dpr) to  [out=90,in=-90](dt1)(dcp)
to  [out=15,in=-90] (df)--(ds) ; 
\end{tikzpicture}
\ =\ 
\begin{tikzpicture}[xscale=2/3,yscale=.8,baseline=-4]
\coordinate (dt1) at (-1.5,2);
\coordinate (t1) at (-.5,2);
\coordinate (t2) at (.5,2);
\node (s) at (.5,.7)[antipode]{};
\node (f) at (.5,.4)[autom]{};
\coordinate (pr) at (-.5,1.7);
\node (ds) at (.5,1.3-2)[antipode]{};
\node (df) at (.5,1-2)[autom]{};
\coordinate (dpr) at (-1.5,1.7);
\coordinate (dcp) at (0-.5,.3-2);
\coordinate (db) at (0-.5,-2);
\draw[thick] (ds) to  [out=90,in=-45] (dpr); 
\draw[line width=3pt, white](dcp) to  [out=110,in=-135]  (pr)(dcp)
to  [out=70,in=-90](f);
\draw[thick] (s) to  [out=90,in=-45] (pr); 
\draw[line width=3pt, white] (dcp) to  [out=90,in=-90]  (t2);
\draw[thick] (dcp) to  [out=90,in=-90]  (t2); 
\draw[thick] (dcp) to  [out=110,in=-135]  (pr)(dcp) to  [out=70,in=-90] (f); 
\draw[thick] (db)--(dcp) (dcp) to  [out=135,in=-90]  (dpr)
to  [out=90,in=-90](dt1)(dcp) to  [out=15,in=-90] (df)--(ds) (pr)--(t1)(f)--(s); 
\end{tikzpicture}
\ =\ 
\begin{tikzpicture}[xscale=2/3,yscale=.8,baseline=-4]
\coordinate (dt1) at (-1.5,2);
\coordinate (t1) at (-.5,2);
\coordinate (t2) at (.5,2);
\node (s) at (.5,.7)[antipode]{};
\node (f) at (.5,.4)[autom]{};
\coordinate (pr) at (-.5,1.7);
\node (ds) at (.5,1.5-2)[antipode]{};
\node (df) at (.5,1.2-2)[autom]{};
\coordinate (dpr) at (-1.5,1.7);
\coordinate (dcp) at (0-.5,.3-2);
\coordinate (db) at (0-.5,0-2);
\coordinate (cp1) at (-1,0);
\coordinate (cp2) at (0,-2+.7);
\draw[thick] (ds) to  [out=90,in=-45] (dpr); 
\draw[line width=3pt, white](cp1) to  [out=90,in=-135]  (pr)(cp2) to [out=90,in=-90](f);
\draw[thick] (s) to  [out=90,in=-45] (pr) ; 
\draw[line width=3pt, white] (dcp) to  [out=90,in=-90]  (t2);
\draw[thick] (dcp) to  [out=90,in=-90]  (t2); 
\draw[thick] (cp1) to  [out=90,in=-135]  (pr)(cp2) to  [out=90,in=-90] (f); 
\draw[thick] (db)--(dcp) (cp1) to  [out=135,in=-90]  (dpr)
to  [out=90,in=-90](dt1)(cp2) to  [out=15,in=-90] (df)--(ds)
(pr)--(t1)(f)--(s)(dcp) to [out=135,in=-90] (cp1)(dcp) to [out=45,in=-90] (cp2); 
\end{tikzpicture}
\ =\ 
\begin{tikzpicture}[xscale=2/3,yscale=.8,baseline=-4]
\coordinate (dt1) at (-1.5,2);
\coordinate (t1) at (-.5,2);
\coordinate (t2) at (.5,2);
\coordinate (pr) at (-.5,1.7);
\node (sn) at (.5,.7-1)[antipode]{};
\node (fn) at (.5,.4-1)[autom]{};
\coordinate (dpr) at (-1.5,1.7);
\coordinate (dcp) at (0-.5,.3-2);
\coordinate (db) at (0-.5,0-2);
\coordinate (cp1) at (-1,0);
\coordinate (cp2) at (.5,0);
\draw[thick] (cp2) to  [out=135,in=-45] (dpr); 
\draw[line width=3pt, white](cp1) to  [out=90,in=-135]  (pr);
\draw[thick] (cp2) to  [out=90,in=-45] (pr); 
\draw[line width=3pt, white] (dcp) to  [out=90,in=-90]  (t2);
\draw[thick] (dcp) to  [out=90,in=-90]  (t2); 
\draw[thick] (cp1) to  [out=90,in=-135]  (pr); 
\draw[thick] (db)--(dcp) (cp1) to  [out=135,in=-90]  (dpr) to  [out=90,in=-90](dt1)
(pr)--(t1)(dcp) to [out=135,in=-90] (cp1)(dcp) to [out=45,in=-90] (fn)--(sn)
to [out=90,in=-90] (cp2); 
\end{tikzpicture}
\ =\ 
\begin{tikzpicture}[xscale=2/3,yscale=.8,baseline=-4]
\coordinate (dt1) at (-1.5,2);
\coordinate (t1) at (-.5,2);
\coordinate (t2) at (.5,2);
\coordinate (pr) at (-.5,1.7);
\coordinate (prn) at (-1,.5);
\coordinate (cpn) at (-1,1);
\node (sn) at (.5,.7-1)[antipode]{};
\node (fn) at (.5,.4-1)[autom]{};
\coordinate (dpr) at (-1.5,1.7);
\coordinate (dcp) at (0-.5,.3-2);
\coordinate (db) at (0-.5,0-2);
\coordinate (cp1) at (-1,0);
\coordinate (cp2) at (.5,0);
\draw[thick] (sn) to  [out=90,in=-45] (prn); 
\draw[line width=3pt, white] (dcp) to  [out=90,in=-90]  (t2);
\draw[thick] (dcp) to  [out=90,in=-90]  (t2); 
\draw[thick] (db)--(dcp)  to  [out=135,in=-90]  (prn) to  [out=90,in=-90](cpn)
(dcp) to [out=45,in=-90] (fn)--(sn)(cpn) to [out=135,in=-90] (dt1)(cpn)
to [out=45,in=-90] (t1); 
\end{tikzpicture}
\ =\ 
\begin{tikzpicture}[scale=2,yscale=-3,baseline=-50]
\node (u) at (10pt,10pt)[coact]{};
\coordinate (c) at (5pt,5pt);
\coordinate (d1) at (0,0);
\coordinate (d2) at (10pt,0);
\coordinate (d3) at (20pt,0);
\draw[thick] (d1) to [out=90,in=-135](c)(d2) to [out=90,in=-45](c)
to [out=90,in=-135](u);
\draw[thick](d3) to [out=90,in=-45](u)--(10pt,15pt);
\end{tikzpicture}\ .
\end{equation}
Similarly, we can prove the compatibility property~\eqref{eq:DYL},
see~\eqref{greq:comp-blue} for its graphical form, which in this case takes the form
\begin{equation}
\begin{tikzpicture}[baseline=30,yscale=.5]
\coordinate (b1) at (0,0);
\coordinate (b2) at (1,0);
\coordinate (cp) at (0,.5);
\coordinate (ac) at (0,2);
\node (cac) at (0,2.5)[coact]{};
\coordinate (pr) at (0,4);
\coordinate (t1) at (0,4.5);
\coordinate (t2) at (1,4.5);
\node (f) at (1,2.25)[autom]{}; 
\draw[thick] (b2) to [out=90,in=-45] (ac);
\draw[thick] (f) to [out=90,in=-45] (pr);
\draw[line width=3pt,white] (cac) to [out=45,in=-90] (t2)(cp) to [out=45,in=-90] (f);
\draw[thick] (b1)--(cp) to  [out=135,in=-135] (ac) (cac)
to [out=135,in=-135] (pr)--(t1) (cp) to [out=45,in=-90] (f);
\draw[thick] (ac)--(cac) to [out=45,in=-90] (t2);
\end{tikzpicture}
\ =\ 
\begin{tikzpicture}[baseline=30,yscale=.5]
\coordinate (b1) at (0,0);
\coordinate (b2) at (1,0);
\coordinate (cp) at (0,.5);
\node (cac) at (1,.5)[coact]{};
\coordinate (pr) at (0,4);
\coordinate (ac) at (1,4);
\coordinate (t1) at (0,4.5);
\coordinate (t2) at (1,4.5);
\draw[thick] (cac)--(pr)--(t1);
\draw[line width=3pt,white] (cp)-- (ac);
\draw[thick] (cp) --(ac)(b1)--(cp) to  [out=105,in=-105](pr);
\draw[thick] (b2)--(cac) to [out=75,in=-75] (ac)--(t2);
\end{tikzpicture}\ .
\end{equation}
Indeed, with a bit longer graphical calculation, we have
\begin{multline}
\begin{tikzpicture}[baseline=60,yscale=1]
\coordinate (b1) at (0,0);
\coordinate (b2) at (1,0);
\coordinate (cp) at (0,.5);
\coordinate (ac) at (0,2);
\node (cac) at (0,2.5)[coact]{};
\coordinate (pr) at (0,4);
\coordinate (t1) at (0,4.5);
\coordinate (t2) at (1,4.5);
\node (f) at (1,2.25)[autom]{}; 
\draw[thick] (b2) to [out=90,in=-45] (ac);
\draw[thick] (f) to [out=90,in=-45] (pr);
\draw[line width=3pt,white] (cac) to [out=45,in=-90] (t2)(cp) to [out=45,in=-90] (f);
\draw[thick] (b1)--(cp) to  [out=135,in=-135] (ac) (cac)
to [out=135,in=-135] (pr)--(t1) (cp) to [out=45,in=-90] (f);
\draw[thick] (ac)--(cac) to [out=45,in=-90] (t2);
\end{tikzpicture}
\ =\ 
\begin{tikzpicture}[baseline=60,yscale=1]
\coordinate (b1) at (0,0);
\coordinate (b2) at (1,0);
\coordinate (cp) at (0,.5);
\coordinate (ac) at (0,2);
\coordinate(cac) at (0,2.5);
\coordinate (pr) at (0,4);
\coordinate (t1) at (0,4.5);
\coordinate (t2) at (1,4.5);
\node (f) at (1,2.25)[autom]{}; 
\node (s) at (.5,3.1)[antipode]{};
\node (f1) at (.5,2.8)[autom]{};
\draw[thick] (b2) to [out=90,in=-45] (ac);
\draw[thick] (f) to [out=90,in=-45] (pr);
\draw[thick] (s) to [out=90,in=-90] (pr);
\draw[line width=3pt,white] (cac) to [out=90,in=-90] (t2)(cp) to [out=45,in=-90] (f);
\draw[thick] (b1)--(cp) to  [out=135,in=-135] (ac) (cac)
to [out=135,in=-135] (pr)--(t1) (cp) to [out=45,in=-90] (f);
\draw[thick] (ac)--(cac) to [out=90,in=-90] (t2);
\draw[thick] (cac) to [out=45,in=-90] (f1)--(s);
\end{tikzpicture}
\ =\ 
\begin{tikzpicture}[baseline=60,yscale=1]
\coordinate (b1) at (0,0);
\coordinate (b2) at (1,0);
\coordinate (cp) at (0,.5);
\coordinate (pr) at (0,4);
\coordinate (t1) at (0,4.5);
\coordinate (t2) at (1,4.5);
\node (f) at (1,2.25)[autom]{}; 
\node (s) at (.5,3.1)[antipode]{};
\node (f1) at (.5,2.8)[autom]{};
\coordinate (pr1) at (-.3,2.3);
\coordinate (pr2) at (0,2.3);
\coordinate (pr3) at (+.3,2.3);
\coordinate (cp1) at (-.3,1.5);
\coordinate (cp2) at (.3,1.5);
\draw[thick] (b2) to [out=90,in=-90] (cp2);
\draw[thick] (f) to [out=90,in=-45] (pr);
\draw[thick] (s) to [out=90,in=-90] (pr);
\draw[thick] (cp2) to [out=135,in=-45] (pr1);
\draw[line width=3pt,white] (cp1) to [out=90,in=-135] (pr2);
\draw[thick]  (cp2) to [out=90,in=-45] (pr2);
\draw[line width=3pt,white] (pr2) to [out=90,in=-90] (t2)(cp)
to [out=45,in=-90] (f)(cp1) to [out=45,in=-135] (pr3);
\draw[thick] (cp1) to [out=135,in=-135] (pr1)(cp1) to [out=90,in=-135] (pr2)(cp1)
to [out=45,in=-135] (pr3);
\draw[thick] (b1)--(cp) to  [out=135,in=-90] (cp1) (pr1)
to [out=90,in=-135] (pr)--(t1) (cp) to [out=45,in=-90] (f);
\draw[thick] (pr2) to [out=90,in=-90] (t2);
\draw[thick] (pr3) to [out=90,in=-90] (f1)--(s)(cp2) to [out=45,in=-45] (pr3);
\end{tikzpicture}
\ =\ 
\begin{tikzpicture}[baseline=60,yscale=1]
\coordinate (b1) at (0,0);
\coordinate (b2) at (1,0);
\coordinate (cp) at (0,.5);
\coordinate (pr) at (0,4);
\coordinate (t1) at (0,4.5);
\coordinate (t2) at (1,4.5);
\node (f) at (1,2.25)[autom]{}; 
\node (s) at (.5,3.1)[antipode]{};
\node (f1) at (.5,2.8)[autom]{};
\coordinate (pr2) at (0,2.3);
\coordinate (pr3) at (+.3,2.3);
\coordinate (cp2) at (.4,1.1);
\draw[thick] (b2) to [out=90,in=-90] (cp2);
\draw[thick] (f) to [out=90,in=-45] (pr);
\draw[thick] (s) to [out=90,in=-90] (pr);
\draw[thick] (cp2) to [out=135,in=-110] (pr);
\draw[line width=3pt,white] (cp) to [out=110,in=-135] (pr2);
\draw[thick]  (cp2) to [out=90,in=-45] (pr2);
\draw[line width=3pt,white] (pr2) to [out=90,in=-90] (t2)(cp)
to [out=45,in=-90] (f)(cp) to [out=90,in=-135] (pr3);
\draw[thick] (cp) to [out=135,in=-135] (pr)(cp) to [out=110,in=-135] (pr2)(cp)
to [out=90,in=-135] (pr3);
\draw[thick] (b1)--(cp) (pr)--(t1) (cp) to [out=45,in=-90] (f);
\draw[thick] (pr2) to [out=90,in=-90] (t2);
\draw[thick] (pr3) to [out=90,in=-90] (f1)--(s)(cp2) to [out=45,in=-45] (pr3);
\end{tikzpicture}
\ =\ 
\begin{tikzpicture}[baseline=60,yscale=1]
\coordinate (b1) at (0,0);
\coordinate (b2) at (1,0);
\coordinate (cp) at (0,.5);
\node (fn1) at (0,1)[autom]{};
\node (sn1) at (0,1.3)[antipode]{};
\coordinate (pr) at (0,4);
\coordinate (t1) at (0,4.5);
\coordinate (t2) at (1,4.5);
\node (f) at (1,2.25)[autom]{}; 
\node (sn2) at (.6,1.8)[antipode]{};
\node (fn2) at (.6,1.5)[autom]{};
\coordinate (prn) at (.5,2.5);
\coordinate (pr2) at (0,2.3);
\coordinate (pr3) at (+.3,2.3);
\coordinate (cp2) at (.4,1.1);
\draw[thick] (b2) to [out=90,in=-90] (cp2);
\draw[thick] (f) to [out=90,in=-45] (pr);
\draw[thick] (prn) to [out=90,in=-90] (pr);
\draw[thick] (cp2) to [out=135,in=-110] (pr);
\draw[line width=3pt,white] (cp) to [out=110,in=-135] (pr2);
\draw[thick]  (cp2) to [out=90,in=-45] (pr2);
\draw[thick]  (sn2) to [out=90,in=-135] (prn);
\draw[line width=3pt,white] (pr2) to [out=90,in=-90] (t2)(cp)
to [out=45,in=-90] (f)(sn1) to [out=90,in=-45] (prn);
\draw[thick] (cp) to [out=135,in=-135] (pr)(cp) to [out=110,in=-135] (pr2)(sn1)
to [out=90,in=-45] (prn);
\draw[thick] (b1)--(cp) (pr)--(t1) (cp) to [out=45,in=-90] (f);
\draw[thick] (pr2) to [out=90,in=-90] (t2);
\draw[thick] (cp2) to [out=45,in=-90] (fn2)(cp)--(fn1)--(sn1)(fn2)--(sn2) (prn)
to [out=-135,in=45] (.46,2.45);
\end{tikzpicture}
\ =\ 
\begin{tikzpicture}[baseline=60,yscale=1]
\coordinate (b1) at (0,0);
\coordinate (b2) at (1,0);
\coordinate (cp) at (0,.5);
\coordinate (pr) at (0,4);
\coordinate (t1) at (0,4.5);
\coordinate (t2) at (1,4.5);
\node (f) at (1,2.25)[autom]{}; 
\node (sn2) at (.4,1.8)[antipode]{};
\node (fn2) at (.4,1.5)[autom]{};
\node (sn3) at (.6,2.55)[antipode]{};
\node (fn3) at (.6,2.25)[autom]{};
\coordinate (prn) at (.5,2.5);
\coordinate (pr2) at (0,2.3);
\coordinate (cp4) at (.8,1.5);
\coordinate (pr4) at (.8,3);
\coordinate (cp2) at (.2,1.1);
\draw[thick] (b2) to [out=90,in=-90] (cp2);
\draw[thick] (pr4) to [out=90,in=-45] (pr);
\draw[thick] (sn2) to [out=90,in=-90] (pr);
\draw[thick] (cp2) to [out=135,in=-110] (pr);
\draw[line width=3pt,white] (cp) to [out=110,in=-135] (pr2);
\draw[thick]  (cp2) to [out=90,in=-45] (pr2);
\draw[line width=3pt,white] (pr2) to [out=90,in=-90] (t2)(cp) to [out=45,in=-90] (cp4);
\draw[thick] (cp) to [out=135,in=-135] (pr)(cp) to [out=110,in=-135] (pr2);
\draw[thick] (b1)--(cp) (pr)--(t1) (cp) to [out=45,in=-90] (cp4);
\draw[thick] (pr2) to [out=90,in=-90] (t2);
\draw[thick] (cp2) to [out=45,in=-90] (fn2)(fn2)--(sn2)(sn3) to [out=90,in=-135] (pr4)
(cp4) to [out=135,in=-90] (fn3)--(sn3)(cp4) to [out=45,in=-90](f)
to [out=90,in=-45](pr4);
\end{tikzpicture}\\
\ =\ 
\begin{tikzpicture}[baseline=60,yscale=1]
\coordinate (b1) at (0,0);
\coordinate (b2) at (1,0);
\coordinate (cp) at (0,.5);
\coordinate (pr) at (0,4);
\coordinate (t1) at (0,4.5);
\coordinate (t2) at (1,4.5);
\node (sn2) at (.4,1.8)[antipode]{};
\node (fn2) at (.4,1.5)[autom]{};
\node (sn3) at (.6,2.25)[antipode]{};
\coordinate (prn) at (.5,2.5);
\coordinate (pr2) at (0,2.3);
\coordinate (cp4) at (.8,1.5);
\coordinate (pr4) at (.8,3);
\coordinate (cp2) at (.2,1.1);
\node (ff) at (.8,1.2)[autom]{};
\draw[thick] (b2) to [out=90,in=-90] (cp2);
\draw[thick] (pr4) to [out=90,in=-45] (pr);
\draw[thick] (sn2) to [out=90,in=-90] (pr);
\draw[thick] (cp2) to [out=135,in=-110] (pr);
\draw[line width=3pt,white] (cp) to [out=110,in=-135] (pr2);
\draw[thick]  (cp2) to [out=90,in=-45] (pr2);
\draw[line width=3pt,white] (pr2) to [out=90,in=-90] (t2)(cp) to [out=45,in=-90] (ff);
\draw[thick] (cp) to [out=135,in=-135] (pr)(cp) to [out=110,in=-135] (pr2);
\draw[thick] (b1)--(cp) (pr)--(t1) (cp) to [out=45,in=-90] (ff)--(cp4);
\draw[thick] (pr2) to [out=90,in=-90] (t2);
\draw[thick] (cp2) to [out=45,in=-90] (fn2)(fn2)--(sn2)(sn3) to [out=90,in=-120] (pr4)
 (cp4) to [out=120,in=-90] (sn3)(cp4) to [out=60,in=-60](pr4);
\end{tikzpicture}
\ =\ 
\begin{tikzpicture}[baseline=60,yscale=1]
\coordinate (b1) at (0,0);
\coordinate (b2) at (1,0);
\coordinate (cp) at (0,.5);
\coordinate (pr) at (0,4);
\coordinate (t1) at (0,4.5);
\coordinate (t2) at (1,4.5);
\node (sn2) at (.4,1.8)[antipode]{};
\node (fn2) at (.4,1.5)[autom]{};
\coordinate (prn) at (.5,2.5);
\coordinate (pr2) at (0,2.3);
\node (cp4) at (.8,2)[counit]{};
\node (pr4) at (.8,2.5)[unit]{};
\coordinate (cp2) at (.2,1.1);
\node (ff) at (.8,1.2)[autom]{};
\draw[thick] (b2) to [out=90,in=-90] (cp2);
\draw[thick] (pr4) to [out=90,in=-45] (pr);
\draw[thick] (sn2) to [out=90,in=-90] (pr);
\draw[thick] (cp2) to [out=135,in=-110] (pr);
\draw[line width=3pt,white] (cp) to [out=110,in=-135] (pr2);
\draw[thick]  (cp2) to [out=90,in=-45] (pr2);
\draw[line width=3pt,white] (pr2) to [out=90,in=-90] (t2)(cp) to [out=45,in=-90] (ff);
\draw[thick] (cp) to [out=135,in=-135] (pr)(cp) to [out=110,in=-135] (pr2);
\draw[thick] (b1)--(cp) (pr)--(t1) (cp) to [out=45,in=-90] (ff)--(cp4);
\draw[thick] (pr2) to [out=90,in=-90] (t2);
\draw[thick] (cp2) to [out=45,in=-90] (fn2)(fn2)--(sn2)
 ;
\end{tikzpicture}
\ =\ 
\begin{tikzpicture}[baseline=60,yscale=1]
\coordinate (b1) at (0,0);
\coordinate (b2) at (1,0);
\coordinate (cp) at (0,.5);
\coordinate (pr) at (0,4);
\coordinate (t1) at (0,4.5);
\coordinate (t2) at (1,4.5);
\node (sn2) at (1,2.3)[antipode]{};
\node (fn2) at (1,2)[autom]{};
\coordinate (prn) at (.5,2.5);
\coordinate (pr2) at (.5,2.25);
\coordinate (cp2) at (.7,.8);
\draw[thick] (b2) to [out=90,in=-90] (cp2);
\draw[thick] (sn2) to [out=90,in=-45] (pr);
\draw[thick] (cp2) to [out=135,in=-90] (pr);
\draw[line width=3pt,white] (cp) to [out=45,in=-135] (pr2);
\draw[thick]  (cp2) to [out=90,in=-45] (pr2);
\draw[line width=3pt,white] (pr2) to [out=90,in=-90] (t2);
\draw[thick] (cp) to [out=135,in=-135] (pr)(cp) to [out=45,in=-135] (pr2);
\draw[thick] (b1)--(cp) (pr)--(t1);
\draw[thick] (pr2) to [out=90,in=-90] (t2);
\draw[thick] (cp2) to [out=45,in=-90] (fn2)--(sn2);
\end{tikzpicture}
\ =\ 
\begin{tikzpicture}[baseline=60,yscale=1]
\coordinate (b1) at (0,0);
\coordinate (b2) at (1,0);
\coordinate (cp) at (0,.5);
\coordinate (pr) at (0,4);
\coordinate (pr1) at (.9,2);
\coordinate (t1) at (0,4.5);
\coordinate (t2) at (1,4.5);
\node (sn2) at (1.2,1.4)[antipode]{};
\node (fn2) at (1.2,1.1)[autom]{};
\coordinate (pr2) at (1,4);
\coordinate (cp2) at (1,.8);
\draw[thick] (b2) to [out=90,in=-90] (cp2);
\draw[thick] (sn2) to [out=90,in=-45] (pr1);
\draw[line width=3pt,white] (cp2) to [out=90,in=-45] (pr2);
\draw[thick] (cp2) to [out=135,in=-135] (pr1) (pr1) to [out=90,in=-45] (pr);
\draw[line width=3pt,white] (cp) to [out=45,in=-135] (pr2);
\draw[thick]  (cp2) to [out=90,in=-45] (pr2);
\draw[line width=3pt,white] (pr2) to [out=90,in=-90] (t2)
  ;
\draw[thick] (cp) to [out=135,in=-135] (pr)(cp) to [out=45,in=-135] (pr2);
\draw[thick] (b1)--(cp) (pr)--(t1) 
;
\draw[thick] (pr2) to [out=90,in=-90] (t2);
\draw[thick] (cp2) to [out=45,in=-90] (fn2)--(sn2)
 ;
\end{tikzpicture}
\ =\ 
\begin{tikzpicture}[baseline=60,yscale=1]
\coordinate (b1) at (0,0);
\coordinate (b2) at (1,0);
\coordinate (cp) at (0,.5);
\node (cac) at (1,.5)[coact]{};
\coordinate (pr) at (0,4);
\coordinate (ac) at (1,4);
\coordinate (t1) at (0,4.5);
\coordinate (t2) at (1,4.5);
\draw[thick] (cac)--(pr)--(t1);
\draw[line width=3pt,white] (cp)-- (ac);
\draw[thick] (cp) --(ac)(b1)--(cp) to  [out=105,in=-105](pr);
\draw[thick] (b2)--(cac) to [out=75,in=-75] (ac)--(t2);
\end{tikzpicture}.
\end{multline}

This completes the proof of part (a) of Theorem~\ref{thm:1}. The proof
of part (b) is analogous, and is omitted.
\qed

\subsection{Proof of Theorem~\ref{thm:2}}
\label{sub.thm2}

In this section we show that the $R$-matrix~\eqref{eq:r-matrix} satisfies the
Yang--Baxter equation~\eqref{tauYB}, and omit the analogous proofs that the
$R$-matrix~\eqref{eq:r-matrix-r} also satisfies the Yang--Baxter equation.

To begin with, the diagrammatic notation for the $R$-matrices $\rho_L$ and $\rho_R$
is given as follows
\begin{equation}
\begin{tikzpicture}[baseline=-3]
\node (x)[hvector] {  $\rho_L $};
\draw[thick,blue]  (x.north west)--+(0,10pt);
\draw[thick,blue]  (x.north east)--+(0,10pt);
\draw[thick,blue]  (x.south west)--+(0,-10pt);
\draw[thick,blue]  (x.south east)--+(0,-10pt);
\end{tikzpicture}
\ =\
\begin{tikzpicture}[xscale=1/2, yscale=1.3,baseline=16]
\coordinate (t1) at (0,1);
\coordinate (t2) at (1,1);
\coordinate (b1) at (0,0);
\coordinate (b2) at (1,0);
\node (f) at (1, .25)[autom, blue]{};
\coordinate (cac) at (0,.25);
\coordinate (ac) at (0,.75);
\draw[thick,blue] (f) to [out=90,in=-45] (ac);
\draw[line width=3pt,white] (cac) to [out=45,in=-90] (t2);
\draw[thick,blue] (cac) to [out=45,in=-90] (t2);
\draw[thick] (cac) to [out=135,in=-135] (ac);
\draw[thick,blue] (b1) to [out=90,in=-90] (cac)(ac) to [out=90,in=-90](t1)(b2)
to [out=90,in=-90](f);
\end{tikzpicture}\ ,
\qquad
\begin{tikzpicture}[baseline=-3]
\node (x)[hvector] {  $ \rho_R $};
\draw[thick,red]  (x.north west)--+(0,10pt);
\draw[thick,red]  (x.north east)--+(0,10pt);
\draw[thick,red]  (x.south west)--+(0,-10pt);
\draw[thick,red]  (x.south east)--+(0,-10pt);
\end{tikzpicture}
\ =\
\begin{tikzpicture}[xscale=1/2, yscale=1.3,baseline=-21,rotate=180]
\coordinate (t1) at (0,1);
\coordinate (t2) at (1,1);
\coordinate (b1) at (0,0);
\coordinate (b2) at (1,0);
\node (f) at (1, .25)[autom, red]{};
\coordinate (cac) at (0,.25);
\coordinate (ac) at (0,.75);
\draw[thick,red] (f) to [out=90,in=-45] (ac);
\draw[line width=3pt,white] (cac) to [out=45,in=-90] (t2);
\draw[thick,red] (cac) to [out=45,in=-90] (t2);
\draw[thick] (cac) to [out=135,in=-135] (ac);
\draw[thick,red] (b1) to [out=90,in=-90] (cac)(ac) to [out=90,in=-90](t1)(b2)
to [out=90,in=-90](f);
\end{tikzpicture}\ .
\end{equation}
The proof of Theorem~\ref{thm:2} is now given as follows:
\begin{multline}
\begin{tikzpicture}[xscale=.5, yscale=1,baseline=40]
\coordinate (t1) at (0,3);
\coordinate (t2) at (1,3);
\coordinate (t3) at (2,3);
\coordinate (b1) at (0,0);
\coordinate (b2) at (1,0);
\coordinate (b3) at (2,0);
\node (f1) at (1, 2.25)[autom, blue]{};
\node (f2) at (2, 1.25)[autom, blue]{};
\node (f3) at (1, .25)[autom, blue]{};
\coordinate (cac1) at (0,2.25);
\coordinate (ac1) at (0,2.75);
\coordinate (cac2) at (1,1.25);
\coordinate (ac2) at (1,1.75);
\coordinate (cac3) at (0,.25);
\coordinate (ac3) at (0,.75);
\draw[thick,blue] (f3) to [out=90,in=-45] (ac3);
\draw[thick,blue] (f2) to [out=90,in=-45] (ac2);
\draw[thick,blue] (f1) to [out=90,in=-45] (ac1);
\draw[line width=3pt,white] (cac3) to [out=45,in=-90] (cac2) to [out=45,in=-90]
(t3)(cac1) to [out=45,in=-90] (t2);
\draw[thick,blue] (cac3) to [out=45,in=-90] (cac2) to [out=45,in=-90] (t3)(cac1)
to [out=45,in=-90] (t2);
\draw[thick] (cac3) to [out=135,in=-135] (ac3)(cac2) to [out=135,in=-135] (ac2)(cac1)
to [out=135,in=-135] (ac1);
\draw[thick,blue] (b1) to [out=90,in=-90] (cac3)(ac3) to [out=90,in=-90](cac1)(b2)
to [out=90,in=-90](f3)(b3) to [out=90,in=-90](f2)
(ac2) to [out=90,in=-90](f1)(ac1) to [out=90,in=-90](t1);
\end{tikzpicture}
\ =\ 
\begin{tikzpicture}[xscale=.5, yscale=1,baseline=40]
\coordinate (t1) at (0,3);
\coordinate (t2) at (1,3);
\coordinate (t3) at (2,3);
\coordinate (b1) at (0,0);
\coordinate (b2) at (1,0);
\coordinate (b3) at (2,0);
\node (f1) at (2, 2.25)[autom, blue]{};
\node (f2) at (1, 1.25)[autom, blue]{};
\node (f3) at (2, .25)[autom, blue]{};
\coordinate (cac1) at (1,2.25);
\coordinate (ac1) at (1,2.75);
\coordinate (cac2) at (0,1.25);
\coordinate (ac2) at (0,1.75);
\coordinate (cac3) at (1,.25);
\coordinate (ac3) at (1,.75);
\draw[thick,blue] (f3) to [out=90,in=-45] (ac3);
\draw[thick,blue] (f2) to [out=90,in=-45] (ac2);
\draw[thick,blue] (f1) to [out=90,in=-45] (ac1);
\draw[line width=3pt,white] (cac3) to [out=45,in=-90] (f1) (cac2) to [out=45,in=-90]
(cac1) to [out=45,in=-90] (t3);
\draw[thick,blue]  (cac3) to [out=45,in=-90] (f1) (cac2) to [out=45,in=-90] (cac1)
to [out=45,in=-90] (t3);
\draw[thick] (cac3) to [out=135,in=-135] (ac3)(cac2) to [out=135,in=-135] (ac2)(cac1)
to [out=135,in=-135] (ac1);
\draw[thick,blue] (b1) to [out=90,in=-90] (cac2)(b2) to [out=90,in=-90](cac3)(ac3)
to [out=90,in=-90](f2)
(ac2) to [out=90,in=-90](t1)(ac1) to [out=90,in=-90](t2)(b3)
to [out=90,in=-90](f3);
\end{tikzpicture}
\ \Leftrightarrow\ 
\begin{tikzpicture}[xscale=.5, yscale=1,baseline=40]
\coordinate (t1) at (0,3);
\coordinate (t2) at (1,3);
\coordinate (t3) at (2,3);
\coordinate (b1) at (0,0);
\coordinate (b2) at (1,0);
\coordinate (b3) at (2,0);
\node (f1) at (2, .9)[autom, blue]{};
\node (f2) at (2, 1.25)[autom, blue]{};
\node (f3) at (1, .25)[autom, blue]{};
\node (fb) at (.8, 1.5)[autom]{};
\coordinate (cac1) at (0,2.25);
\coordinate (ac1) at (0,2.75);
\coordinate (cac2) at (1,1.25);
\coordinate (ac2) at (1,1.75);
\coordinate (cac3) at (0,.25);
\coordinate (ac3) at (0,.75);
\draw[thick,blue] (f3) to [out=90,in=-45] (ac3);
\draw[thick,blue] (ac2) to [out=90,in=-45] (ac1);
\draw[thick,blue] (f2) to [out=90,in=-45] (ac2);
\draw[line width=3pt,white] (cac3) to [out=45,in=-90] (cac2) to [out=45,in=-90]
(t3)(cac1) to [out=45,in=-90] (t2);
\draw[thick,blue] (cac3) to [out=45,in=-90] (cac2) to [out=45,in=-90] (t3)(cac1)
to [out=45,in=-90] (t2);
\draw[thick] (cac3) to [out=135,in=-135] (ac3)(cac2) to [out=135,in=-90] (fb)
to [out=90,in=-135](ac2)(cac1) to [out=135,in=-135] (ac1);
\draw[thick,blue] (b1) to [out=90,in=-90] (cac3)(ac3) to [out=90,in=-90](cac1)(b2)
to [out=90,in=-90](f3)(b3) to [out=90,in=-90](f1)
to [out=90,in=-90](f2) (ac1) to [out=90,in=-90](t1);
\end{tikzpicture}
\ =\ 
\begin{tikzpicture}[xscale=.5, yscale=1,baseline=40]
\coordinate (t1) at (0,3);
\coordinate (t2) at (1,3);
\coordinate (t3) at (2,3);
\coordinate (b1) at (0,0);
\coordinate (b2) at (1,0);
\coordinate (b3) at (2,0);
\coordinate (f1) at (2,1.5);
\node (f1m) at (1, .25)[autom, blue]{};
\node (f2m) at (2, .6)[autom, blue]{};
\node (f3) at (2, .25)[autom, blue]{};
\coordinate (cac1) at (1,2.25);
\coordinate (ac1) at (1,2.75);
\coordinate (cac2) at (0,1.25);
\coordinate (ac2) at (0,1.75);
\coordinate (cac3) at (1,.75);
\coordinate (ac3) at (1,1.25);
\draw[thick,blue] (f2m) to [out=90,in=-45] (ac3);
\draw[thick,blue] (ac3) to [out=90,in=-45] (ac2);
\draw[thick,blue] (f1) to [out=90,in=-45] (ac1);
\draw[line width=3pt,white] (cac3) to [out=45,in=-90] (f1) (cac2) to [out=45,in=-90]
(cac1) to [out=45,in=-90] (t3);
\draw[thick,blue]  (cac3) to [out=45,in=-90] (f1) (cac2) to [out=45,in=-90] (cac1)
to [out=45,in=-90] (t3);
\draw[thick] (cac3) to [out=135,in=-135] (ac3)(cac2) to [out=135,in=-135] (ac2)(cac1)
to [out=135,in=-135] (ac1);
\draw[thick,blue] (b1) to [out=90,in=-90] (cac2)(b2) to [out=90,in=-90](f1m)
to [out=90,in=-90](cac3) (ac2) to [out=90,in=-90](t1)(ac1)
to [out=90,in=-90](t2)(b3) to [out=90,in=-90](f3)--(f2m);
\end{tikzpicture}
\ \Leftrightarrow\ 
\begin{tikzpicture}[xscale=.5, yscale=1,baseline=40]
\coordinate (t1) at (0,3);
\coordinate (t2) at (1,3);
\coordinate (t3) at (2,3);
\coordinate (b1) at (0,0);
\coordinate (b2) at (1,0);
\coordinate (b3) at (2,0);
\node (fb) at (.8, 1.5)[autom]{};
\coordinate (cac1) at (0,2.25);
\coordinate (ac1) at (0,2.75);
\coordinate (cac2) at (1,1.25);
\coordinate (ac2) at (1,1.75);
\coordinate (cac3) at (0,.25);
\coordinate (ac3) at (0,.75);
\draw[thick,blue] (b2) to [out=90,in=-45] (ac3);
\draw[thick,blue] (ac2) to [out=90,in=-45] (ac1);
\draw[thick,blue] (b3) to [out=90,in=-45] (ac2);
\draw[line width=3pt,white] (cac3) to [out=45,in=-90] (cac2)
to [out=45,in=-90] (t3)(cac1) to [out=45,in=-90] (t2);
\draw[thick,blue] (cac3) to [out=45,in=-90] (cac2) to [out=45,in=-90] (t3)(cac1)
to [out=45,in=-90] (t2);
\draw[thick] (cac3) to [out=135,in=-135] (ac3)(cac2) to [out=135,in=-90] (fb)
to [out=90,in=-135](ac2)(cac1) to [out=135,in=-135] (ac1);
\draw[thick,blue] (b1) to [out=90,in=-90] (cac3)(ac3)
to [out=90,in=-90](cac1) (ac1) to [out=90,in=-90](t1);
\end{tikzpicture}
\ =\ 
\begin{tikzpicture}[xscale=.5, yscale=1,baseline=40]
\coordinate (t1) at (0,3);
\coordinate (t2) at (1,3);
\coordinate (t3) at (2,3);
\coordinate (b1) at (0,0);
\coordinate (b2) at (1,0);
\coordinate (b3) at (2,0);
\coordinate (f1) at (2,1.5);
\coordinate (cac1) at (1,2.25);
\coordinate (ac1) at (1,2.75);
\coordinate (cac2) at (0,1.25);
\coordinate (ac2) at (0,1.75);
\coordinate (cac3) at (1,.75);
\coordinate (ac3) at (1,1.25);
\draw[thick,blue] (b3) to [out=90,in=-45] (ac3);
\draw[thick,blue] (ac3) to [out=90,in=-45] (ac2);
\draw[thick,blue] (f1) to [out=90,in=-45] (ac1);
\draw[line width=3pt,white] (cac3) to [out=45,in=-90] (f1) (cac2) to [out=45,in=-90]
(cac1) to [out=45,in=-90] (t3);
\draw[thick,blue]  (cac3) to [out=45,in=-90] (f1) (cac2) to [out=45,in=-90] (cac1)
to [out=45,in=-90] (t3);
\draw[thick] (cac3) to [out=135,in=-135] (ac3)(cac2) to [out=135,in=-135] (ac2)(cac1)
to [out=135,in=-135] (ac1);
\draw[thick,blue] (b1) to [out=90,in=-90] (cac2)(b2) to [out=90,in=-90](cac3)
(ac2) to [out=90,in=-90](t1)(ac1) to [out=90,in=-90](t2);
\end{tikzpicture}\\
\ \Leftrightarrow\ 
\begin{tikzpicture}[xscale=.5, yscale=1,baseline=40]
\coordinate (t1) at (0,3);
\coordinate (t2) at (1,3);
\coordinate (t3) at (2,3);
\coordinate (b1) at (0,0);
\coordinate (b2) at (1,0);
\coordinate (b3) at (2,0);
\node (fb) at (.5, 1.5)[autom]{};
\coordinate (cac1) at (0,1.7);
\coordinate (ac1) at (0,2.75);
\coordinate (cac2) at (1,1.25);
\coordinate (ac2) at (1,1.75);
\coordinate (cac3) at (0,.25);
\coordinate (ac3) at (0,1.3);
\coordinate (pr) at (0,2.3);
\coordinate (cp) at (0,.7);
\draw[thick,blue] (b2) to [out=90,in=-45] (ac3);
\draw[thick,blue] (b3) to [out=90,in=-45] (ac1);
\draw[thick] (fb) to [out=90,in=-45](pr);
\draw[line width=3pt,white] (cac3) to [out=45,in=-90] (t3) (cp) to [out=45,in=-90]
(fb)(cac1) to [out=45,in=-90] (t2);
\draw[thick,blue] (cac3) to [out=45,in=-90] (t3)(cac1) to [out=45,in=-90] (t2);
\draw[thick] (cac3) to [out=135,in=-90] (cp) to [out=135,in=-135](ac3)(cp)
to [out=45,in=-90] (fb) (cac1) to [out=135,in=-135] (pr)
to [out=90,in=-135](ac1);
\draw[thick,blue] (b1) to [out=90,in=-90] (cac3)(ac3) to [out=90,in=-90](cac1)
(ac1) to [out=90,in=-90](t1);
\end{tikzpicture}
\ =\ 
\begin{tikzpicture}[xscale=.5, yscale=1,baseline=40]
\coordinate (t1) at (0,3);
\coordinate (t2) at (1,3);
\coordinate (t3) at (2,3);
\coordinate (b1) at (0,0);
\coordinate (b2) at (1,0);
\coordinate (b3) at (2,0);
\coordinate (f1) at (2,1.5);
\coordinate (cac1) at (1,2);
\coordinate (ac1) at (1,2.5);
\coordinate (cac2) at (0,.5);
\coordinate (ac2) at (0,2.5);
\coordinate (cac3) at (1,.5);
\coordinate (ac3) at (1,1);
\coordinate (pr) at (0,2);
\coordinate (cp) at (0,1);
\draw[thick,blue] (b3) to [out=90,in=-45] (ac2);
\draw[line width=3pt,white] (cac3) to [out=45,in=-90] (f1); 
\draw[thick,blue]  (cac3) to [out=45,in=-90] (f1);
\draw[thick] (cac3) to [out=135,in=-45] (pr);
\draw[line width=3pt,white] (cp) to [out=45,in=-135] (ac1);
\draw[thick] (cp) to [out=45,in=-135] (ac1);
\draw[thick,blue] (f1) to [out=90,in=-45] (ac1);
\draw[line width=3pt,white] (cac2) to [out=45,in=-90] (t3) to [out=45,in=-90] (t3);
\draw[thick,blue]  (cac2) to [out=45,in=-90] (t3) to [out=45,in=-90] (t3);
\draw[thick](cac2) to [out=135,in=-90] (cp) to [out=135,in=-135] (pr)
to [out=90,in=-135] (ac2);
\draw[thick,blue] (b1) to [out=90,in=-90] (cac2)(b2)
to [out=90,in=-90](cac3) (ac2)
to [out=90,in=-90](t1)(ac1) to [out=90,in=-90](t2);
\end{tikzpicture}
\ \Leftarrow\ 
\begin{tikzpicture}[xscale=.5, yscale=1,baseline=40]
\coordinate (t1) at (0,3);
\coordinate (t2) at (1,3);
\coordinate (t3) at (2,3);
\coordinate (b1) at (0,0);
\coordinate (b2) at (1,0);
\coordinate (b3) at (2,0);
\node (fb) at (.5, 1.5)[autom]{};
\coordinate (cac1) at (0,1.7);
\coordinate (ac1) at (0,2.75);
\coordinate (cac2) at (1,1.25);
\coordinate (ac2) at (1,1.75);
\coordinate (cac3) at (0,.25);
\coordinate (ac3) at (0,1.3);
\coordinate (pr) at (0,2.3);
\coordinate (cp) at (0,.7);
\draw[thick,blue] (b2) to [out=90,in=-45] (ac3);
\draw[thick,blue] (b3) to [out=90,in=-45] (ac1);
\draw[thick] (fb) to [out=90,in=-45](pr);
\draw[line width=3pt,white]  (cp) to [out=45,in=-90] (fb)(cac1) 
to [out=45,in=-90] (t2);
\draw[thick,blue] (cac1) to [out=45,in=-90] (t2);
\draw[thick] (b1) to [out=90,in=-90] (cp) to [out=135,in=-135](ac3)(cp)
to [out=45,in=-90] (fb) (cac1) to [out=135,in=-135] (pr)
to [out=90,in=-135](ac1);
\draw[thick,blue] (ac3) to [out=90,in=-90](cac1) (ac1) to [out=90,in=-90](t1);
\end{tikzpicture}
\ =\ 
\begin{tikzpicture}[xscale=.5, yscale=1,baseline=40]
\coordinate (t1) at (0,3);
\coordinate (t2) at (1,3);
\coordinate (t3) at (2,3);
\coordinate (b1) at (0,0);
\coordinate (b2) at (1,0);
\coordinate (b3) at (2,0);
\coordinate (f1) at (2,1.5);
\coordinate (cac1) at (1,2);
\coordinate (ac1) at (1,2.5);
\coordinate (cac2) at (0,.5);
\coordinate (ac2) at (0,2.5);
\coordinate (cac3) at (1,.5);
\coordinate (ac3) at (1,1);
\coordinate (pr) at (0,2);
\coordinate (cp) at (0,1);
\draw[thick,blue] (b3) to [out=90,in=-45] (ac2);
\draw[line width=3pt,white] (cac3) to [out=45,in=-90] (f1); 
\draw[thick,blue]  (cac3) to [out=45,in=-90] (f1);
\draw[thick] (cac3) to [out=135,in=-45] (pr);
\draw[line width=3pt,white] (cp) to [out=45,in=-135] (ac1);
\draw[thick] (cp) to [out=45,in=-135] (ac1);
\draw[thick,blue] (f1) to [out=90,in=-45] (ac1);
\draw[thick](b1) to [out=90,in=-90] (cp) to [out=135,in=-135] (pr)
to [out=90,in=-135] (ac2);
\draw[thick,blue] (b2) to [out=90,in=-90](cac3) (ac2) to [out=90,in=-90](t1)(ac1)
to [out=90,in=-90](t2);
\end{tikzpicture}
\ \Leftarrow\ 
\begin{tikzpicture}[xscale=.5, yscale=1,baseline=40]
\coordinate (t1) at (0,3);
\coordinate (t2) at (1,3);
\coordinate (t3) at (2,3);
\coordinate (b1) at (0,0);
\coordinate (b2) at (1,0);
\coordinate (b3) at (2,0);
\node (fb) at (.5, 1.5)[autom]{};
\coordinate (cac1) at (0,1.7);
\coordinate (ac1) at (0,2.75);
\coordinate (cac2) at (1,1.25);
\coordinate (ac2) at (1,1.75);
\coordinate (cac3) at (0,.25);
\coordinate (ac3) at (0,1.3);
\coordinate (pr) at (0,2.3);
\coordinate (cp) at (0,.7);
\draw[thick,blue] (b2) to [out=90,in=-45] (ac3);
\draw[thick] (fb) to [out=90,in=-45](pr);
\draw[line width=3pt,white]  (cp) to [out=45,in=-90] (fb)(cac1)
to [out=45,in=-90] (t2);
\draw[thick,blue] (cac1) to [out=45,in=-90] (t2);
\draw[thick] (b1) to [out=90,in=-90] (cp) to [out=135,in=-135](ac3)(cp)
to [out=45,in=-90] (fb) (cac1)
to [out=135,in=-135] (pr) to [out=90,in=-90](t1);
\draw[thick,blue] (ac3) to [out=90,in=-90](cac1);
\end{tikzpicture}
\ =\ 
\begin{tikzpicture}[xscale=.5, yscale=1,baseline=40]
\coordinate (t1) at (0,3);
\coordinate (t2) at (1,3);
\coordinate (t3) at (2,3);
\coordinate (b1) at (0,0);
\coordinate (b2) at (1,0);
\coordinate (b3) at (2,0);
\coordinate (f1) at (2,1.5);
\coordinate (cac1) at (1,2);
\coordinate (ac1) at (1,2.5);
\coordinate (cac2) at (0,.5);
\coordinate (ac2) at (0,2.5);
\coordinate (cac3) at (1,.5);
\coordinate (ac3) at (1,1);
\coordinate (pr) at (0,2.5);
\coordinate (cp) at (0,.5);
\draw[thick] (cac3)--(pr);
\draw[line width=3pt,white] (cp)--(ac1);
\draw[thick] (cp) --(ac1);
\draw[thick,blue]  (cac3) to [out=45,in=-45] (ac1);
\draw[thick](b1) to [out=90,in=-90] (cp) to [out=135,in=-135] (pr)
to [out=90,in=-90] (t1);
\draw[thick,blue] (b2) to [out=90,in=-90](cac3)
(ac1) to [out=90,in=-90](t2);
\end{tikzpicture}
\end{multline}
where the last equality is the compatibility Equation~\eqref{greq:comp-blue}.
This completes the proof of Theorem~\ref{thm:1} for left Yetter--Drinfel'd f-objects.
The proof of the right Yetter--Drinfel'd f-objects is obtained by rotating the above
diagrams by 180 degrees, followed by replacing the blue color by the red color. 
\qed


\section{Braided tensor algebras and Nichols algebras}
\label{sec.tensor}

\subsection{Braided tensor algebras}
\label{sub.bta}

In this section we specialize the abstract language of Hopf f-objects and the
Yetter--Drinfel'd f-objects to the context of a braided category
$\mathcal{C}$ which, as a monoidal category, is a subcategory  of the category
$\mathbf{Vect}_\bF$ of vector spaces over a field $\bF$ with the monoidal
structure given by the tensor product $\otimes_\bF$. The objects of $\mathcal{C}$
will be called \emph{\color{blue} braided vector spaces}. In this case, a Hopf
f-object and a Yetter--Drinfel'd f-object will be respectively called
a \emph{\color{blue} Hopf f-algebra} and
\emph{\color{blue} Yetter--Drinfel'd f-module}. 

Recall that it follows from the definition
that a Hopf f-algebra is a pair $(H,\phi)$ of a braided Hopf algebra and an
automorphism $\phi$ of it. 
There is an elementary universal construction of such pairs $(H,\phi)$ that we now
discuss. 

Fix a braided vector space $V$ of finite dimension $n$. Then, the tensor algebra
$T(V)$ has a unique structure of a braided Hopf algebra determined by declaring all
elements of $V$ to be primitive. We define the rank of $T(V)$ to be the dimension
of $V$, and call the braided Hopf algebra $T(V)$ to be of diagonal type if the
braiding on $V$ is diagonal with respect a basis  $B$ of $V$.

In this case, $T(V)$ is $\Z_{\ge0}^n$-graded and admits a rich Abelian group of braided
Hopf algebra automorphisms. Namely, any map $t\colon B\to \bF_{\ne0}$ corresponds to
a braided Hopf algebra automorphism $\phi_t$ of $T(V)$ uniquely determined by
\begin{equation}
\phi_t b=t_b b,\quad \forall b\in B,
\end{equation}
where we denote by $t_b$ the image $t(b)$. We call such automorphisms
scaling automorphisms.

Summarizing, a finite dimensional vector space $V$ with a diagonal braiding with
respect to a basis $B$ of $V$, together with a map $t \colon B\to \bF_{\ne0}$
determines a pair $(T(V),\phi_t)$ of a braided Hopf algebra and an automorphism
of it. Using Theorems~\ref{thm:1}--\ref{thm:2}, we obtain multi-parameter
infinite-dimensional $R$-matrices over the vector space $T(V)$. Our interest is to
find rigid $R$-matrices which correspond to finite-dimensional Yetter--Drinfel'd
f-modules. We discuss this next.

\subsection{Nichols algebras}
\label{sub.nichols}

It turns out that the braided tensor algebras $T(V)$ defined above have a canonical
quotient called Nichols algebra which is a braided Hopf algebra. It can be finite or
infinite dimensional.

Recall that a \emph{\color{blue} Nichols algebra} over a braided vector space
$V$ is the quotient braided Hopf algebra
$\mathfrak{B}(V)=T(V)/J$ of the tensor algebra $T(V)$ over the maximal (braided)
Hopf algebra ideal $J$ intersecting trivially the part $\bF\oplus V\subset T(V)$.
In the case when the braiding is of diagonal type, the scaling automorphism $\phi_t$ of
$T(V)$ decends to an automorphism of the braided Hopf algebra
$\mathfrak{B}(V)=T(V)/J$ leading thereby to a Hopf f-algebra which we will call
Nichols f-algebra. Thus, finite dimensional Nichols f-algebras can be used as an
input to the construction of multiparameter knot invariants in~\eqref{steps}.

Examples of finite dimensional Nichols algebras are the nilpotent Borel parts of
Lustig's small quantum groups. A detailed description of finite dimensional Nichols
algebras can be found for example in~\cite{AS:finite, AS:pointed, A:on.fin.dim}.
Like quantum groups, 
Nichols algebras have PBW bases~\cite{Kharchenko} and the ones with diagonal braiding
have been classified by Heckenberger~\cite{Heckenberger, Heckenberger:classification}
building on the work of Kharchenko~\cite{Kharchenko} and
Andruskiewitsch--Schneider~\cite{AS:finite}. The list of diagonal Nichols algebras
of rank (that is, dimension of $V$) at most 3 is given in Tables 1 and 2
of~\cite{Heckenberger}, and from this, it follows that the majority of finite
rank Nichols algebras do not come from quantum groups. A presentation of Nichols
algebras of diagonal type in terms of generators and relations is given by
Angiono~\cite{Angiono}.

\subsection{Sub/quotient Yetter--Drinfel'd f-modules of Nichols f-algebras}
\label{sub.subquotient}

If a Nichols algebra is infinite dimensional, we cannot immediately proceed to the
construction of knot invariants. 

It turns out that a Nichols f-algebra  $\mathfrak{B}(V)$, as a left/right
Yetter--Drinfel'd f-module over itself, has a canonical quotient
$^L\mathfrak{B}(V)$ and a canonical subspace $\mathfrak{B}(V)^R$ which are
left and right Yetter--Drinfel'd f-modules over $\mathfrak{B}(V)$ respectively.
The construction of these f-modules
is as follows:
\be
\label{DYLdef}
^L\mathfrak{B}(V) = \mathfrak{B}(V)/ \mathfrak{B}(V)W_{\delta_L}, \qquad
W_{\delta_L}=\{ x \in \mathfrak{B}(V)\setminus \BF \,\, |\,\, \delta_L x=1\otimes x \} 
\ee
where elements of $W_{\delta_L}$ are nothing else but the \emph{\color{blue}coinvariant}
elements in degree $\ge 1$ with respect to the left coaction $\delta_L$,
see~\cite[Def.~8.2.1]{Radford}; and
\be
\label{DYRdef}
\mathfrak{B}(V)^R = U_{\lambda_R}
\ee
where $U_{\lambda_R}\subset \mathfrak{B}(V)$ is the smallest subspace of
$\mathfrak{B}(V)$ that satisfies
\be
\label{DYRdef2}
\Delta W_{\lambda_R} \subset U_{\lambda_R} \otimes \mathfrak{B}(V), \qquad
W_{\lambda_R}=\{ x \in \mathfrak{B}(V) \,\, |\,\,
\lambda_R(x \otimes y)= x \epsilon y, \,\,
\text{for all} \,\, y \in \mathfrak{B}(V) \}
\ee
where elements
of $W_{\lambda_R}$ are nothing else but \emph{\color{blue} invariant} elements with
respect to the right action $\lambda_R$, see~\cite[Def.~11.2.3]{Radford}. Note that,
by taking into account the fact that $\epsilon y=0$ for all $y\in V$, the invariance
condition takes the form $\lambda_R(x \otimes y)=0$ for all $y\in V$.

If any one of the $\mathfrak{B}(V)$ f-modules $^L\mathfrak{B}(V)$ or
$\mathfrak{B}(V)^R$ is finite dimensional, then, by Theorem~\ref{thm:2}, it can be used
in~\eqref{steps} for construction of polynomial knot invariants. 

In the next sections we illustate the quotients and the subspaces of a braided
tensor algebra $T(V)$ when the dimension of $V$ is 1 or 2.


\section{The rank $1$ tensor algebra}
\label{sec.rank1}

\subsection{Definition}
\label{sub.rank1def}

In this section we compute from first principles the $R$-matrices of
Theorem~\ref{thm:1} for the rank $1$ tensor algebra, with no reference to Lie theory.
As we will find out, the corresponding knot invariants are none other than the
colored Jones and the ADO polynomials. 

The rank $1$ tensor algebra $T(\bF)$ is identified with the polynomial algebra
$\bF[x]$ in one indeterminate. It is an infinite dimensional $\bF$-vector space with
basis $B=\{x^k \, | \, k \in \BZ_{\geq 0} \}$. 

The Hopf algebra structure and the braided structure of $T(\bF)$ are determined
by
\be
\label{hopf1}
\Delta x = x\otimes 1 + 1\otimes x, \qquad \tau(x\otimes x) = q \, x\otimes x \,.
\ee
The above equations, together with the axioms of a braided Hopf algebra, and
the choice of the basis, uniquely
determine the braided Hopf algebra structure. The formulas involve the
$q$-Pochhammer symbol $(x;q)_n$ and the $q$-binomial coefficients ${k\brack m}_q$
defined by
\begin{equation}
\label{qpoc}
(x;q)_n:=\prod_{i=0}^{n-1}(1-xq^i), \qquad
{k\brack m}_q:=\frac{(q;q)_k}{(q;q)_{k-m}(q;q)_m} \,.
\end{equation}
Explicitly, we have the following.

\begin{lemma}
\label{lem.H1}
The coproduct, the antipode and the scaling automorphism $\phi_t$ of $T(\bF)$ are
given by
\begin{align}
\label{Delta1}  
\Delta x^k & =\sum_{m=0}^k{k\brack m}_q x^{k-m}\otimes x^m \\
\label{S1}
S x^k &= (-1)^k q^{k(k-1)/2} x^k \\
\label{phi1}
\phi_t x^k &= t^k x^k 
\end{align}
respectively.
\end{lemma}

\begin{proof}
The primitivity of $x$ implies that
\begin{equation}
\Delta x=x_1+x_2,\quad x_1:=x\otimes 1,\quad x_2:=1\otimes x, 
\end{equation}
and the braiding implies that
\be
\quad x_2x_1=qx_1x_2 \,.
\ee
This, combined with the $q$-binomial formula, gives
\begin{equation}
  \Delta x^k=(x_1+x_2)^k=\sum_{m=0}^k{k\brack m}_q x_1^{k-m}x_2^m
  =\sum_{m=0}^k{k\brack m}_q x^{k-m}\otimes x^m \,.
\end{equation}
This proves~\eqref{Delta1}. To prove~\eqref{S1}, apply~\eqref{Hr5tau} for
$x^k \in T(\bF)$, use $\eta \, \epsilon \, x^k=\delta_{k,0}$ and compute
\begin{align*}
\nabla(\operatorname{id}_H \otimes S) \Delta x^k &=
\nabla(\operatorname{id}_H \otimes S) \Big(
\sum_{m=0}^k{k\brack m}_q x^{k-m}\otimes x^m \Bigr)
\\ & =
\nabla \Big(\sum_{m=0}^k{k\brack m}_q x^{k-m}\otimes S x^m \Bigr) 
 =
\sum_{m=0}^k{k\brack m}_q x^{k-m} S x^m  \,.
\end{align*}
This is a linear system of equations that uniquely determines $S x^k$ by
induction on $k$. Since
\be
\label{qbinom2}
\sum_{m=0}^k{k\brack m}_q (-1)^m q^{m(m-1)/2} = \delta_{k,0}
\ee
Equation~\eqref{S1} follows. Finally~\eqref{phi1} is clear since $\phi_t$ is
an automorphism and $\phi_t x = t x$. 
\end{proof}

\subsection{The left and right Yetter--Drinfel'd f-modules}
\label{sub.rank1DY}

In this section we compute the $R$-matrices of Theorem~\ref{thm:2} explicitly.
We first compute the doubly iterated coproduct~\eqref{DD2}, the
coaction~\eqref{eq:coaction} and the $R$-matrix~\eqref{eq:r-matrix}. The formulas
involve the $q$-multinomial coefficients defined by
\begin{equation}
{k\brack m,n}_q:=\frac{(q;q)_k}{(q;q)_{k-m-n}(q;q)_m(q;q)_n} \,.
\end{equation}

\begin{lemma}
\label{lem.deltaL1}
The doubly iterated coproduct $\Delta^{(2)}$, the left coaction $\delta_L$ and the
$R$-matrix~\eqref{eq:r-matrix} $^L\rho$ are given by
\begin{align}
\label{DeltaL1}
\Delta^{(2)} x^k &=
\sum_{m=0}^k\sum_{n=0}^{k-m}{k\brack m,n}_q  x^{k-m-n}\otimes x^{m}\otimes x^n  \\
\label{eq:coaction-r-1}
\delta_L x^k &=
\sum_{m=0}^k{k\brack m}_q(tq^{m};q)_{k-m}x^{k-m}\otimes x^{m} \\
\label{R1left}
\rho_L(x^k\otimes x^l) &=
\sum_{m=0}^k{k\brack m}_q(tq^{k-m};q)_{m}(tq^{k-m})^lx^{l+m}\otimes x^{k-m} \,.
\end{align}
\end{lemma}

\begin{proof}
We compute 
\be
\begin{aligned}
\Delta^{(2)} x^k &=\sum_{m=0}^k{k\brack m}_q x^{k-m}\otimes\Delta x^m
=\sum_{m=0}^k{k\brack m}_q x^{k-m}\otimes\sum_{n=0}^m{m\brack n}_q x^{m-n}\otimes x^n
\\
&=\sum_{0\le n\le m\le k}{k\brack m}_q {m\brack n}_q x^{k-m}\otimes x^{m-n}\otimes x^n
\\
&=\sum_{n=0}^k\sum_{m=0}^{k-n}{k\brack m+n}_q {m+n\brack n}_q x^{k-m-n}
\otimes x^{m}\otimes x^n
\\
&=\sum_{n=0}^k\sum_{m=0}^{k-n}{k\brack m,n}_q  x^{k-m-n}\otimes x^{m}\otimes x^n
=\sum_{m=0}^k\sum_{n=0}^{k-m}{k\brack m,n}_q  x^{k-m-n}\otimes x^{m}\otimes x^n
\end{aligned}
\ee
and, using Equations~\eqref{phi1} and~\eqref{S1},
\be
\begin{aligned}
\delta_L x^k &=\sum_{m=0}^k\sum_{n=0}^{k-m}{k\brack m,n}_q
(\nabla\otimes\operatorname{id}_H)(\operatorname{id}_H\otimes \tau)
(x^{k-m-n}\otimes x^{m}\otimes S \phi_t \, x^n)
\\
&=\sum_{m=0}^k\sum_{n=0}^{k-m}q^{mn}{k\brack m,n}_q  (\nabla\otimes\operatorname{id}_H)
(x^{k-m-n}\otimes S \phi_t x^n\otimes x^{m})\\
&=\sum_{m=0}^k\sum_{n=0}^{k-m}q^{mn}{k\brack m,n}_q x^{k-m-n}S\phi_t x^n\otimes x^{m}
\\
&=\sum_{m=0}^k\sum_{n=0}^{k-m}q^{mn}(-t)^nq^{n(n-1)/2}{k\brack m,n}_q
x^{k-m}\otimes x^{m}
\\
&=\sum_{m=0}^k{k\brack m}_q\sum_{n=0}^{k-m}{k-m\brack n}_q(-tq^m)^nq^{n(n-1)/2}
x^{k-m}\otimes x^{m}
\\
&=\sum_{m=0}^k{k\brack m}_q(tq^{m};q)_{k-m}x^{k-m}\otimes x^{m} 
\end{aligned}
\ee
where the last equality follows from the $q$-binomial theorem.
Thus, the $R$-matrix~\eqref{eq:r-matrix} for $Y=H$ is given by
\be
\begin{aligned}
\rho_L(x^k\otimes x^l)
&=(\nabla\otimes\operatorname{id}_H)(\operatorname{id}_H\otimes \tau_{H,H})
(\delta_L\otimes \phi_H)(x^k\otimes x^l) \\  
&=\sum_{m=0}^k{k\brack m}_q(tq^{m};q)_{k-m} t^l
(\nabla\otimes\operatorname{id}_H)(\operatorname{id}_H\otimes \tau_{H,H})
(x^{k-m}\otimes x^{m} \otimes x^l) \\ 
&=\sum_{m=0}^k{k\brack m}_q(tq^{m};q)_{k-m}(tq^m)^lx^{k+l-m} \otimes x^{m} \\
&=\sum_{m=0}^k{k\brack m}_q(tq^{k-m};q)_{m}(tq^{k-m})^lx^{l+m}\otimes x^{k-m} \,.
\end{aligned}
\ee
\end{proof}

We next compute the doubly iterated product~\eqref{DD2}, the right 
action~\eqref{eq:action} and the $R$-matrix~\eqref{eq:r-matrix-r}.

\begin{lemma}
\label{lem.deltaR1}
The doubly iterated product $\nabla^{(2)}$, the right action $\lambda_R$ and the
$R$-matrix~\eqref{eq:r-matrix-r} are given by
\begin{align}
\label{nablaR1}
\nabla^{(2)} (x^k\otimes x^l\otimes x^m)  &= x^{k+l+m} \\
\label{deltaR1}
\lambda_R(x^k\otimes x^l) &= (tq^k;q)_l x^{k+l} \\
\label{R1right}
\rho_R(x^k\otimes x^l) &=
\sum_{m=0}^l{l\brack m}_q(tq^{k})^{l-m}(tq^k;q)_{m} x^{l-m}\otimes  x^{k+m} \,.
\end{align}
\end{lemma}

\begin{proof}
Equation~\eqref{nablaR1} is clear. 
To calculate the right action $\lambda_R(x^k\otimes x^l)$, we start with the
case $l=1$:
\be
\label{eq:action-of-x}
\begin{aligned}
\lambda_R(x^k\otimes x)&=\nabla^{(2)}(S\phi_t \otimes\operatorname{id}_{H\otimes H} )
(\tau \otimes \operatorname{id}_G)(x^k\otimes x\otimes 1+x^k\otimes 1\otimes x)\\
&=\nabla^{(2)}(S\phi_t \otimes\operatorname{id}_{H\otimes H} )
(q^k x\otimes x^k\otimes 1+1\otimes x^k\otimes x)\\
&=\nabla^{(2)}(-tq^k x\otimes x^k\otimes 1+1\otimes x^k\otimes x)
=-tq^k x^{k+1}+x^{k+1}\\
&=(1-tq^k)x^{k+1} \,.
\end{aligned}
\ee
Now, we have
\be
\begin{aligned}
\lambda_R(x^k\otimes x^l)&=\lambda_R(\lambda_R(x^k\otimes x)\otimes x^{l-1})
=(1-tq^k)\lambda_R(x^{k+1}\otimes x^{l-1})\\
&=(1-tq^k)(1-tq^{k+1})\lambda_R(x^{k+2}\otimes x^{l-2})=\dots
=(tq^k;q)_{l}\lambda_R(x^{k+l}\otimes x^{l-l})\\
&=(tq^k;q)_l x^{k+l} \,.
\end{aligned}
\ee
Thus, the $R$-matrix~\eqref{eq:r-matrix-r} for $Y=H$ is given by
\be
\begin{aligned}
\rho_R(x^k\otimes x^l)&=\sum_{m=0}^l{l\brack m}_q(\phi_H\otimes \lambda_R)
(\tau_{H,H}\otimes \operatorname{id}_H)(x^k\otimes x^m\otimes x^{l-m})\\
&=\sum_{m=0}^l{l\brack m}_qq^{km}t^m(tq^k;q)_{l-m} x^m\otimes  x^{k+l-m}\\
&=\sum_{m=0}^l{l\brack m}_q(tq^{k})^{l-m}(tq^k;q)_{m} x^{l-m}\otimes  x^{k+m} \,.
\end{aligned}
\ee
\end{proof}

The $R$-matrices~\eqref{R1left} and~\eqref{R1right} depend on two
variables $t$ and $q$, and using the basis $B=\{x^k \, | \, k \in \BZ_{\geq 0} \}$,
their entries are in $\BZ[t^{\pm 1}, q^{\pm 1}]$ and satisfy the Yang--Baxter
equation on an infinite dimensional space $T(\bF)$. 

However, to define knot invariants as state-sums, we need to have rigid
$R$-matrices over finite dimensional vector spaces. In the remaining subsections
we give several solutions to this problem and identify the corresponding knot
invariants. 

\subsection{$q$ a root of unity: the ADO polynomials}
\label{sub.rank1ADO}
  
The Nichols f-algebra $\mathfrak{B}(\bF)$ is finite-dimensional if $q=\omega$ is
a root of unity of order $\operatorname{ord}(\omega)=N>1$ (for $N=1$, see Remark
\ref{rem.N=1} below).
In this case, it follows that ${N \brack k}_\omega=0$ for $0 < k < N$ and
Equation~\eqref{Delta1} implies that $x^N$ is primitive and thus generates a Hopf
ideal of $\bF[x]$ with finite-dimensional Nichols algebra $\bF[x]/(x^N)$.

In this case, the left $R$-matrix~\eqref{R1left} coincides with the $R$-matrix
of Akutsu--Deguchi--Ohtsuki \cite{ADO} and the knot invariant of Theorem~\ref{thm:3}
is the ADO polynomial times the identity matrix.

\begin{remark}
\label{rem.N=1} 
When $N=1$, the Nichols algebra $\mathfrak{B}(\bF)=\bF[x]$ is
infinite dimensional. In this exceptional case, and we will replace it with
the 1-dimensional algebra obtained by imposing the relation $x=0$ in the list of
finite dimensional Nichols algebras, despite the fact that $x\in V$.
\end{remark}

\subsection{$q$ generic: colored Jones polynomials}
\label{sub.rank1Jones}

When $q$ is not a root of unity, the Nichols f-algebra $\mathfrak{B}(\bF)=\bF[x]$ is
infinite-dimensional. However, it turns out that one can extract finite dimensional
left or right Yetter--Drinfel'd f-modules if
\be
\label{tqN}
t=q^{1-n}, \qquad n\in\BZ_{>0} \,.
\ee

Indeed, under this assumption for the scaling automorphism,
Equation~\eqref{eq:coaction-r-1}, implies that $x^n$ is a coinvariant element, 
$\delta_L x^n=1\otimes x^n$. This gives a quotient left Yetter--Drinfel'd f-module
$\bF[x]/(x^n)$ of dimension $n$. The corresponding $R$-matrix is the one of the
$n$-colored Jones polynomial. 

Moreover, under the same assumption~\eqref{tqN}, Equation~\eqref{eq:action-of-x} implies
that $x^{n-1}$ is an invariant element, whose coproduct generates the $n$-dimensional
space with basis $x^k$, $0\le k\le n-1$, and this gives an $n$-dimensional right
Yetter--Drinfel'd f-submodule of $\bF[x]$. The corresponding knot invariant is the
identity matrix times the $n$-th colored Jones polynomial. 

Summarising, in the rank $1$ case, the corresponding matrix-valued knot invariants 
are the identity times the ADO and the colored Jones polynomials. 
  

\section{A rank $2$ tensor algebra}
\label{sec.rank2}

In this section we discuss the case of Nichols f-algebras of rank 2 of diagonal type.
In order to keep the construction as simple as possible, we consider the Nichols
algebra $\mathfrak{B}(V,c)$ associated with two-dimensional vector space $V$ with
basis $B=\{x_1,x_2\}$ and diagonal braiding
\be
\label{eq:tau2}
\tau(x_i\otimes x_j)=q_{ij}x_j\otimes x_i, \qquad
c =
\begin{pmatrix}
 -1&q_{12}\\
 q_{21}&-1
\end{pmatrix} \,.
\ee
Denote by $H_c$ the corresponding Nichols algebra. 
In Heckenberger's list~\cite[Table 1]{Heckenberger:rank2} the isomorphism type of
$H_c$ is determined by the parameter $q:=q_{12}q_{21}$ of the
generalized Dynkin diagram (see also~\cite[Defn.~3.1]{Heckenberger:examples2}).

For generic values of $q$, $H_c$, is an infinite-dimensional quotient of the free
noncommutative algebra in $x_1$ and $x_2$ by the 2-sided ideal generated by
$x_1^2, x_2^2$, and a basis of $H_c$ is given by alternating words
in letters $x_1$ and $x_2$.

\subsection{$q$ a root of unity: two-variable knot polynomials over cyclotomic fields}
\label{sub.rank2omega}

When $q=\omega$ is a root of unity of order $N \geq 1$, we have 
$(x_2x_1)^N+(-q_{21}x_1x_2)^N = 0 \in H_c$, and the Nichols algebra $H_c$
is $4N$-dimensional with a basis given by all alternating words in $x_1$ and $x_2$
of length $\le 2N$, excluding $(x_2x_1)^N$.
$H_c$ is a Nichols f-algebra with scaling automorphism defined by
$\phi_t x_i = t_i x_i$ for $i=1,2$ where $t_1$ and $t_2$ are two independent
invertible elements. To emphasize the dependence of this f-algebra
on the braiding $c$ and the automorphism $\phi_t$, we will denote it by
$H_{c,t}$. Theorem~\ref{thm:2} constructs a left $R$-matrix
$R_{c,t}$ on $H_{c,t}$ with entries in $\BZ[t_1,t_2,q_{12}^{\pm 1},q_{21}^{\pm 1}]$
(keeping in mind that $q_{12} q_{21} = \omega$) which is invertible with determinant
$(t_1 t_2)^{N^2}$ for all of its rotated versions. 


Using this $R$-matrix together with Theorem~\ref{thm:3} we arrive at
the following knot invariants. 

\begin{definition}
\label{def.rank2N}  
For a root of unity $\omega$, we have the knot invariant
\be
\label{Womega}
K \mapsto J_{R_{c,t}}(K) \in \operatorname{End}(H_{c,t})
\ee
and denote by $\Lambda_{\omega,K}(t_1,t_2) \in \Z[\omega, t_1^{\pm 1},t_2^{\pm 1}]$
the $(1,1)$-entry of $J_{R_{c,t}}(K)$ with respect to the above basis of
$H_{c,t}$.
\end{definition}

A priori, the knot invariant depends on the Nichols f-algebra $H_{c,t}$, and hence
on $q_{12}$, $q_{21}$, $t_1$ and $t_2$. However, the polynomial invariant
depends only on $t_1$, $t_2$ and on the product
$\omega=q=q_{12} q_{21}$. This follows from the fact that the isomorphism class of the
Nichols algebra depends on $q$ alone (see~\cite[Defn.~3.1]{Heckenberger:examples2}).
What's more, the invariant $\Lambda_\omega$ satisfies the symmetry
\be
\label{Omsym}
\Lambda_{\omega,K}(t_1,t_2) = \Lambda_{\omega,K}(t_2,t_1) \,.
\ee
This follows from the existence of a unique isomorphism
\be
\label{Hsigma}
\s: H_{c,t} \to H_{\overline{c},\overline{t}}, \qquad \overline{c}=c^t,
\qquad \overline{(t_1,t_2)}=(t_2,t_1) 
\ee
that satisfies $\s(x_1)=x_2$ and $\s(x_2)=x_1$, and consequently, the $R_{c,t}$-matrix
satisfies \newline $(\s \otimes \s) R_{c,t} (\s^{-1} \otimes \s^{-1})
= R_{\overline{c},\overline{t}}$. 

\begin{conjecture}
\label{conj:1}
For every knot $K$, we have 
\begin{equation}
J_{R_{c,t}}(K) = \Lambda_{\omega,K}(t_1,t_2)\operatorname{id}_{H_{c,t}} .
\end{equation}
\end{conjecture}

\begin{remark}
The above knot invariant $\Lambda_\omega$ was defined using the left $R$-matrix of
the Nichols f-algebra. Of course, we could consider the invariant from the right
$R$-matrix, but the two appear to be equal.
\end{remark}

\subsection{The cases $q=\pm1$}

The first invariant from the above definition occurs when $\omega=1$, thus $N=1$,
where the corresponding Nichols f-algebra $H_{c,t}$ is 4-dimensional. Its left and
right $R$-matrices are easy to analyze, and the corresponding knot invariant satisfies
\be
\Lambda_{1,K}(t_1,t_2) = \Delta_K(t_1) \Delta_K(t_2)
\ee
where $\Delta_K(t) \in \BZ[t^{\pm1}]$ is the canonically normalized Alexander polynomial.

The next invariant occurs when $\omega=-1$, thus $N=2$ and the Nichols algebra
is 8-dimensional. We now discuss the properties of this invariant. 

To begin with, the Nichols f-algebra $H_{c,t}$ is 8-dimensional and isomorphic
to the nilpotent Borel subalgebra of the small quantum group $u_q(\mathfrak{sl}_3)$
with $q=\sqrt{-1}$. A basis for $H_{c,t}$ is
$$
\{1,x_1,x_2,x_1 x_2, x_2 x_1,x_1 x_2 x_1, x_2 x_1 x_2, x_1 x_2 x_1 x_2 \}
$$
and the corresponding $64 \times 64$ left $R$-matrix $R_{c,t}$ has entries in
$\BZ[t_1,t_2,q_{21}^{\pm 1}]$ and can be computed explicitly. It is a sparse matrix
with only $157$ out of $4096$ entries (about $3.8\%$) nonzero. Due to its size,
we do not present these entries here, but give a sample value
\be
\label{sampleR}
\begin{aligned}
R(x_1 x_2 x_1 x_2 \otimes x_2 x_1) &=
s^2 t^2 \, x_2 x_1 \otimes x_1 x_2 x_1 x_2
- q_{21}^{-1} s^2 t (1 + t) \, x_1 x_2 x_1 \otimes x_2 x_1 x_2
\\ & \hspace{-1cm}
+ q_{21}^2 (-1 + s) s t^2 \, x_2 x_1 x_2 \otimes x_1 x_2 x_1
+ q_{21}^{-1} (-1 + s) s t \, x_1 x_2 x_1 x_2 \otimes x_2 x_1 \,,
\end{aligned}
\ee
where for simplicity we abbreviate $t_1$ and $t_2$ by $s$ and $t$.
The determinant of this $R$-matrix and all of its rotated versions is $(st)^{64}$.


For all knots for which we computed the invariant, we confirmed that
Conjecture~\ref{conj:1} holds, and  the Laurent polynomial
$\Lambda_{-1,K}(s^2,t^2)$ coincides with Harper's polynomial
$\Delta_{\mathfrak{sl}_3}(s,t)$~\cite{Harper2020}. 

\noindent $\bullet$ {\bf Symmetry.} 
In addition to the symmetry~\eqref{Omsym} for $\Lambda_\omega$, it turns out 
that the polynomial $\Lambda_{-1,K}(t,s)$ is invariant under the involutions
\be
\label{12symmetries}
(s,t) \mapsto (t,s), \qquad (s,t) \mapsto (s,-1/(s t)), \qquad
(s,t) \mapsto (1/s,1/t)
\ee
which generate a group $G$ of order $12$. The existence of these additional
symmetries follow from the identification of the Nichols algebra with the
Borel part of the small quantum group, whose Weyl group is the symmetric group
of order 3, and together with $\s$, generates the group $G$ of order 12.
The invariant polynomial ring can be identified with 
\be
\label{Ginv}
\BZ[t^{\pm 1}_1,t^{\pm 1}_2]^G = \BZ[u,v], 
\ee
where
\begin{equation}
\label{uv}
u=\langle s \rangle + \langle t \rangle - \langle st \rangle -2,
\qquad v= \langle s^2 t \rangle + \langle s t^2 \rangle - \langle s/t \rangle - 2,
\qquad \langle x \rangle = x + x^{-1} \,.
\end{equation}
Thus, we can write the polynomial $\Lambda_{-1,K}$ in the form
\begin{equation}
\label{QtildeQ}
\Lambda_{-1,K}(s,t)=\tilde \Lambda_K(u,v),\quad \tilde \Lambda_K(u,v)\in \Z[u,v] \,.
\end{equation}

\noindent $\bullet$ {\bf Chirality and mutation.} 
Experimentally, it appears that the polynomial $\Lambda_{-1,K}$ does not distinguish
a knot from its mirror image. Regarding mutation, it distinguishes the pair
$(11n34,11n42)$ of mutant knots but not the pair $(11n73,11n74)$ of mutant knots. 

\noindent $\bullet$ {\bf Specialization.}
Experimentally, it appears that the specialization $u=0$ reproduces the
Alexander--Conway polynomial $\nabla_K(z)$ 
\begin{equation}
\tilde \Lambda_{-1,K}(0,z^2)=\nabla_K(z).
\end{equation}


\noindent $\bullet$ {\bf Comparison with other knot polynomials.}
Regarding the independence of $\Lambda_{-1}$ from other knot polynomials, 
we have the following observations.
\begin{itemize}
\item[(a)]
  The knots $\overline{7_4}$ and $9_2$ (where $\overline K$ is the mirror image of
  $K$) have equal Knot Floer Homology (a well-known fact; see
  Manolescu~\cite{Manolescu}) and confirmed by SnapPy, thus have Seifert genus
  $1$ and none is fibered~\cite{snappy}. On the other hand, the two knots
  have different $\Lambda_{-1}$-polynomials; see Table~\ref{tab:1}.
\item[(b)]
  The colored Jones  and the ADO polynomials do not distinguish mutant pairs
  of knots, since the corresponding tensor product of representations is
  multiplicity-free. This fact was pointed out to us by J. Murakami and T. Ohtsuki.
  On the other hand, the $\Lambda_{-1}$-polynomial sometimes detects mutation, and
  sometimes does not. In particular, it distinguishes the mutant pair of knots
  $11n42$ (Kinoshita--Terasaka knot) and $11n34$  (Conway knot) --- the fact
  that was pointed out by Harper in~\cite{Harper2020} for his polynomial
  $\Delta_{\mathfrak{sl}_3}(s,t)$.
  On the other hand,  it does not distinguish the mutant pair of knots
  $11n74$ (a fibered, Seifert genus $2$ knot) and $11n73$ (a non-fibered,
  Seifert genus $3$ knot). 
\item[(c)]
  The knots $8_8$ and $10_{129}$ have isomorphic Khovanov homology~\cite{B-N:Kh},
  yet different $\Lambda_{-1}$-polynomials. 
\end{itemize}


\noindent $\bullet$ {\bf Values.} Table~\ref{tab:1} gives the result of computer
calculation for all knots of up to 6 crossings, and few higher crossing knots
that appear in the above discussion. 

\begin{table}[htpb!]
\begin{center}
\begin{tabular}{c||c}
$K$ & $\tilde \Lambda_K(u,v)$\\
\hline\hline
$3_1$ & $1 + 4 u + u^2 + v$\\
\hline
$4_1$ & $1 - 6 u + u^2 - v$\\
\hline
$5_1$ &$ 1 + 12 u + 19 u^2 + 8 u^3 + u^4 + (3 + 7 u + 3 u^2) v + v^2$\\
\hline
$5_2$ & $1 + 10 u + 6 u^2 + 2 v$\\
\hline
$6_1$ & $1 - 10 u + 6 u^2 - 2 v$\\
\hline
$6_2$ & $1 - 8 u - 15 u^2 + 2 u^3 + u^4 + (-1 - 9 u + u^2) v - v^2$\\
\hline
$6_3$ & $1 + 2 u + 15 u^2 + 6 u^3 + u^4 + (1 + 9 u + u^2) v + v^2$\\
\hline
\multirow{2}{1em}{$7_1$} & $1 + 24 u + 86 u^2 + 104 u^3 + 53 u^4 + 12 u^5 + u^6$\\ 
&$+ (6 + 35 u + 60 u^2 + 33 u^3 + 5 u^4) v + (5 + 10 u + 6 u^2) v^2 + v^3$\\
\hline
$7_4$ & $(1 + 2 u) (1 + 18 u) + 4 v$\\
\hline
$8_8$ & $1 + 10 u + 36 u^2 + 28 u^3 + 6 u^4 + 2 (1 + 9 u + 3 u^2) v + 2 v^2$\\
\hline
\multirow{2}{1.25em}{$8_{17}$} & $1 - 14 u - 23 u^2 - 38 u^3 + 10 u^4 + 8 u^5 + u^6$\\
 & $ + (-1 - 17 u - 44 u^2 + 5 u^3 + 2 u^4) v+ (-2 - 13 u + u^2) v^2 - v^3$\\
 \hline
 $9_2$ & $1 + 20 u + 24 u^2 + 4 v$\\
 \hline
\multirow{3}{1.35em}{$10_2$} & $1 - 150 u^2 - 380 u^3 - 279 u^4 - 44 u^5 + 25 u^6 + 
 10 u^7 + u^8 $\\
 &$+ (2 - 55 u - 260 u^2 - 274 u^3 - 58 u^4 + 19 u^5 + 
5 u^6) v$\\
&$ + (-5 - 62 u - 91 u^2 - 24 u^3 + 5 u^4) v^2 + (-5 - 
15 u - 2 u^2) v^3 - v^4$\\
\hline
$10_{129}$ & $1 + 10 u + 32 u^2 + 36 u^3 + 6 u^4 + 2 (1 + 8 u + 2 u^2) v + 2 v^2$\\
\hline
$11n34$ & $1 + 12 u + 8 u^2 + 60 u^3 + 48 u^4 + 8 u^5 + 
 2 u (1 + 2 u) (-1 + 6 u) v + 2 u^2 v^2$\\
 \hline
 $11n42$ & $1 + 12 u + 8 u^2 - 12 u^3 - 2 u v$\\
 \hline
  $11n73$ & \multirow{2}{20em}{$1 + 20 u + 10 u^2 + 4 u^3 + u^4
 + 2 (1 + 4 u + u^2) v + v^2$}\\
  $11n74$&\\
  \hline
\end{tabular}
\end{center}
\caption{The polynomial $\tilde \Lambda_K(u,v)$ for some knots $K$.}
\label{tab:1}
\end{table}

\subsection{$q$ generic: two-variable polynomials}
\label{sub.rank2generic}

In this section we classify all finite-dimensional right Yetter--Drinfel'd f-modules
by classifying all invariant vectors. Recall that for generic $q$, the Nichols
f-algebra $H_{c,t}$ has a basis that consists of all alternating words in the letters
$x_1$ and $x_2$, where $x_1^2=x_2^2=0$. Thus, every basis element is one of the
following forms
\be
\label{Babcd}
(x_1 x_2)^a, \quad (x_2 x_1)^b,  \quad x_2 (x_1 x_2)^c,
\quad x_1 (x_2 x_1)^d
\ee
with  $a,b,c,d \in \BZ_{\geq0}$. Moreover, $H_{c,t}$ is
$\Z_{\ge0}^2$-graded, and thus also $\Z_{\ge0}$-graded where the $\BZ_{\ge0}$-degree is
the sum of the components of the $\BZ_{\ge0}^2$-degree.


It follows from~\eqref{Babcd} that the degree $2n-1$ part of $H_{c,t}$ is the
direct sum of two bi-degrees $(n,n-1)$ and $(n-1,n)$, each of them being
one-dimensional. This implies that there are no invariant vectors of degree $2n-1$.
Indeed, the only vector of bi-degree $(n-1,n)$ is
$$
x:=(x_2x_1)^{n-1}x_2=x_2(x_1x_2)^{n-1}
$$
which is $\lambda_R$-annihilated by $x_2$ but not by $x_1$:
\begin{equation}
\lambda_R(x\otimes x_1)=(x_2x_1)^{n}+(-q_{21})^{n}t_1(x_1x_2)^{n}
\end{equation}
which never vanishes since the vectors $(x_2x_1)^{n}$ and $(x_1x_2)^{n}$ are linearly
independent.

Thus, invariant vectors can only be of even degree $2n$ coming from
bi-degree $(n,n)$. The corresponding subspace is two-dimensional with the basis
vectors $(x_1x_2)^n$ and $(x_2x_1)^n$. Taking a vector of the form
\begin{equation}
\label{eq:gen-el}
v_{n,\alpha}=(x_1x_2)^n+\alpha (x_2x_1)^n, \quad \alpha\in\bF,
\end{equation}
we calculate  the $\lambda_R$-action on it of the generating elements:
\begin{equation}
\lambda_R(v_{n,\alpha}\otimes x_1) =(1-\alpha t_1(-q_{21})^n)(x_1x_2)^nx_1,
\quad
\lambda_R(v_{n,\alpha}\otimes x_2)=(\alpha -t_2(-q_{12})^n )(x_2x_1)^nx_2.
\end{equation}
Thus, $v_{n,\alpha}$ is an invariant vector if
\begin{equation}
\label{eq:fin2}
\alpha=t_2(-q_{12})^n
\end{equation}
and
\begin{equation}
\label{eq:fin}
t_1t_2q^n=1 
\end{equation}
where $q=q_{12} q_{21}$.


\begin{proposition}
\label{prop.rank2}  
Assume that $q$ is not a root of unity, and the parameters
$(q,t_1,t_2)$ satisfy~\eqref{eq:fin} for some $n \geq 1$, and
$(1-t_1)(1-t_2)\ne0$.
Then, the right Yetter--Drinfel'd f-module $Y_n$ generated by the element
$v_{n,\alpha}$ defined in \eqref{eq:gen-el} with $\alpha$ given by~\eqref{eq:fin2}
is $4n$-dimensional and it is the linear span of the vector $v_{n,\alpha}$
and all vectors of degree less or equal to $2n-1$.
\end{proposition}

\begin{proof}
The coproduct of $v_{n,\alpha}$ always contains the term $1\otimes v_{n,\alpha}$
so that $1\in Y_n$. The $\lambda_R$-action of the generating elements on $1$ gives
\begin{equation}
\lambda_R(1\otimes x_i)=(1-t_i)x_i,\qquad i=1,2 \,.
\end{equation}

By the assumption on $t_1$ and $t_2$, we conclude that vectors $x_1$ and $x_2$ are
both in $Y_n$. Assume by induction that both vectors of odd degree $2k-1$ are
contained in $Y_n$ where $1\le k<n$. Then, the $\lambda_R$-action on them of the
generating elements produces two vectors in degree $2k$
\begin{equation}
\label{eq.deg-2k-1}
\lambda_R((x_1x_2)^{k-1}x_1\otimes x_2)=(x_1x_2)^k+t_2(-q_{12})^k(x_2x_1)^k
\end{equation}
and
\begin{equation}
 \label{eq.deg-2k-2}
\lambda_R((x_2x_1)^{k-1}x_2\otimes x_1)=(x_2x_1)^k+t_1(-q_{21})^k(x_1x_2)^k
\end{equation}
which are linearly independent if $t_1t_2q^k\ne1$. Thus, the
$\lambda_R$-action of the generating elements on all vectors of degree $2k$ produces
all vectors of degree $2k+1$. We conclude that all vectors of degree $\le 2n-1$ are
in $Y_n$. Now, equations~\eqref{eq.deg-2k-1} and \eqref{eq.deg-2k-2} at $k=n$ imply
that both vectors are proportional to $v_{n,\alpha}$. 
\end{proof}

We can parametrize the variables $(t_1,t_2,q)$ that satisfy $t_1 t_2 q^n=1$ by
$t_1=1/(q^{n/2} t)$ and $t_2=t/q^{n/2}$. Theorem~\ref{thm:2} defines an $R$-matrix
$T_n$ on the right Yetter--Drinfel'd f-module $Y_n$, and combined with Theorem
~\ref{thm:3} we arrive at the following knot invariants.  

\begin{definition}
\label{def.rank2DY}  
For every integer $n \geq 1$ we have the knot invariant
\be
\label{WN}
K \mapsto J_{T_n}(K) \in \operatorname{End}(Y_n)
\ee
and denote by $V_{n,K}(t,q)\in\Z[q^{\pm 1/2},t^{\pm 1}]$ the $(1,1)$-entry of
$J_{T_n}(K)$. 
\end{definition}

\begin{conjecture}
\label{conj:Y_n}
For every knot $K$, we have 
\begin{equation}
J_{T_n}(K) =V_{n,K}(t,q)\operatorname{id}_{Y_n} .
\end{equation}
\end{conjecture}

Note that the symmetry $\sigma$ from~\eqref{Hsigma} corresponds under the
above parametrizations to the symmetry $t \leftrightarrow t^{-1}$, and as a
result, we have
\be
\label{tiQsym}
V_{n,K}(t,q)=V_{n,K}(t^{-1},q) \,.
\ee

\subsection{A rank-2 analogue of the Jones polynomial}
\label{sub.rank2.48}

In this section we discuss the knot invariants coming from the $4$ and $8$-dimensional
DY-modules $ Y_1$ and $ Y_2$.

Calculations for $n=1$ indicate that Conjecture~\ref{conj:Y_n} holds true and
that $V_{1,K}(t,q)$ coincides with the Links--Gould two-variable knot
polynomial coming from the quantum superalgebra $\mathrm{U}_q(\mathfrak{gl}(2|1))$
~\cite{Links-Gould}. 

When $n=2$, the knot polynomial
$V_{2,K}(t,q) \in \BZ[q^{\pm 1},t^{\pm 1}]$ is in a sense a rank 2 analogue of the
Jones polynomial. We discuss the properties of this invariant now. 

\noindent $\bullet$ {\bf Symmetry.}
The symmetry ~\eqref{tiQsym} implies that $V_2$ can be written in the form

\be
\label{QPtuq}
V_{2,K}(t,q) = \tilde V_{2,K}(u,q) \in \BZ[q^{\pm 1},u],
\quad u=t+t^{-1}-q-q^{-1} \,.
\ee

\noindent $\bullet$ {\bf Chirality and mutation.}
The 2-variable polynomial invariant $V_2$ detects chirality, for instance,
it distinguishes the left handed and right handed trefoils. Indeed,
$V_{2,\overline {3_2}}(t,q) = V_{3_2}(t,q^{-1})$ (a property that
conjecturally holds for all knots) and
\be
\tilde V_{2,\overline{3_1}}(u,q) =
1 + (q + 2 q^3 - q^4 + q^5 - q^6) u + (q^2 + q^4 - q^5) u^2 \,.
\ee
The $V_2$-polynomial distinguishes the Conway-KT pair $(11n34,11n42)$
of mutant knots, as well as the mutant pair $(11n73,11n74)$
where the $t$-degree is $(12,8)$ in both cases.

\noindent $\bullet$ {\bf Specialization.}
Experimentally, it appears that for all knots we have
\begin{itemize}
\item[]
  $\tilde V_{2,K}(0,q)=1$ (equivalently, $V_{2,K}(q,q) =1$).
\item[]
  $\tilde V_{2,K}(z^2,1)=\nabla_K(z)^2$ (equivalently,
  $V_{2,K}(t,1) =\Delta_K(t)^2$) where $\nabla_K(z)$ is the Alexander--Conway
  polynomial. 
\end{itemize}

\noindent $\bullet$ {\bf Relation with the genus of a knot.}
Experimentally, for all the knots in Table~\ref{tab:1}, as well as for a few
12 and 13 crossing knots, and for the 3-strand pretzel knots, we have:
\be
\label{eq.genus}
\deg_t V_{2,K} =4 g(K)
\ee
where the Seifert genus $ g(K)$ is the smallest genus of a spanning surface of a knot.
Here, by $t$-degree of a Laurent polynomial of $t$ we mean the difference between
the highest and the lowest power of $t$. Following the ideas in the works
~\cite{NVdV2022,Kohli-Tahar}, we expect that $4 g(K)$ is an upper bound for the
degrees in~\eqref{eq.genus}. Further computations are needed to see to which
extent the equality holds. 

\noindent $\bullet$ {\bf Comparison with other knot polynomials.}
Regarding the independence of $V_2$ from other knot polynomials, since
it detects mutation, it is not determined neither by the colored Jones polynomials,
nor by the sequence of the ADO polynomials. Moreover, 
\begin{itemize}
\item[]
  The knots $\overline{7_4}$ and $9_2$ have equal Knot Floer Homology but different
  $V_2$-polynomials.
\item[]
  The knots $8_8$ and $10_{129}$ have isomorphic Khovanov homology,
  yet different $V_2$-polynomial. 
\end{itemize}

\noindent $\bullet$ {\bf Values.}
The explicit values of the $V_2$ polynomials (even in the shorter form
$\tilde V_2$) are considerably more complicated than those of Table~\ref{tab:1},
but we have computed them for all the knots that appear in Table~\ref{tab:1} and for
several knots with 12 and 13 crossings. To give an idea of the complexity
involved, for the knots with at most 5 crossings we have:


\begin{tiny}
\begin{align*}
\tilde V_{2,\overline{3_1}} =&
1 + (q + 2 q^3 - q^4 + q^5 - q^6) u + (q^2 + q^4 - q^5) u^2 ,
\\
\tilde V_{2,4_1} =&
1 + (- q^{-3} + q^{-2} - 2 q^{-1} +2 - 2 q + q^2 - q^3) u + (q^{-2} - q^{-1} +1 -
     q + q^2) u^2,
\\
\tilde V_{2,\overline{5_1}} =&
1 + (2 q + 3 q^3 - q^4 + 3 q^5 - q^6 + 2 q^7 - q^8 + q^9 - 2 q^{10} + 
    q^{11} - q^{12}) u \\ & + (4 q^2 + 7 q^4 - 3 q^5 + 10 q^6 - 6 q^7 + 
    6 q^8 - 7 q^9 + 3 q^{10} - 3 q^{11}) u^2 \\ & + (3 q^3 + 6 q^5 - 3 q^6 + 
    6 q^7 - 6 q^8 + 3 q^9 - 3 q^{10}) u^3 + (q^4 + q^6 - q^7 + q^8 - 
    q^9) u^4,
\\
\tilde V_{2,\overline{5_2}} =&
1 + (q + 3 q^3 - q^4 + 3 q^5 - 2 q^6 + 2 q^7 - 2 q^8 + q^9 - 
    q^{10}) u + (3 q^2 - 2 q^3 + 6 q^4 - 3 q^5 + 3 q^6 - 3 q^7 + q^8 - 
    q^9) u^2 \,.
\\
\end{align*}
\end{tiny}
For the knots that appear in the comparison section, we have
\begin{tiny}
\begin{align*}
\tilde V_{2,\overline{7_4}} =&
1 + ( q+ 3 q^{3}+ 4 q^{5}- q^{6} + 5q^{7} - 3 q^{8}+ 4 q^{9}  - 4 q^{10} + 2 q^{11}
- 3 q^{12} + q^{13}    -q^{14} ) u\\
&  + (  9 q^{2} - 12 q^{3}+ 22 q^{4} - 12 q^{5} + 26 q^{6}- 17 q^{7}  + 15 q^{8}
- 14 q^{9} + 5 q^{10}- 6 q^{11}  + q^{12} -q^{13}) u^2,
\\
\tilde V_{2,9_2} =&
1 + ( q + 3q^{3}+ 4 q^{5}- q^{6}+ 3 q^{7} - 2 q^{8}+ 2 q^{9} - 2 q^{10}  + 2 q^{11}
- 2 q^{12}+ 2 q^{13} - 2 q^{14} + 2 q^{15} - 2 q^{16}+ q^{17}  -q^{18} ) u\\
&    + (7 q^{2}- 2 q^{3}+ 12 q^{4}  - 6 q^{5} + 10 q^{6} - 7 q^{7}+ 9 q^{8}- 7 q^{9}
+ 7 q^{10} - 7 q^{11}+ 5 q^{12} - 5 q^{13}+ 3 q^{14} - 3 q^{15}  + q^{16}   -q^{17}) u^2 
\end{align*}
\end{tiny}

\noindent
and
\begin{tiny}
\begin{align*}
\tilde V_{2,8_8} =&
1 + (- q^{-6} + q^{-5} + 2 q^{-3} - q^{-2} + 2 q^{-1} -2 + q + q^3 + q^4 + q^7 - 
2 q^8 + 2 q^9 - q^{10}) u \\ & + (- q^{-5} + q^{-4} - q^{-3} + 3 q^{-2}
- 4 q^{-1} +7 - 4 q + 12 q^2 - 10 q^3 + 11 q^4 - 9 q^5 + 7 q^6 - 5 q^7 + 
    2 q^8 - q^9) u^2 \\ & + (- 3 q^{-4} + 5 q^{-3} - 10 q^{-2} + 18 q^{-1} -17 + 25 q - 
    19 q^2 + 19 q^3 - 14 q^4 + 9 q^5 - 6 q^6 + 2 q^7 - 
    q^8) u^3 \\ & + (- 3 q^{-3} + 5 q^{-2} - 7 q^{-1} + 13 - 10 q + 12 q^2 - 9 q^3 + 
    7 q^4 - 5 q^5 + 2 q^6 - q^7) u^4,
    \\
\tilde V_{2,10_{129}} =&
1 + (- q^{-10} + q^{-9} - q^{-8} + q^{-7} + 2 q^{-4} + q^{-1} -2 + 3 q - 
2 q^2 + 3 q^3 - 2 q^4 + q^5) u \\ & + (q^{-9} - 3 q^{-8} + q^{-7}
+ q^{-6} - 4 q^{-5} + 8 q^{-4} - 9 q^{-3} + 11 q^{-2} - 8 q^{-1} + 11  - 5 q + 6 q^2 - 
    3 q^3 + 3 q^4 - 3 q^5 + 2 q^6 - q^7) u^2 \\ & + (2 q^{-8} - 4 q^{-7} + 
    2 q^{-6} - 10 q^{-4} + 20 q^{-3} - 30 q^{-2} + 42 q^{-1} -38 + 40 q - 26 q^2 + 
    18 q^3 - 10 q^4 + 4 q^5 - 2 q^6) u^3 \\ & + (- 3 q^{-5} + 5 q^{-4}
    - 8 q^{-3} + 12 q^{-2} - 12 q^{-1} + 16 - 10 q + 8 q^2 - 5 q^3 + 2 q^4 - q^5) u^4 \,,
\end{align*}
\end{tiny}

\noindent
as well as 
\begin{tiny}
\begin{align*}
\tilde V_{2,11n34} =&
1 + (- q^{-10} + 2 q^{-9} - 2 q^{-7} + 4 q^{-5} - 2 q^{-4} - 4 q^{-3}
+ 6 q^{-2} - 2 q^{-1} -4 + 6 q - 6 q^2 + 6 q^3 - 7 q^4 + 8 q^5 - 3 q^6 - 4 q^7
\\ & + 5 q^8 - 
q^9 - 2 q^{10} + 2 q^{12} - q^{13}) u 
+ (- 2 q^{-7} + 5 q^{-6} - 5 q^{-5} + 4 q^{-4} - 5 q^{-3} + 11 q^{-2} - 13 q^{-1}
+ 9 - 4 q + 2 q^2 - q^3 \\ & - 2 q^4 + 5 q^5 - 7 q^6 + 2 q^7 + 2 q^8 - q^9 - 2 q^{10}
+ 4 q^{11} - 2 q^{12}) u^2 
+ (2 q^{-8} - 5 q^{-7} + 3 q^{-6} + 4 q^{-5} - 15 q^{-4} + 26 q^{-3} \\ & - 33 q^{-2}
+ 41 q^{-1} -48 + 53 q - 53 q^2 + 49 q^3 - 44 q^4 + 36 q^5 - 27 q^6 + 13 q^7
- 5 q^9 + 5 q^{10} - 2 q^{11}) u^3 
+ (3 q^{-7} - 9 q^{-6} \\ & + 12 q^{-5} - 9 q^{-4} - 4 q^{-3} + 23 q^{-2}
- 44 q^{-1} + 69 - 85 q + 85 q^2 - 69 q^3 + 44 q^4 - 23 q^5 + 4 q^6
+ 9 q^7 - 12 q^8 + 9 q^9 \\ & - 3 q^{10}) u^4 
+ (3 q^{-6} - 9 q^{-5} + 12 q^{-4} - 11 q^{-3} + 17 q^{-1} -28 + 
38 q - 38 q^2 + 28 q^3 - 17 q^4 + 11 q^6 - 12 q^7 + 9 q^8 - 
3 q^9) u^5 \\ & 
+ (q^{-5} - 3 q^{-4} + 3 q^{-3} - 1 q^{-2} - 2 q^{-1} + 5 - 5 q + 
5 q^2 - 5 q^3 + 2 q^4 + q^5 - 3 q^6 + 3 q^7 - q^8) u^6,
\\
\tilde V_{2,11n42} =&
1 + (q^{-10} + 2 q^{-9} - 2 q^{-7} + 4 q^{-5} - 2 q^{-4} - 4 q^{-3}
+ 6 q^{-2} - 2 q^{-1} -4 + 6 q - 6 q^2 + 6 q^3 - 7 q^4 + 8 q^5 - 3 q^6
- 4 q^7 + 5 q^8 \\ & - q^9 - 2 q^{10} + 2 q^{12} - q^{13}) u
+ (- 2 q^{-7} + 4 q^{-6} - q^{-5} - 2 q^{-4} + 8 q^{-2} - 15 q^{-1}
+ 17 - 13 q + 11 q^2 - 9 q^3 + 8 q^5 - 12 q^6 \\ & + 8 q^7 - 2 q^8
- 2 q^{10} + 4 q^{11} - 2 q^{12}) u^2
+ (q^{-8} - 2 q^{-7} + q^{-6} - q^{-5} + 2 q^{-4} - 4 q^{-3} + 7 q^{-2}
- 8 q^{-1} + 10 - 10 q + 10 q^2 \\ & - 9 q^3 + 5 q^4 - 4 q^5 + 3 q^6 - 4 q^7
+ 5 q^8 - 3 q^9 + 2 q^{10} - q^{11}) u^3
+ (q^{-5} - 3 q^{-4} + 3 q^{-3} - q^{-2} - q^{-1} + 3 - 4 q + 4 q^2 \\ & - 3 q^3
+ q^4 + q^5 - 3 q^6 + 3 q^7 - q^8) u^4 \,,
\end{align*}
\end{tiny}

\noindent
and finally the values for the pair $(K11n73,K11n74)$ which are not distinguished
by the polynomial $\Lambda_{-1}$ of Section~\ref{sub.rank2omega}
\begin{tiny}
\begin{align*}
\tilde V_{2,11n73} =&
1 + (- q^{-13} + 2 q^{-12} - 2 q^{-10} + q^{-9} + 2 q^{-8} - 4 q^{-7} + 4 q^{-5} - 
5 q^{-4} + 6 q^{-3} - 6 q^{-2} + 8 q^{-1} -6
+ 3 q + 3 q^2 - 2 q^3 - q^4 \\ & + 4 q^5 - 2 q^6 + q^9 - q^{10}) u
+ (- 3 q^{-12} + 7 q^{-11} - 5 q^{-10} + 2 q^{-9} + q^{-8} - 2 q^{-7} - 2 q^{-6}
+ 2 q^{-5} - 4 q^{-4} + 5 q^{-3} - 5 q^{-2} \\ & + 5 q^{-1} + 2
- 5 q + 13 q^2 - 8 q^3 + 6 q^4 - 4 q^5 + 4 q^6 - 3 q^7 + q^8 - q^9) u^2
+ (- 5 q^{-11} + 14 q^{-10} - 18 q^{-9} + 22 q^{-8} - 21 q^{-7} \\ & + 14 q^{-6}
- 12 q^{-5} + 6 q^{-4} - 8 q^{-2} + 14 q^{-1} - 15
+ 22 q - 18 q^2 + 20 q^3 - 14 q^4 + 6 q^5 - 2 q^6 - 2 q^7 + q^8) u^3
+ (- 6 q^{-10} \\ & + 18 q^{-9} - 26 q^{-8} + 36 q^{-7} - 38 q^{-6} + 26 q^{-5}
- 15 q^{-4} - 10 q^{-3} + 32 q^{-2} - 40 q^{-1} + 46 
- 33 q + 20 q^2 - 9 q^3 - 2 q^4 + 5 q^5 \\ & - 6 q^6 + 3 q^7) u^4
+ (- 4 q^{-9} + 12 q^{-8} - 16 q^{-7} + 20 q^{-6} - 17 q^{-5} + 2 q^{-4}
+ 6 q^{-3} - 17 q^{-2} + 26 q^{-1} - 20
+ 16 q - 6 q^2 \\ & - 5 q^3 + 6 q^4 - 6 q^5 + 3 q^6) u^5
+ (- q^{-8} + 3 q^{-7} - 3 q^{-6} + 2 q^{-5} - q^{-4} - 2 q^{-3} + 3 q^{-2}
- 3 q^{-1} + 4
- 2 q + q^3 - 2 q^4 + q^5) u^6,
\\
\tilde V_{2,11n74} =&
1 + (- q^{-13} + 2 q^{-12} - 2 q^{-10} + q^{-9} + 2 q^{-8} - 4 q^{-7}
+ 4 q^{-5} - 5 q^{-4} + 6 q^{-3} - 6 q^{-2} + 8 q^{-1} - 6
+ 3 q + 3 q^2 - 2 q^3 - q^4 \\ & + 4 q^5 - 2 q^6 + q^9 - q^{10}) u
+ (- 2 q^{-12} + 4 q^{-11} - 2 q^{-10} + q^{-8} + 4 q^{-7} - 10 q^{-6} + 8 q^{-5}
- 10 q^{-4} + 5 q^{-3} + q^{-2} - q^{-1} \\ & + 10
- 11 q + 13 q^2 - 6 q^3 + 3 q^4 - q^5 + 3 q^6 - 3 q^7 + q^8 - q^9) u^2
+ (- q^{-11} + 2 q^{-10} - 2 q^{-9} + 2 q^{-8} + 2 q^{-7} - 4 q^{-6} \\ & + 4 q^{-5}
- 7 q^{-4} + q^{-3} + 6 
- q + 4 q^2 - q^4 + q^5 - 2 q^6) u^3
+ (q^{-6} - 2 q^{-5} + 2 q^{-4} - 3 q^{-3} + 2 q^{-2} - q^{-1} + 2 
+ q - q^2 \\ & + q^3 - q^4) u^4 \,.
\end{align*}
\end{tiny}


\section{Summary of rank 1 and rank 2 polynomials}
\label{sec.summary}

Here, we summarise the polynomial invariants that we have considered
in the form of two tables, one at roots of unity and another at generic values of $q$.
Within each table, the underlying  $R$-matrix is a linear automorphism of $V\otimes V$
for a finite dimensional vector space $V$ given by a Nichols algebra when $q$ is a
root of unity and by a Yetter--Drinfel'd f-module when $q$ is generic. We denote
the normalized Alexander polynomial by $\Delta(t)$, the ADO polynomial at a root
of unity $\omega$ by $\operatorname{ADO}_\omega(t)$, and the $n$-colored Jones
polynomial (with $n=2$ equal to the Jones polynomial) by $J_n(q)$. All these
polynomials are normalized to take the value $1$ at the unknot. 

\begin{table}[htpb!]
\begin{center}
\begin{tabular}{|c||c|c|} \hline
 & rank 1& rank 2 \\ \hline\hline
$\dim(V)$ & $N$ & $4N$ \\ \hline
Invariant & $\operatorname{ADO}_\omega(t)$ &$\Lambda_\omega(t_1,t_2)$ \\ \hline
$\omega=1$  & 1&  $\Lambda_{1} \stackrel{?}{=}\Delta(t_1)\Delta(t_2)$\\ \hline
  $\omega=-1$ & $\Delta(t)$ &
$\Lambda_{-1} \stackrel{?}{=} \text{Harper polynomial}$
  \\ \hline  
\end{tabular}
\end{center}
\caption{Knot polynomials at primitive $N$-th roots of unity $\omega$.}
\label{tab:rank1}
\end{table}

\begin{table}[htpb!]
\begin{center}
\begin{tabular}{|c||c|c|} \hline
 & rank 1 & rank 2 \\ \hline\hline
$\dim(V)$ & $n$ & $4n$ \\ \hline
Invariant & $J_n(q)$ & $V_n(t,q)$ \\ \hline
$n=1$  & 1 & $ V_1 \stackrel{?}{=}\text{Links--Gould polynomial}$ \\ \hline
$n=2$ & $J(q)$ & $V_2(t,q)$ \\ \hline  
\end{tabular}
\end{center}
\caption{Knot polynomials at generic values of $q$.}
\label{tab:rank2}
\end{table}

The invariants at roots of unity should be dual to those
at generic values of $q$ in the sense of equivalences
\begin{subequations}
\begin{align}
\label{duality1}
\{ \mathrm{ADO}_\omega(t) \mid \omega\text{ is a root of unity} \}
& \longleftrightarrow \{ J_n(q) \mid n \geq 1\},
\\
\label{duality2}
\{ \Lambda_\omega(t_1,t_2) \mid \omega\text{ is a root of unity} \}
& \longleftrightarrow \{V_n(t,q) \mid n \geq 1\}.
\end{align}
\end{subequations}
Mathematically, the meaning of these equivalences is that for each knot, the set
of its polynomial invariants on the left determines the set of its invariants
on the right and vice-versa.
The equivalence~\eqref{duality1} is known, and it follows from Habiro's expansion
of the colored Jones polynomial~\cite{Habiro:WRT,Willets}.
Explicitly, the left hand side determines the right hand side of the above
equivalences by
\begin{subequations}
\begin{align}
\label{eduality1}    
\mathrm{ADO}_\omega(\omega^{1-n}) &= J_n(\omega) ,
\\
\label{eduality2}
\Lambda_\omega(t \omega^{-n/2},t^{-1}\omega^{-n/2} ) &= V_n(t,\omega) 
\end{align}
\end{subequations}
valid for all positive integers $n>0$ and all complex roots of unity $\omega$.
A proof of these equalities will be given in a subsequent publication. 
Equation~\eqref{eduality1} (with our normalization of the ADO and colored Jones
polynomials) appears in Murakami--Nagatomo~\cite{Murakami-Nagatomo}. 

As a sanity check, Equation~\eqref{eduality2} for $\omega=-1$ is a consistency relation
with the values of the $\Lambda_{-1}$ polynomials given in Table~\ref{tab:1} with
the polynomials $V_2$ given in this section.

Finally, we mention that the interpretation of ~\eqref{duality1} and ~\eqref{duality2}
in mathematical physics is a kind of duality or correspondence arising in
supersymmetric quantum field theories; see for example~\cite{Gukov-Park}.


\subsection*{Acknowledgements}

The paper started during the Oberwolfach meeting 2331 in August 2023. The authors
wish to thank the organizers and Mathematisches Forschungsinstitut Oberwolfach for
its hospitality, and Nicol\'as Andruskiewitsch, Norbert A'Campo, Anton Alekseev,
David Cimasoni, J\'er\^ome Dubois, Mikhail Khovanov, Kazuo Habiro,
Michael Heusener, Val\'erian Montessuit, Jun Murakami, Tomotada Ohtsuki, Martin Palmer,
Cristina Palmer-Anghel, Sunghyuk Park, Muze Ren and Louis-Hadrien Robert for
useful conversations. 

The work of R.K. is partially supported by the Swiss National Science Foundation
research program NCCR The Mathematics of Physics (SwissMAP), the ERC Synergy grant
Recursive and Exact New Quantum Theory (ReNew Quantum), by the Swiss National Science
Foundation, the subsidies no 200020 192081 and no 200020 200400, and the Russian
Science Foundation, the subsidy no 21-41-00018.


\bibliographystyle{hamsalpha}
\bibliography{biblio}
\end{document}